%% file: Main.tex
\documentclass[11pt]{article}

\usepackage{vmargin}
\setmarginsrb{28mm}{25mm}{28mm}{25mm}{0pt}{0mm}{0pt}{0mm}
\usepackage{amssymb}
\usepackage[fleqn]{amsmath}
\usepackage{nccmath}
\usepackage{amsthm}
\usepackage{pgfplots} 
\usepackage{graphicx}
\usepackage[utf8x]{inputenc}
\usepackage{tikz}
\usepackage{bbm}
\usepackage{subcaption}
\usepackage[boxruled]{algorithm2e}
\usepackage{mathtools}
\usepackage{lipsum}
\usepackage[title,titletoc]{appendix}
\usepackage{booktabs}
\usepackage{here}
\usepackage[]{hyperref}
\usepackage{enumerate} 
\usepackage[nottoc]{tocbibind}
\usepackage{booktabs}
\usepackage{enumitem}

\usepackage{indentfirst}
\usepackage{emptypage}

\usetikzlibrary{backgrounds,arrows,positioning}

\usepackage[english]{babel}
\usepackage{lmodern}

\usepackage{quiver}
\usepackage{yhmath}

\usepackage[framemethod=TikZ]{mdframed}

\usepackage{lscape}

\newcommand{\C}{\operatorname{\mathbb{C}}}
\newcommand{\Z}{\operatorname{\mathbb{Z}}}
\newcommand{\N}{\operatorname{\mathbb{N}}}

\newcommand{\R}{\operatorname{\mathbb{R}}}

\newcommand{\soc}{\operatorname{\mathrm{soc}}}
\newcommand{\rad}{\operatorname{\mathrm{rad}}}

\newcommand{\im}{\operatorname{\mathrm{im}}}
\newcommand{\id}{\operatorname{\mathrm{id}}}
\newcommand{\het}{\operatorname{\mathrm{ht}}}

\newcommand{\SL}{\operatorname{\mathrm{SL}}}
\newcommand{\GL}{\operatorname{\mathrm{GL}}}
\newcommand{\JSF}{\operatorname{\mathrm{JSF}}}
\newcommand{\AGL}{\operatorname{\mathrm{AGL}}}
\newcommand{\ch}{\operatorname{\mathrm{ch}}}

\newcommand{\Sp}{\operatorname{\mathrm{Sp}}}
\newcommand{\Spin}{\operatorname{\mathrm{Spin}}}

\newcommand{\g}{\operatorname{\mathfrak{g}}}
\newcommand{\h}{\operatorname{\mathfrak{h}}}

\newcommand{\sln}{\operatorname{\mathfrak{sl}}}

\newcommand{\Ug}{\operatorname{\mathcal{U}(\mathfrak{g})}}

\newcommand{\XR}{X_{\mathbb{R}}}

\newtheorem{thm}{Theorem}[section]
\newtheorem{lemma}[thm]{Lemma}
\newtheorem{defn}[thm]{Definition}
\newtheorem{prop}[thm]{Proposition}

\newtheorem{cor}[thm]{Corollary}
\newtheorem{question}{Question}

\theoremstyle{definition}

\newtheorem{rem}[thm]{Remark}
\newtheorem{notation}[thm]{Notation}
\newtheorem{argument}{Argument}

\newcommand\sbullet{%
	\raisebox{-1ex}{\scalebox{2.7}{$\cdot$}}%
}

\date{\today}

\begin{document}

\pagenumbering{roman}
\setcounter{page}{1}
\input{Titlepage}

\cleardoublepage
\tableofcontents

\cleardoublepage

\renewcommand{\labelitemi}{$\bullet$}

\section*{List of notations}
\addcontentsline{toc}{section}{List of notations}
\vspace{1em}
\begin{tabular}{lcl}
	$\Phi$ && root system\\
	$\Pi$&& base of $\Phi$\\
	$\alpha^\vee$&& coroot of $\alpha$\\
	$X$&& weight lattice of $\Phi$\\
	$X^+$&& set of dominant weights\\
	$\leq$&& usual partial order on $X$\\
	$\omega_i$&& fundamental dominant weight corresponding to $\alpha_i$\\
	$\alpha_h$&& highest short root in $\Phi$\\
	$\rho$&& half sum of all positive roots\\
	$h$&& Coxeter number of $\Phi$\\
	$s_\alpha$&& reflection corresponding to $\alpha$\\
	$W$&& Weyl group of $\Phi$\\
	$w_0$&& longest element in $W$\\
	$s_{\alpha,mp}$&& affine reflection\\
	$W_p$&& affine Weyl group associated to $\Phi$ and $p$\\
	$w\sbullet x$&& $=w(x+\rho)-\rho$\\
	$D$&& $=X^+-\rho$, fundamental domain for the dot action of $W$ on $X$\\
	$\widehat{C}$&& upper closure of an alcove $C$\\
	$\overline{C}$&& closure of an alcove $C$\\
	$C_1$&& fundamental alcove\\
	$\uparrow$&& linkage order on $X$\\
	$\g$&& simple Lie algebra associated to $\Phi$ (over $\C$)\\
	$\Ug$&& universal enveloping algebra of the Lie algebra $\g$\\
	$G$&& simply connected Chevalley group with root system $\Phi$\\
	$X_\alpha$&& root subgroup corresponding to $\alpha$\\
	$T$&& maximal torus of $G$\\
	$B$&& Borel subgroup of $G$\\
	$G_{\C}$&& simply connected Chevalley group with root system $\Phi$ over $\C$\\
	$\soc M$&& socle of the module $M$\\
	$\rad M$&& radical of the module $M$\\
	$m_M(\lambda)$&& multiplicity of the weight $\lambda$ in $M$\\
	$S^aM$&& $a$-th symmetric power of $M$\\
	$L(\lambda)$&& simple $G$-module of highest weight $\lambda$\\
	$[M:L(\lambda)]$&& multiplicity of the composition factor $L(\lambda)$ in $M$\\
	$M^*$&& dual module of $M$\\
	$M^\tau$&& contravariant dual of $M$\\
	$\Delta(\lambda)$&& Weyl module of highest weight $\lambda$\\
	$\nabla(\lambda)$&& costandard module of highest weight $\lambda$\\
	$T(\lambda)$&& indecomposable tilting module of highest weight $\lambda$\\
	$\ch M$&& character of the module $M$\\
	$\chi(\lambda)$&& Weyl character associated to $\lambda$\\
	$\nu_p(m)$&& $p$-adic valuation of $m$\\
	$L_{\C}(\lambda)$&& simple $G_{\C}$-module of highest weight $\lambda$
\end{tabular}

\cleardoublepage

\section*{Introduction}
\addcontentsline{toc}{section}{Introduction}

Tensor products are of great interest in representation theory. Determining their structure has been the subject of much research. When the category of representations we are working with is semisimple, the central question is to find the decomposition of a tensor product of simple modules into a direct sum of simple modules. In particular, this is the case for representations of simple Lie algebras over the field of complex numbers, and for representations of simple algebraic groups over the same field. In the case of $\sln_2(\C)$, the Clebsch-Gordan formula gives an answer to this question.

Another question that has been the subject of research is to determine whether certain simple modules appear several times in the decomposition of the tensor product, and thus classify tensor products without multiplicity. In the case of simple Lie algebras and simple algebraic groups, this question was resolved in 2003 by Stembridge (\cite{Stembridge}). In particular, this classification shows that if the tensor product of two simple modules is multiplicity-free, then the highest weight of one of the two modules is a multiple of a fundamental weight. This fact is no longer true in positive characteristic.

When we move on to a field of positive characteristic, the category of representations of a simple algebraic group is no longer semisimple. Other questions, often complicated, may then arise, such as finding indecomposable direct summands and classifying completely reducible tensor products. The classification of multiplicity-free modules is also of interest. Recently, Gruber showed that a multiplicity-free tensor product of simple modules is necessarily completely reducible, which gives another motivation for this classification.

In this project, we will therefore focus on the classification of multiplicity-free tensor products of simple modules over an algebraically closed field of characteristic $p>0$. Our main question is:
\begin{question}
	Given a simply connected simple algebraic group $G$, for which pairs of simple modules $L(\lambda)$ and $L(\mu)$ is the tensor product $L(\lambda)\otimes L(\mu)$ multiplicity-free ?
\end{question}

We will provide the complete classification in the case of $\SL_2$ and $\SL_3$, and show a number of important results in the case of $\Sp_4$. In addition, we will show that, under certain assumptions, being completely reducible implies being multiplicity-free. Using the classification of completely reducible tensor products of simple modules for $\SL_n$ over a field of characteristic $2$, established by Gruber (\cite{article_jonathan}), we will answer our question in the case of $\SL_n$ for $p=2$.

The first part of this project recalls important notions of representation theory that will be used later. In the second part, we recall some results related to tensor products and show that we can restrict our attention to simple modules with $p$-restricted highest weight in order to answer our question. In the third part, we show some connections between multiplicity-freeness over $\C$ and in positive characteristic. In parts 4 to 7, we proceed to the classification of multiplicity-free tensor products in the cases of $\SL_2$, $\SL_3$, $\Sp_4$ (partial classification only), and $\SL_n$ for $p=2$ respectively. This work could be continued on the one hand by completing the classification for $\Sp_4$, and on the other hand by generalising these results to other simply connected simple algebraic groups.

It should also be noted that we used Magma (\cite{magma}) in order to compute the composition factors of certain tensor products and to have concrete examples, which enabled us to have a better understanding of the structure of these tensor products.

We assume that the reader is familiar with the representation theory of semisimple Lie algebras, as well as with the basics of algebraic group theory. We will therefore not repeat the relative notions in the preliminaries, but refer the reader to \cite{Humphreys_Lie} for the representation theory of semisimple Lie algebras and to \cite{Testerman_Gunter} for the theory of algebraic groups.\\

\emph{Acknowledgements:} First of all, I would like to thank Professor Donna Testerman for supervising this project and giving me the opportunity to discover this beautiful field of mathematics, as well as for her useful comments, her precious time and her support throughout the semester.  I would also like to thank Aluna Rizzoli for his explanation of the basics of Magma, and Jonathan Gruber for his help and for providing me with the alcove pictures. Finally, I would like to thank Maite Carli for her careful review of this project, and Amandine Favre for her constant support and her review of this project.

\cleardoublepage

\pagenumbering{arabic}
\setcounter{page}{1}

\section{Preliminaries}

In this section, we introduce all the notions needed to solve our problem and introduce some arguments which will be used several times in the next parts of this project.


\subsection{Weights and alcoves}
We start by recalling some results about root systems and weights. Then, we will define the notion of an alcove, which has a very important role in the structure of some modules.\\

Let $X_{\R}$ be a Euclidean space of dimension $n$ with scalar product $(\; ,\;):X_{\R}\times X_{\R}\to \R$ and let $\Phi\subseteq X_{\R}$ be an irreducible root system. We fix $\Pi=\{\alpha_1,\ldots,\alpha_n\}$ a base of $\Phi$ and denote by $\Phi^+$ the set of positive roots with respect to $\Pi$. Moreover, we fix $p$ a prime.\\

For $\alpha\in \Phi$, we define its \emph{coroot} $\alpha^\vee:=\frac{2\alpha}{(\alpha,\alpha)}\in \XR$.
	We set 
	$$\XR^+:=\{x\in \XR|\; (x,\alpha^\vee)\geq 0\quad \forall \alpha\in \Phi^+\},$$
and define the \emph{weight lattice} of $\Phi$ to be the set
	$$X:=\{\lambda\in X_{\R}|\; (\lambda,\alpha^\vee)\in \Z\quad \forall \alpha\in \Phi\}.$$
Moreover we define the set of \emph{dominant weights} to be the set
	$$X^+:=X\cap \XR^+=\{\lambda\in X|\; (\lambda,\alpha^\vee)\geq 0\quad \forall \alpha\in \Phi^+\}.$$
There is a partial order on $X$ given by $\lambda\leq \mu$ if $\lambda-\mu$ is a $\N$-linear combination of simple roots, i.e. $\lambda-\mu=\sum\limits_{\alpha\in \Pi}n_\alpha\alpha$ with $n_\alpha\in \N \; \forall\alpha\in \Pi$.\\

Since $\Pi$ is a basis of $X_{\R}$, it follows that $\{\alpha^\vee|\; \alpha\in \Pi\}$ is a basis of $X_{\R}$. Thus it admits a dual basis with respect to $(\; ,\;)$ and there exists a set $\{\omega_\alpha|\; \alpha\in \Pi\}$ whose elements satisfy
$$(\omega_\alpha,\alpha^\vee)=1\qquad \text{and} \qquad (\omega_\alpha,\beta^\vee)=0\quad\forall \beta \in \Pi\setminus\{\alpha\}.$$
We observe that the $\omega_\alpha\in X^+$ for all $\alpha\in \Pi$. We call the weights $\omega_\alpha$ with $\alpha\in \Pi$ the \emph{fundamental dominant weights}. One can easily check that every weight is a $\Z$-linear combination of the fundamental dominant weights.
To simplify the notation, we set $\omega_i:=\omega_{\alpha_i}$ for $i\in\{1,\ldots,n\}$. We will use the numeration of simple roots given in \cite[11.4]{Humphreys_Lie}, so the numeration of the fundamental dominant weights will correspond to this labelling of Dynkin diagrams.

\begin{defn}
	A dominant weight $\lambda\in X^+$ is called \emph{$p$-restricted} if $(\lambda,\alpha^\vee)<p$ for all $\alpha\in \Pi$.
\end{defn}

We denote the highest short root in the root system $\Phi$ (with respect to $\Pi$) by $\alpha_h$ and the half sum of all positive roots by $\rho:=\frac{1}{2}\sum\limits_{\alpha\in \Phi^+}\alpha.$ This element satisfies $\rho=\sum\limits_{i=1}^n\omega_i$ \linebreak(\cite[13.3]{Humphreys_Lie}).
The \emph{Coxeter number} of the root system $\Phi$ is $h:=(\rho,\alpha_h^\vee)+1$.

\begin{rem}\label{highest_short_root}
	The coroot $\alpha_h^\vee$ is the highest root in the dual root system $\Phi^\vee$ and for any dominant weight $\lambda\in X^+$, we have $(\lambda,\alpha^\vee)\leq (\lambda,\alpha_h^\vee)$ for all $\alpha\in \Phi^+$. More generally, for every $x\in \XR^+$ and $\alpha\in \Phi^+$, we have $(x,\alpha^\vee)\leq (x,\alpha_h^\vee)$. Recall also that $(\alpha,\alpha_h^\vee)\geq 0$ for all $\alpha\in \Phi^+$.
\end{rem}

\begin{defn}
	A set $Y\subseteq X$ is called \emph{saturated} if for all $\lambda\in Y, \; \alpha\in \Phi$ and $i$ between $0$ and $(\lambda,\alpha^\vee)$, we have $\lambda -i\alpha\in Y$. 
\end{defn}

	For $\alpha\in \Phi$ we define the reflection $s_\alpha\in \GL(X_{\R})$ by 
	$$s_\alpha(v):=v-\frac{2(v,\alpha)}{(\alpha,\alpha)}\alpha\qquad \text{for } v\in X_{\R}.$$
	We also define the \emph{Weyl group} of $\Phi$ by
	$$W:=\langle s_\alpha|\; \alpha\in \Phi\rangle\subseteq\GL(X_{\R}).$$
	The group $W$ is generated by $\{s_\alpha|\; \alpha\in \Pi\}$. More precisely, $(W,\{s_\alpha|\; \alpha\in \Pi\})$ is a Coxeter system (\cite[1.5]{Humphreys_Coxeter}). It contains a unique longest element which we denote by $w_0\in W$ (\cite[1.8]{Humphreys_Coxeter}).


	 For $\alpha\in \Phi$ and $m\in \Z$, we define the affine reflection $s_{\alpha,mp}\in \AGL(X_{\R})$ by
	$$s_{\alpha,mp}(v):=s_\alpha(v)+mp\alpha\qquad \text{for } v\in X_{\R}.$$
	The \emph{affine Weyl group} associated to $\Phi$ and $p$, denoted by $W_p$, is the group 
	$$W_p:=\langle s_{\alpha,mp}|\; \alpha\in \Phi,\; m\in \Z\rangle\subseteq\AGL(X_{\R}).$$
The \emph{dot action} of $W_p$ on $X_{\R}$ is the group action given by 
$$w\sbullet x:=w(x+\rho)-\rho \qquad \text{for }\; w\in W_p \:\text{ and }\; x\in \XR.$$
For $\alpha\in \Phi$ and $m\in \Z$, we define the \emph{reflection hyperplane} of $s_{\alpha,mp}$ (for the dot action) to be the set 
$$H_{\alpha,m}:=\{x\in X_{\R}|\; (x+\rho,\alpha^\vee)=mp\}.$$
Since $W$ is a subgroup of $W_p$, we can restrict the dot action to $W$. We set
$$D:=\{\lambda\in X|\; \lambda+\rho\in X^+\},$$ which is a fundamental domain for the dot action of $W$ on $X$. (This follows from the facts that $\XR^+$ is a fundamental domain for the action of $W$ on $\XR$ (\cite[1.12]{Humphreys_Coxeter}) and that $X$ is preserved by $W$.) 

\begin{defn}
	Let $n=(n_{\alpha})_{\alpha\in \Phi^+}\in \Z^{|\Phi^+|}$. We define
	$$C_n:=\{x\in X_{\R}|\; (n_\alpha-1)p<(x+\rho,\alpha^\vee)<n_{\alpha}p\quad \text{for all }\alpha\in \Phi^+\}.$$
	We say that $C_n$ is an \emph{alcove} if it is a non-empty set.
	For $C_n$ an alcove, its \emph{upper closure} is the set
	$$\widehat{C_n}:=\{x\in X_{\R}|\; (n_\alpha-1)p<(x+\rho,\alpha^\vee)\leq n_{\alpha}p\quad \text{for all }\alpha\in \Phi^+\},$$
	and its \emph{closure} is the set
	$$\overline{C_n}:=\{x\in X_{\R}|\; (n_\alpha-1)p\leq(x+\rho,\alpha^\vee)\leq n_{\alpha}p\quad \text{for all }\alpha\in \Phi^+\}.$$
\end{defn}
	Alternatively, we can define an alcove to be a connected component of $X_{\R}\setminus\bigcup\limits_{\substack{\alpha\in \Phi\\m\in \Z}}H_{\alpha,m}$.

\begin{defn}
	The \emph{fundamental alcove} is the alcove 
	$$C_1:=\{x\in X_{\R}|\; 0<(x+\rho,\alpha^\vee)<p\quad \text{for all }\alpha\in \Phi^+\}.$$
\end{defn}

\begin{rem}\label{C1_highest_short_root}
	By Remark \ref{highest_short_root}, we have
	$$C_1=\{x\in X_{\R}|\;(x+\rho,\alpha_h^\vee)<p \text{  and  } 0<(x+\rho,\alpha^\vee)\quad \text{for all }\alpha\in \Pi\}.$$
\end{rem}

\begin{lemma}\label{C1}
	Let $\lambda \in X^+\cap \widehat{C_1}$ and $\mu\in X^+$ be such that $\mu\leq \lambda$. We have $\mu \in \widehat{C_1}$.
\end{lemma}

\begin{proof}
	Let $a_1,\ldots a_n\in \N$ be such that $\mu=\lambda-\sum\limits_{i=1}^nc_i\alpha_i$. By assumption, $\mu\in X^+$, thus $(\mu+\nolinebreak\rho, \alpha)\geq 0$ for all $\alpha\in \Pi$. Thus, using Remark \ref{C1_highest_short_root}, we only need to show that $(\mu+\nolinebreak\rho,\alpha_h^\vee)\leq\nolinebreak p$. Since $(\lambda+\rho,\alpha_h^\vee)\leq p$ and $(\alpha_i,\alpha_h^\vee)\geq 0$ for all $\alpha_i\in \Pi$, we have
	$$(\mu+\rho,\alpha_h^\vee)=(\lambda-\sum_{i=1}^nc_i\alpha_i+\rho,\alpha_h^\vee)=(\lambda+\rho,\alpha_h^\vee)-\sum_{i=1}^nc_i(\alpha_i,\alpha_h^\vee)\leq(\lambda+\rho,\alpha_h^\vee)\leq p,$$
	so $\mu\in \widehat{C_1}$.
\end{proof}

\begin{defn}
	An alcove $C$ is \emph{$p$-restricted} if there exists a $p$-restricted dominant weight $\lambda\in X^+$ such that $\lambda\in \widehat{C}$.
\end{defn}

\begin{thm}[{\cite[4.5 and 4.8]{Humphreys_Coxeter}}]\label{W_simply_transitively}
	The affine Weyl group $W_p$ acts simply transitively on the set of alcoves. Moreover, $\overline{C_1}$ is a fundamental domain for the dot action of $W_p$ on $\XR$.
\end{thm}

\begin{defn}
	Let $\lambda,\mu\in X$. The weight $\lambda$ is \emph{linked} to $\mu$ if $\lambda=\mu$ or if there exist affine reflections $s_{\beta_1,m_1p},\dots,s_{\beta_t,m_tp}\in W_p$ such that 
	$$\lambda\leq s_{\beta_1,m_1p}\sbullet \lambda\leq \ldots \leq s_{\beta_t,m_tp}\cdots s_{\beta_1,m_1p}\sbullet \lambda =\mu.$$
	In this case, we write $\lambda \uparrow \mu$.
\end{defn}

\begin{rem}
	The relation $\uparrow$ is a partial order on $X$.
\end{rem}

\subsection{Chevalley groups and algebraic groups}\label{Chevalley}

In this section, we recall some definitions about linear algebraic groups, following \cite{Testerman_Gunter}, and we construct Chevalley groups following \cite{Steinberg}.

\subsubsection{Linear algebraic groups}
	Let $G$ be a linear algebraic group. A \emph{Borel subgroup} $B\leq G$ is a closed, connected, solvable subgroup of $G$ which is maximal with respect to all these properties. The \emph{radical} $R(G)$ of $G$ is the maximal closed connected solvable normal subgroup of $G$. The group $G$ is \emph{semisimple} if $R(G)=1$. A non-trivial semisimple algebraic group $G$ is \emph{simple} if it has no non-trivial proper closed connected normal subgroups. A representation $\rho :G\to \GL(V)$ is \emph{rational} if $\rho$ is a morphism of algebraic groups.

\subsubsection{Chevalley groups}
We fix a numeration of the roots $\Phi=\{\alpha_1,\ldots,\alpha_m\}$ such that $\het \alpha_i\leq \het \alpha_j$ for all $i\leq j$ (recall that $\Pi=\{\alpha_1,\ldots,\alpha_n\}$). Let $\g$ be the simple Lie algebra associated to $\Phi$ (over $\C$) with Cartan subalgebra $\h$. We denote by $\Ug$ its universal enveloping algebra and by $\kappa:\g\times\g\to \C$ the \emph{Killing form} on $\g$. We fix $\{e_\alpha,\; h_\beta|\; \alpha\in \Phi,\; \beta\in \Pi\}$ a Chevalley basis of $\g$ which satisfies the following properties:
\begin{enumerate}[label = (\arabic*)] 
	\item $h_\alpha$ is the coroot of $\alpha$, i.e. $h_\alpha=\frac{2t_\alpha}{\kappa(t_\alpha,t_\alpha)}$ where $t_\alpha \in \h$ is the unique element such that $\kappa(t_\alpha,h)=\alpha(h)$ for all $h\in \h$,
	\item $e_\alpha\in \g_\alpha$ for all $\alpha\in \Phi$,
	\item $[h_\alpha h_\beta]=0$ for all $\alpha,\beta\in \Pi$,
	\item $[h_\alpha e_\beta]=(\beta,\alpha^\vee) e_\alpha$,
	\item $[e_\alpha e_{-\alpha}]=h_\alpha$ for all $\alpha\in \Pi$,
	\item $[e_\alpha e_\beta]=0$ if $\alpha+\beta\notin \Phi$ and $\beta \neq -\alpha$,
	\item if $\beta-r\alpha,\ldots,\beta +q\alpha$ is the $\alpha$-string through $\beta$, then $[e_\alpha e_\beta]=N_{\alpha,\beta}e_{\alpha+\beta}$ if $\alpha+\beta \in \Phi$ with $N_{\alpha,\beta}=-N_{-\alpha,-\beta}=\pm (r+1)$.
\end{enumerate}
The existence of this basis is proven in \cite[25.2]{Humphreys_Lie}, \cite[4.2.1]{Carter} and \cite[Chapter 1]{Steinberg}.

For $\alpha\in \Phi^+$, we set $f_\alpha:=e_{-\alpha}$. Moreover, for any sequences of non-negative integers $A=(a_1,\ldots,a_m),B=(b_1,\ldots,b_m)\in \N^m,\; C=(c_1,\ldots,c_n)\in \N^n$, we define
\begin{ceqn}
\begin{align*}
	&E^A:=\frac{e_{\alpha_1}^{a_1}}{a_1!}\cdots \frac{e_{\alpha_m}^{a_m}}{a_m!}\in \Ug,\\
	&F^B:=\frac{f_{\alpha_1}^{b_1}}{b_1!}\cdots \frac{f_{\alpha_m}^{b_m}}{b_m!}\in \Ug,\\
	&H^C:=\binom{h_{\alpha_1}}{c_1}\cdots\binom{h_{\alpha_n}}{c_n}\in \Ug,
\end{align*}
\end{ceqn}
where $$\binom{h_{\alpha_i}}{c_i}:=\frac{h_{\alpha_i}(h_{\alpha_i}-1)\cdots(h_{\alpha_i}-c_i+1)}{c_i!}.$$

Using the PBW Theorem (see \cite[17.3]{Humphreys_Lie} or \cite[Chapter 2]{Steinberg}), the set $\{F^BH^CE^A\}$ is a basis of $\Ug$. 
We define $\Ug_{\Z}$ to be the subring of $\Ug$ generated by $\{\frac{e_\alpha^a}{a!}|\; \alpha\in \Phi,\; a\in \N\}$, $\Ug_{\Z}^\pm$ to be the subring of $\Ug$ generated by $\{\frac{e_\alpha^a}{a!}|\; \alpha\in \Phi^\pm,\; a\in \N\}$, and $\Ug_{\Z}^\circ$ the subring of $\Ug$ generated by $\{\binom{h_{\alpha}}{c}|\; \alpha\in \Phi,\; a\in \N\}$. Then $\{F^BH^CE^A\}$ is a $\Z$-basis of $\Ug_{\Z}$ \linebreak(\cite[Chapter 2]{Steinberg}).

Let $V$ be an irreducible finite-dimensional $\Ug$-module with highest weight $\lambda$. There exists $v^+\in V$ a maximal vector for $\Ug^\circ\Ug^+$ of weight $\lambda$. We define $M:=\Ug_{\Z}^-v^+$, which is a lattice in $V$. Then $M$ is stable under $\Ug_{\Z}$. For an arbitrary field $k$, we define $\Ug_k:=\Ug\otimes_{\Z}k$ and 
$V_k:=M\otimes_{\Z}k$, which has thus the structure of a $\Ug_k$-module.

For $\alpha\in \Phi$ and $t\in k$, we define
$$x_\alpha(t):=\exp(te_\alpha)=\sum_{i=0}^\infty t^i\frac{e_{\alpha}^i}{i!}.$$
Since $e_\alpha$ acts nilpotently on $V_k$, the map $x_\alpha(t)$ is well-defined and is an automorphism of $V_k$ (\cite[Chapter 3]{Steinberg}). 
We call the group 
$$G=G(V,k)=\langle x_{\alpha}(t)|\; \alpha\in \Phi,\; t\in k\rangle\subseteq\GL(V_k)$$
the \emph{Chevalley group} associated to $V$ and $k$. The type of $G$ is the type of the root system $\Phi$.\\

We fix $G=G(V,k)$ a Chevalley group. For $\alpha\in \Phi$, we define the \emph{root subgroup} corresponding to $\alpha$ to be $$X_\alpha:=\{x_{\alpha}(t)|\; t\in k\}\leq G.$$ 
	We also define 
	$$U:=\langle X_\alpha|\; \alpha\in \Phi^+\rangle\leq G\quad  \text{ and  } \quad 
	U^-:=\langle X_\alpha|\; \alpha\in \Phi^-\rangle\leq G.$$
	Furthermore, for $t\in k^*$, we define
	$$w_\alpha(t):=x_{\alpha}(t)x_{-\alpha}(-t^{-1})x_{\alpha}(t)\in G\quad \text{ and }\quad  h_{\alpha}(t):=w_{\alpha}(t)w_{\alpha}(1)^{-1}\in G$$ and set $$T:=\langle h_{\alpha}(t)|\; \alpha\in \Phi,\; t\in k^*\rangle\leq G.$$
	Finally, we set $$B:=\langle U,T\rangle\leq G.$$

From now on, we assume that $k$ is algebraically closed. Then $G$ is a semisimple algebraic group (over $k$) with maximal torus $T$ and Borel subgroup $B$ (\cite[Theorem 6]{Steinberg}). We call $\Phi$ the \emph{root system associated to $G$}. 

Let $X_G$ be the lattice of all weights appearing in rational representations of $G$. Then $X_G\subseteq X$ (\cite[Section 9.2]{Testerman_Gunter}), and we say that $G$ is \emph{simply connected} if $X_G=X$. For each type of root system, there exists a unique simply connected Chevalley group of this type (up to isomorphism). For a root system of type $A_n$, we have $G=\SL_{n+1}$, for a root system of type $B_n$, we have $G=\Spin_{2n+1}$, and for a root system of type $C_n$, we have $G=\Sp_{2n}$ (\cite[Chapter 3]{Steinberg}).

If $G=G(V,k)$ is simply connected and $G(V',k)$ is another Chevalley group of the same type, there exists a surjective homomorphism $G\to G(V',k)$ (\cite[Corollary 5]{Steinberg}). In particular, $V'$ has the structure of a $G$-module.

For the rest of this paper, we fix $k$ an algebraically closed field of characteristic $p>0$, $\Phi$ a root system with base $\Pi$, weight lattice $X$ and set of dominant weights $X^+$, $G=G(V,k)$ the simply connected Chevalley group with root system $\Phi$, and $B,T\leq G$ as in the last section. Moreover, we fix $W$ the Weyl group associated to $\Phi$ and we set $G_{\C}=G(V,\C)$.


\subsection{Modules}
In this section, we define several notions related to modules for $G$. In particular, we define Weyl modules and tilting modules. All the modules that we consider are finite-dimensional and correspond to rational representations. Moreover, by module, we always mean $G$-module.

\subsubsection{First definitions and irreducible modules}

We start by recalling some basic definitions and the classification of finite-dimensional simple $G$-modules.\\

Let $M$ be a $G$-module.  Its \emph{socle}, denoted by $\soc M,$ is the sum of all its simple submodules and its \emph{radical}, denoted by $\rad M$, is the intersection of all its maximal submodules. The socle $\soc M$ is the largest completely reducible submodule of $M$. The radical $\rad M$ is the smallest submodule of $M$ such that $M/\rad M$ is completely reducible (\cite[I 2.14]{Jantzen}). A vector $v\in M$ is a \emph{maximal vector} with respect to $B$ if $Bv\subseteq kv$. For $\lambda\in X$, we denote by $m_M(\lambda):=\dim M_\lambda$ the multiplicity of the weight $\lambda$ in $M$, where $M_\lambda$ is the weight space associated to the weight $\lambda$ in $M$. The weight $\mu$ is called the \emph{highest weight} of $M$ if every $\nu\in X$ with $m_M(\nu)>0$ satisfy $\nu\leq \mu$.
For $a\in \N$, we define the \emph{$a$-th symmetric power} of $M$ to be 
	$$S^aM:=M^{\otimes a}/\langle P-\sigma(P)\rangle$$
	where $P$ is a pure tensor and $\sigma(v_1\otimes\ldots\otimes v_a)=v_{\sigma(1)}\otimes\ldots\otimes v_{\sigma(a)}$ for $\sigma\in S_a$. By multilinear algebra, if $(v_1,\ldots, v_m)$ is an ordered basis of $M$, then $\{v_{i_1}\otimes\ldots\otimes v_{i_a}\}_{1\leq i_1\leq \ldots\leq i_a\leq m}$ is a basis of $S^aV$.

The irreducible $G$-modules are classified by their highest weight.

\begin{thm}[{\cite[31.3]{Humphreys_LAG}}]
	Let $\lambda \in X^+$ be a dominant weight. Up to isomorphism, there exists a unique irreducible module with highest weight $\lambda$ which we denote by $L(\lambda)$. This module satisfies $m_{L(\lambda)}(\lambda)=1$. Moreover, every irreducible module is of the form $L(\nu)$ for some $\nu\in X^+$.
\end{thm}

\begin{prop}[{\cite[2.4]{Pierce}}]\label{submodule_comp_red}
	Let $M$ be a completely reducible module, and $N<M$ be a submodule. Then $N$ and $M/N$ are completely reducible.
\end{prop}

	Let $M$ be a $G$-module. A \emph{composition series} for $M$ is a sequence of submodules
	$$0=M_0\subseteq M_1\subseteq \ldots \subseteq M_n=M$$
	such that the quotients $M_i/M_{i-1}$ are simple for all $i\in \{1,\ldots,n\}$.
	For $\nu\in X^+$, we write 
	$$[M : L(\nu)]:=|\{i\in\{1,\ldots,n\}|\; M_i/M_{i-1}\cong L(\nu)\}|.$$
	Due to Jordan-Hölder Theorem, the value $[M : L(\nu)]$ does not depend on the choice of the composition series (see for example \cite[Theorem 3.11]{Erdmann}).
An irreducible module $L(\nu)$ is called a \emph{composition factor} of $M$ if it appears in a composition series for $M$, i.e. if $[M:L(\nu)]>0$.
If $M$ is a $G$-module with composition series $0=M_0\subseteq M_1\subseteq \ldots \subseteq M_n=M$, we write this composition series $[L(\nu_n),\ldots, L(\nu_1)]$ where $L(\nu_i)\cong M_i/M_{i-1}$.\\

We are now ready to define the central notion of this project.
\begin{defn}
	A module $M$ is \emph{multiplicity-free} if all composition factors appear with multiplicity $1$, i.e. $[M:L(\nu)]\leq 1$ for all $\nu\in X^+$.
	If a module $M$ is not multiplicity-free, we say that $M$ \emph{has multiplicity}.
\end{defn}

\begin{lemma}[{\cite[1.30]{Testerman_AMS}}]\label{-1}
	Let $\lambda=\sum\limits_{i=1}^na_i\omega_i\in X^+$ with $0\leq a_i<p$ for all $i\in\{1,\ldots, n\}$. Then for $i=1,\ldots,n$ and $0\leq r\leq a_i$, we have 
	$$m_{L(\lambda)}(\lambda-r\alpha_i)=1.$$
\end{lemma}

\subsubsection{Duality}

Now we define two notions of duality in the category of $G$-modules.

\begin{defn}
	Let $M$ be a $G$-module. Its \emph{dual} $M^*$ is the usual dual vector space of $M$ with $G$-action given by 
	$$(gf)(m)=f(g^{-1}m)\qquad \text{for }f\in M^*,\; g\in G,\; m\in M.$$
\end{defn}

\begin{prop}[{\cite[II 1.16]{Jantzen}}]\label{antiautomorphism}
	There exists $\tau$ an antiautomorphism of $G$ which satisfies 
	$$\tau^2=\id_G,\quad \tau|_T=\id_T\quad \text{and}\quad  \tau(X_\alpha)=X_{-\alpha}\quad \text{for all }\alpha \in \Phi.$$
	Moreover, if $G=\SL_n(k)$ and $T:=\{\text{diagonal matrices}\}$, we can take $\tau$ to be the matrix transposition, i.e. $\tau(g)=g^t$ for $g\in G$.
\end{prop}

\begin{defn}
	Let $M$ be a $G$-module. Its \emph{contravariant dual} $M^\tau$ is the dual vector space $M^*$ with action defined by
	$$(gf)(m)=f(\tau(g)m)\qquad \text{for }f\in M^*,\; g\in G,\; m\in M,$$
	where $\tau$ is the antiautomorphism from Proposition \ref{antiautomorphism}. The module $M$ is called \emph{contravariantly self-dual} if $M^\tau\cong M$.
\end{defn}

\begin{rem}\label{cont_tens_prod}
	One can easily check that $M^\tau\otimes N^\tau\cong(M\otimes N)^\tau$ for any $G$-modules $M,N$.
\end{rem}

\begin{rem}[{\cite[II 2.12]{Jantzen}}]\label{irred_cont_self}
	Irreducible modules are contravariantly self-dual.
\end{rem}

\begin{lemma}[{\cite[V 4.2]{Jonathan_these}}]\label{coselfdual_reducible}
	Let $M$ be a contravariantly self-dual module. If $M$ is multiplicity-free, then $M$ is completely reducible.
\end{lemma}


\subsubsection{Weyl modules}

In this section, we will define the so called Weyl modules. In characteristic $0$, those modules and the irreducible modules coincide. In positive characteristic, the Weyl modules are no longer irreducible, but are still useful to understand the irreducible modules.

Let $\lambda \in X^+$, $V'$ the $\Ug$-module of highest weight $\lambda$ and $V'_k$ be as defined in section \ref{Chevalley}. The group $G$ acts naturally on $V'_k$. We call this $G$-module the \emph{Weyl module} of highest weight $\lambda$\linebreak and denote it by $\Delta(\lambda)$. The Weyl module $\Delta(\lambda)$ is generated by a maximal vector for $B$ of weight $\lambda$ and satisfies the following universal property (\cite[II 2.13]{Jantzen}):
\begin{lemma}\label{Weyl_univ_prop}
	Let $V$ be a $G$-module generated by a maximal vector for $B$ of weight $\lambda\in X^+$. There exists a surjective morphism $\Delta(\lambda)\to V$.
\end{lemma}

\begin{prop}[{\cite[21.3]{Humphreys_Lie}}]\label{saturated}
	Let $\lambda\in X^+$. The set $\{\nu\in X|\; m_{\Delta(\lambda)}(\nu)>0\}$ is saturated with highest weight $\lambda$.
\end{prop}

To compute weight multiplicities in Weyl modules, we can use Freudenthal's formula, whose proof can be found in \cite[22.3]{Humphreys_Lie}.

\begin{thm}[Freudenthal's formula]\label{Freudenthal}
	Let $\lambda \in X^+$ and $\mu\in X$. Then
	$$((\lambda+\rho,\lambda+\rho)-(\mu+\rho,\mu+\rho))m_{\Delta(\lambda)}(\mu)=2\sum_{\alpha\in \Phi^+}\sum_{i=1}^\infty m_{\Delta(\lambda)}(\mu+i\alpha)(\mu+i\alpha,\alpha).$$
\end{thm}

It can be combined with the following proposition.

\begin{prop}[{\cite[Proposition A]{Cavallin}}]\label{Cavallin}
	Let $\lambda=\sum\limits_{i=1}^na_i\omega_i\in X^+$ and $\mu=\lambda-\sum\limits_{i=1}^nc_i\alpha_i$ with $c_i\in \N$ for all $i$. Suppose the existence of a non-empty subset $J\subseteq\{1,\ldots, n\}$ such that $c_j\leq a_j$ for all $j\in J$. Let $\lambda'=\lambda-\sum\limits_{j\in J}(a_j-c_j)\omega_j$ and $\mu'=\lambda'-\sum\limits_{i=1}^nc_i\alpha_i$. Then
	$$m_{\Delta(\lambda)}(\mu)=m_{\Delta(\lambda')}(\mu').$$
\end{prop}

For $\lambda\in X^+$, we define the \emph{costandard} module of highest weight $\lambda$ by
	$$\nabla(\lambda):=\Delta(-w_0(\lambda))^*.$$
The module $\nabla(\lambda)$ is also called the \emph{induced module} or \emph{dual Weyl module} in the literature, and satisfies $\nabla(\lambda)\cong \Delta(\lambda)^\tau$ (\cite[II 2.13]{Jantzen}).

\begin{prop}[{\cite[II 2.4 and 2.14]{Jantzen}}]\label{radical}
	Let $\lambda\in X^+$. We have $\Delta(\lambda)/\rad \Delta(\lambda)\cong L(\lambda)$ and $\soc \nabla(\lambda)\cong L(\lambda)$.
\end{prop}


\subsubsection{Filtrations and tilting modules}

Another class of useful modules are the so-called tilting modules, which we define in this subsection. They will be very useful to show that some tensor products have multiplicity. Before that, we define two special kinds of filtrations. We end this section by stating the classification of indecomposable tilting modules.

\begin{defn}
	Let $M$ be a $G$-module. A \emph{Weyl filtration} of $M$ is a sequence 
	$$0=M_0\subseteq M_1\subseteq \ldots\subseteq M_n=M$$
	of submodules such that $M_i/M_{i-1}$ is a Weyl module for all $i\in\{1,\ldots,n\}$, i.e. there exist dominant weights $\lambda_1,\ldots,\lambda_n\in X^+$ such that $M_i/M_{i-1}\cong \Delta(\lambda_i)$ for all $i\in\{1,\ldots,n\}$.
\end{defn}

\begin{thm}[{\cite[Theorem 1]{Mathieu}}]\label{Weyl_filtration}
	Let $\lambda,\;\mu\in X^+$ be dominant weights. The tensor product $\Delta(\lambda)\otimes \Delta(\mu)$ admits a Weyl filtration.
\end{thm}

\begin{defn}
	Let $M$ be a $G$-module. A \emph{good filtration} of $M$ is a sequence 
	$$0=M_0\subseteq M_1\subseteq \ldots\subseteq M_n=M$$
	of submodules such that there exist dominant weights $\lambda_1,\ldots,\lambda_n\in X^+$ with $M_i/M_{i-1}\cong \nabla(\lambda_i)$ for all $i\in\{1,\ldots,n\}$.
\end{defn}

\begin{rem}
	A $G$-module $M$ admits a Weyl filtration if and only if its dual $M^*$ admits a good filtration if and only if $M^\tau$ admits a good filtration.
\end{rem}

\begin{defn}
	A module $M$ is a \emph{tilting module} if it admits a Weyl filtration and a good filtration.
\end{defn}

An important result is that the tensor product of two tilting modules is again a tilting module. We will use it several times in the next sections without further reference.
\begin{thm}\label{tensor_tilting}
	Let $M,N$ be two tilting modules. Then $M\otimes N$ is a tilting module.
\end{thm}
\begin{proof}
	This is a direct consequence of \cite[Theorem 1]{Mathieu}.
\end{proof}

Like the simple modules, the indecomposable tilting modules are classified by their highest weight.

\begin{prop}[{\cite[II E.6]{Jantzen}}]\label{tilting}
	Let $\lambda\in X^+$ be a dominant weight. There exists a unique indecomposable tilting module $T(\lambda)$ with highest weight $\lambda$ and $m_{T(\lambda)}(\lambda)=1$. Moreover, for every tilting module $T$, there exist dominant weights $\nu_1,\ldots,\nu_n\in X^+$ such that
	$$T\cong \bigoplus_{i=1}^nT(\nu_i).$$
\end{prop}

\begin{lemma}\label{tilting_coselfdual}
	Every tilting module is contravariantly self-dual.
\end{lemma}

\begin{proof}
	If $T$ is a tilting module, then $T^\tau$ is a tilting module. By definition, $(-)^\tau$ preserves the weights of the representation. Thus, if $T$ is indecomposable with highest weight $\lambda$, then so is $T^\tau$, and we conclude by uniqueness in Proposition \ref{tilting} that $T(\lambda)^\tau\cong T(\lambda)$.
\end{proof}

\begin{cor}\label{tilting_irred}
	Let $T(\lambda)$ be an indecomposable tilting module. Then $T(\lambda)$ is multiplicity-free if and only if $T(\lambda)$ is irreducible.
\end{cor}

\begin{proof}
	This is a direct consequence of Lemmas \ref{tilting_coselfdual} and \ref{coselfdual_reducible}.
\end{proof}

\begin{lemma}\label{argument_tilting}
	Let $M$ be a tilting module. Let $\eta \in X^+$ be such that $L(\eta)$ is a composition factor of $M$ and $T(\eta)$ is not irreducible. Then $M$ has multiplicity.
\end{lemma}

\begin{proof}
	Using Proposition \ref{tilting}, there exist $\nu_1, \ldots, \nu_s\in X^+$ such that $M\cong \bigoplus_{i=1}^sT(\nu_i)$. There exists $\nu_i\geq \eta$ such that $L(\eta)$ is a composition factor of $T(\nu_i)$. If $\nu_i>\eta$, then $T(\nu_i)$ is not irreducible, hence it has multiplicity by Corollary \ref{tilting_irred}, and so $M$ has multiplicity. If $\nu_i=\eta$, we conclude using the assumption that $T(\eta)$ is not irreducible.
\end{proof}

\begin{lemma}\label{tilting_irreducible}
	Let $\lambda\in X^+$. If $\Delta(\lambda)\cong L(\lambda)$, then $T(\lambda)\cong \nabla(\lambda)\cong L(\lambda)$. Else, $T(\lambda)$ is not irreducible.
\end{lemma}

\begin{proof}
	If $\Delta(\lambda)$ is irreducible, then so is $\nabla(\lambda)$. In particular, $L(\lambda)\cong \Delta(\lambda)\cong \nabla(\lambda)$ admits a Weyl filtration and a good filtration. Thus, $L(\lambda)$ is a tilting module, and we conclude by uniqueness in Proposition \ref{tilting}.
	
	Otherwise, $\Delta(\lambda)$ appear in the Weyl filtration of $T(\lambda)$, hence $T(\lambda)$ is not irreducible.
\end{proof}

\subsection{Characters}

A lot of information about a $G$-module $M$ is given by the dimensions of its weight spaces. 
These informations are encoded in the character of the module, a notion that we define in this section. 
Later on, we will use those characters to compute the composition factors of tensor products of simple $G$-modules, and in particular to show that some of them are multiplicity-free.

\begin{lemma}[{\cite[Lemma 15.3]{Testerman_Gunter}}] \label{orbit}
	Let $M$ be an irreducible $G$-module, $\lambda\in X$ and $w\in W$. Then
	$$m_M(\lambda)=m_M(w\lambda).$$
\end{lemma}

\begin{defn}
	Let $M$ be a $G$-module. Its \emph{character} is the formal sum
	$$\ch M:=\sum_{\lambda\in X} m_M(\lambda)e^\lambda \in \Z[X],$$
	where $\Z[X]$ has $\Z$-basis $\{e^\lambda|\; \lambda\in X\}$.
\end{defn}
We denote by $\Z[X]^W$ the fixed points in $\Z[X]$ for the natural action of $W$. By \linebreak Lemma \ref{orbit}, we have $\ch M\in \Z[X]^W$ for every $G$-module $M$.

\begin{rem}
	Let $M,N$ be two $G$-modules. We have $\ch(M\oplus N)=\ch M +\ch N$ and \linebreak$\ch(M\otimes N)=\ch M \cdot \ch N$.
\end{rem}

\begin{notation}
	We denote the character of the Weyl module of highest weight $\lambda\in X^+$ by
	$$\chi(\lambda):=\ch \Delta(\lambda).$$
\end{notation}

\begin{thm}[Weyl's character formula]\label{Weyl_character_formula}
	For $\lambda\in X^+$, we have
	$$\chi(\lambda)=\frac{\sum_{w\in W}\det(w)e^{w(\lambda+\rho)}}{\sum_{w\in W}\det(w)e^{w(\rho)}}.$$
\end{thm}

\begin{cor}[Weyl's degree formula]\label{Weyl_degree}
	For $\lambda\in X^+$, we have 
	$$\dim \Delta(\lambda)=\frac{\prod_{\alpha\in \Phi^+}(\lambda+\rho,\alpha)}{\prod_{\alpha\in \Phi^+}(\rho,\alpha)}.$$
\end{cor}
Proofs of Weyl's character formula and Weyl's degree formula are given in \cite[24]{Humphreys_Lie}.\\

Weyl's character formula allows us to extend our definition of character for non-dominant weights.
\begin{defn}
	Let $\lambda \in X$. The \emph{Weyl character} associated to $\lambda$ is the formal element
	$$\chi(\lambda):=\frac{\sum_{w\in W}\det(w)e^{w(\lambda+\rho)}}{\sum_{w\in W}\det(w)e^{w(\rho)}}.$$
\end{defn}

\begin{lemma}\label{action_characters}
	For $\lambda\in X$, we have 
	\begin{enumerate}[label = \emph{(\arabic*)}] 
		\item $\chi(w\sbullet \lambda )=\det(w)\chi(\lambda)\quad \forall w\in W$,
		\item $\chi(\lambda)=0\quad \forall\lambda\in D\setminus X^+$.
	\end{enumerate}
\end{lemma}

\begin{proof}
	For $g\in W$, we have
	\begin{align*}
		\chi(g\sbullet \lambda )&=\frac{\sum_{w\in W}\det(w)e^{w(g\sbullet\lambda+\rho)}}{\sum_{w\in W}\det(w)e^{w(\rho)}}=\frac{\sum_{w\in W}\det(w)e^{w(g(\lambda+\rho))}}{\sum_{w\in W}\det(w)e^{w(\rho)}}\\
		&=\det(g)\frac{\sum_{w\in W}\det(wg)e^{wg(\lambda+\rho)}}{\sum_{w\in W}\det(w)e^{w(\rho)}}=\det(g)\frac{\sum_{w\in W}\det(w)e^{w(\lambda+\rho)}}{\sum_{w\in W}\det(w)e^{w(\rho)}}\\
		&=\det(g)\chi(\lambda).
	\end{align*}
	Now observe that for $\lambda\in D\setminus X^+$, there exists $\alpha\in\Pi$ such that $s_\alpha\sbullet \lambda=\lambda$. Therefore, $\chi(\lambda)=-\chi(\lambda)$ so $\chi(\lambda)=0$.
\end{proof}

The following lemma will be useful to compute an explicit decomposition of a product of characters into a sum of irreducible or Weyl characters. It ensures the existence and the uniqueness of such a decomposition. We will use it several times in the next sections without further reference.

\begin{lemma}[{\cite[II 5.8]{Jantzen}}]\label{basis_character}
	The set of characters of irreducible modules \linebreak$\{\ch L(\lambda)|\; \lambda\in X^+\}$ is a $\Z$-basis of $\Z[X]^W$. Moreover, the set of Weyl characters \linebreak$\{\chi(\lambda)|\; \lambda\in X^+\}$ is a $\Z$-basis of $\Z[X]^W$.
\end{lemma}

\begin{prop}[{\cite[Proposition 2.1]{Stembridge}}]\label{product_characters}
	For $\lambda,\mu \in X^+$, we have $$\chi(\lambda)\chi(\mu)=\sum_{\nu\in X}m_{\Delta(\lambda)}(\nu)\cdot\chi(\mu+\nu).$$
\end{prop}

\begin{cor}\label{products_characters_bis}
	For $\lambda,\mu \in X^+$, we have 
	$$\ch L(\lambda)\cdot\chi(\mu)=\sum_{\nu\in X}m_{L(\lambda)}(\nu)\cdot\chi(\mu+\nu).$$
\end{cor}

\begin{proof}
	Since $\ch L(\lambda)\in \Z[X]^W$, there exist $\lambda_1,\ldots,\lambda_n\in X^+$ and $a_1,\ldots,a_n\in \Z$ such that $\ch L(\lambda)=\sum\limits_{i=1}^na_i\chi(\lambda_i)$ (Lemma \ref{basis_character}). For $\nu\in X$, we have $m_{L(\lambda)}(\nu)=\sum\limits_{i=1}^na_im_{\Delta(\lambda_i)}(\nu)$. Using Proposition \ref{product_characters} in the second equality below, we get
	\begin{align*}
		\ch L(\lambda)\cdot\chi(\mu)&=\left(\sum_{i=1}^na_i\chi(\lambda_i)\right)\chi(\mu)=\sum_{i=1}^na_i\sum_{\nu\in X}m_{\Delta(\lambda_i)}(\nu)\cdot\chi(\mu+\nu)\\
		&=\sum_{\nu\in X}\left(\sum_{i=1}^na_im_{\Delta(\lambda_i)}(\nu)\right)  \cdot\chi(\mu+\nu)=\sum_{\nu\in X}m_{L(\lambda)}(\nu)\cdot\chi(\mu+\nu).\qedhere
	\end{align*}
\end{proof}

\subsubsection{Jantzen $p$-sum formula}

As previously claimed, the Weyl modules are not always irreducible in positive characteristic. Thus, it will be useful to compute their composition factors. An important tool for this computation is the so-called Jantzen $p$-sum formula. We will use it to compute the composition factors of Weyl modules with $p$-restricted highest weight.\\

Let $m\in \N^*$ be a positive integer. Recall that $p$ is a fixed prime. Let $a,b\in \N$ be such that $p\nmid b$ and $m=p^ab$. The \emph{$p$-adic valuation} of $n$ is $\nu_p(m):=a$.

\begin{prop}[Jantzen $p$-sum formula, {\cite[II 8.19]{Jantzen}}]\label{Jantzen_p_sum}
	Let $\lambda\in X^+$. There exists a filtration
	$$\Delta(\lambda)\supseteq \Delta(\lambda)^1\supseteq\Delta(\lambda)^2\supseteq\ldots$$
	such that 
	$$\sum_{i>0}\ch \Delta(\lambda)^i=\sum_{\alpha\in \Phi^+}\; \sum_{0<mp<(\lambda+\rho,\alpha^\vee)}\nu_p(mp)\chi(s_{\alpha,mp}\sbullet\lambda)$$
	and 
	$$\Delta(\lambda)/\Delta(\lambda)^1\cong L(\lambda).$$
\end{prop}

\begin{notation}
	We set
	$$\JSF(\lambda):=\sum_{\alpha\in \Phi^+}\; \sum_{0<mp<(\lambda+\rho,\alpha^\vee)}\nu_p(mp)\chi(s_{\alpha,mp}\sbullet\lambda).$$
\end{notation}

\begin{rem}
	Observe that, for $\lambda\in X^+$ a $p$-restricted weight, we have
	$$(\lambda+\rho,\alpha^\vee)\leq p\sum_{i=1}^n(\omega_i,\alpha^\vee)=p(\rho,\alpha^\vee)\leq p(\rho,\alpha_h^\vee)=p(h-1).$$
	Therefore, if $p\geq h-1$, all the $p$-adic valuations in $\JSF(\lambda)$ are equal to $1$.
\end{rem}

\begin{rem}\label{Jantzen_2_comp}
	Let $\lambda\in X^+$. If there exists $\mu\in X^+$ such that $\JSF(\lambda)=\ch L(\mu)$, then $\Delta(\lambda)$ admits two composition factors, $L(\lambda)$ and $L(\mu)$. Since $\Delta(\lambda)/\rad\Delta(\lambda)\cong L(\lambda)$, it follows that $\Delta(\lambda)$ admits a unique composition series, given by $[L(\lambda),L(\mu)]$.
\end{rem}

\subsection{Linkage principle}
Another useful tool to compute the composition factors of a Weyl module is the Strong Linkage Principle. It allows us to show that Weyl modules with highest weight in the fundamental alcove are irreducible.

\begin{prop}[The Strong Linkage Principle, {\cite[II 6.13]{Jantzen}}]\label{strong_linkage}
	Let $\lambda,\mu\in X^+$ be dominant weights. If $$[\Delta(\lambda):L(\mu)]>0,$$ 
	then $\mu\uparrow\lambda$.
\end{prop}

\begin{prop}[{\cite[II 6.24]{Jantzen}}]\label{link_reverse}
	Let $\lambda \in X^+$ be a dominant weight. Suppose that $\mu\in X$ is maximal in the set $\{\nu\in X|\;\nu\uparrow\lambda,\; \nu\neq \lambda\}$ with respect to the ordering $\uparrow$.
	If $\mu\in X^+$ and $\mu\notin \{\lambda-p\alpha|\; \alpha\in \Phi^+\}$, then
	$$[\Delta(\lambda):L(\mu)]=1.$$
\end{prop}

\begin{lemma}\label{fundamental_alcove}
	For every $\lambda\in X^+\cap \widehat{C_1}$ we have $L(\lambda)\cong \Delta(\lambda)$.
\end{lemma}

\begin{proof}
	Let $\mu \in X^+$ be such that $L(\mu)$ is a composition factor of $\Delta(\lambda)$. By the Strong Linkage Principle (Proposition \ref{strong_linkage}), we have $\mu \uparrow \lambda$, and in particular, $\mu\leq \lambda$, so $\mu \in \widehat{C_1}$ by \linebreak Lemma \ref{C1}. Moreover, $\mu\in W_p\sbullet \lambda$. Since $\overline{C_1}$ is a fundamental domain for the dot action of $W_p$ (Lemma \ref{W_simply_transitively}), we have $\mu\in \overline{C_1}\cap W_p\sbullet \lambda=\{\lambda\}$. We conclude that $\mu=\lambda$, therefore $\Delta(\lambda)$ is irreducible.
\end{proof}

\begin{lemma}\label{A2_fundamental}
Let $\lambda\in \widehat{C_1}\cap X^+$ and $\nu\in X^+$ such that $\nu \leq \lambda$. Then $L(\nu)\cong\Delta(\nu)$.
\end{lemma}

\begin{proof}
This is a direct consequence of Lemmas \ref{C1} and \ref{fundamental_alcove}.
\end{proof}

\begin{lemma}[{\cite[II 4.16]{Jantzen}}]\label{Weyl_tensor_Weyl}
	Let $\lambda,\;\mu\in X^+$. The tensor product $\Delta(\lambda)\otimes \Delta(\mu)$ admits a submodule isomorphic to $\Delta(\lambda+\mu)$.
\end{lemma}

\subsection{An argument to count multiplicities}\label{argument}
In this subsection, we provide an argument to compute the multiplicities of each composition factor of the tensor product of two simple modules. We will use it several times later to show that some tensor products of two simple modules have multiplicity.

\begin{argument}\label{argument1}
Let $\lambda,\mu\in X^+$ be dominant weights, and let $M=L(\lambda)\otimes L(\mu)$. A vector $v\otimes w\in M$ is a weight vector of weight $\nu$ if and only if $v$ is a weight vector in $L(\lambda)$ of weight $\nu_1$, $w$ is a weight vector in $L(\mu)$ of weight $\nu_2$ and $\nu_1+\nu_2=\nu$. Therefore, 
$$m_M(\nu)=\sum_{\substack{\nu_1,\nu_2\in X\\ \nu_1+\nu_2=\nu}}m_{L(\lambda)}(\nu_1)m_{L(\mu)}(\nu_2).$$

Suppose that $\nu_1, \ldots, \nu_s\in X^+$ is the complete list of the dominant weights corresponding to all composition factors of $M$ (with multiplicity). For every weight $\eta\in X$, we have $$m_M(\eta)=\sum_{i=1}^sm_{L(\nu_i)}(\eta).$$

We compute the $\nu_i$'s as follows. We set $\nu_1=\lambda+\mu$.
Suppose that we have already $\nu_1,\ldots,\nu_t$ for $t<s$. Let $\eta\in X^+$ be such that 
$$m_M(\eta)>\sum_{i=1}^tm_{L(\nu_i)}(\eta)\quad \text{ and }\quad m_M(\nu)=\sum_{i=1}^tm_{L(\nu_i)}(\nu)\quad \text{for every dominant weight } \nu> \eta.$$
It follows that $m_{L(\nu_i)}(\nu)=0$ for every $i>t$ and every $\nu>\eta$. In particular, $\nu_i\ngtr \eta$ for every $i>t$. Moreover, there exists $i>t$ such that $m_{L(\nu_i)}(\eta)>0$, thus $\nu_i\geq \eta$. Therefore, we deduce that $\nu_i=\eta$ for some $i>t$, and we can choose $\nu_{t+1}=\eta$.
\end{argument}

\cleardoublepage
\section{Properties of tensor products}
In this section, we establish some properties of tensor products. We state Steinberg's tensor product theorem (see \cite[II 3.16]{Jantzen} for a proof), which allows us to restrict our attention to tensor products of irreducible modules with $p$-restricted highest weight in order to answer our question.

\subsection{Steinberg's tensor product theorem}

\begin{thm}[{\cite[Theorem 9.4.3]{Springer_LAG}}]\label{Frob_endo}
	The Frobenius endomorphism $k\to k:c\mapsto c^p$ induces a group endomorphism $F:G\to G$ given by $x_\alpha(c)\mapsto x_\alpha(c^p)$ for all $\alpha\in \Phi,\; c\in k$ and $F(t)=t^p$ for $t\in T$.
\end{thm}


\begin{notation}
	Let $\phi:G\to \GL(M)$ be a representation. We denote by $M^{(p^i)}$ the vector space $M$ with $G$-action corresponding to the representation $\phi\circ F^i:G\to \GL(M)$ where $F$ is the group endomorphism described in Theorem \ref{Frob_endo}.
\end{notation}

\begin{prop}[{\cite[Proposition 16.6]{Testerman_Gunter}}]
	For every $\lambda\in X^+$, we have an isomorphism of $G$-modules $L(p\lambda)\cong L(\lambda)^{(p)}$.
\end{prop}

\begin{thm}[Steinberg's tensor product theorem]\label{Steinberg}
	Let $\lambda=\lambda_0+p\lambda_1+\ldots+p^n\lambda_n \in X^+$ be a dominant weight such that $\lambda_i$ is $p$-restricted for all $i\in\{0,\ldots,n\}$. We have an isomorphism
	$$L(\lambda)\cong L(\lambda_0)\otimes L(\lambda_1)^{(p)}\otimes\cdots \otimes L(\lambda_n)^{(p^n)}.$$
\end{thm}

\subsection{Reduction to $p$-restricted highest weights}


\begin{lemma}[{\cite[Lemma 4.12]{article_jonathan}}]\label{completely_reducible}
	Let $\lambda,\mu \in X^+$ be dominant weights. If the $G$-module $L(\lambda)\otimes L(\mu)$ is multiplicity-free, then it is completely reducible.
\end{lemma}

\begin{proof}
	Using Remarks \ref{cont_tens_prod} and \ref{irred_cont_self}, we have
	$$(L(\lambda)\otimes L(\mu))^\tau\cong L(\lambda)^\tau\otimes L(\mu)^\tau \cong L(\lambda)\otimes L(\mu).$$ 
	We conclude by Lemma \ref{coselfdual_reducible}.
\end{proof}

\begin{thm}[{\cite[Theorem A]{article_jonathan}}]\label{p_restricted}
	 Let $\lambda,\mu\in X^+$ be $p$-restricted dominant weights. If the $G$-module $L(\lambda)\otimes L(\mu)$ is completely reducible, then all its composition factors are $p$-restricted.
\end{thm}

\begin{cor}\label{composition_p}
	Let $\lambda,\mu\in X^+$ be $p$-restricted dominant weights. If the $G$-module $L(\lambda)\otimes L(\mu)$ is multiplicity-free, then all its composition factors are $p$-restricted.
\end{cor}

\begin{proof}
	This is a direct consequence of Lemma \ref{completely_reducible} and Theorem \ref{p_restricted}.
\end{proof}

\begin{cor}\label{sum_p_res}
	Let $\lambda,\mu \in X^+$ be $p$-restricted weights. If $\lambda+\mu$ is not $p$-restricted, then $L(\lambda)\otimes L(\mu)$ has multiplicity.
\end{cor}

\begin{proof}
	We know that $L(\lambda+\mu)$ is a composition factor of $L(\lambda)\otimes L(\mu)$. Thus we conclude by Corollary \ref{composition_p}.
\end{proof}

\begin{prop} \label{padic}
		Let $\lambda=\lambda_0+p\lambda_1+\ldots+p^n\lambda_n,\;\mu=\mu_0+p\mu_1+\ldots+p^n\mu_n\in X^+$ be dominant weights such that $\lambda_i,\mu_i$ are $p$-restricted for all $i\in\{0,\ldots,n\}$. Then $L(\lambda)\otimes L(\mu)$ is multiplicity-free if and only if $L(\lambda_i)\otimes L(\mu_i)$ is multiplicity-free for all $i\in\{0,\ldots,n\}$.
\end{prop}

\begin{proof}
	First, using Steinberg's tensor product theorem (Theorem \ref{Steinberg}), we have an isomorphism 
	$$L(\lambda)\otimes L(\mu)\cong \bigotimes_{i=0}^n( L(\lambda_i)^{(p^i)}\otimes L(\mu_i)^{(p^i)})\cong \bigotimes_{i=0}^n(L(\lambda_i)\otimes L(\mu_i))^{(p^i)}.$$
	Clearly, if there exists $i\in \{0,\ldots,n\}$ such that $L(\lambda_i)\otimes L(\mu_i)$ has multiplicity, then \linebreak$(L(\lambda_i)\otimes L(\mu_i))^{(p^i)}$ has multiplicity, thus $L(\lambda)\otimes L(\mu)$ has multiplicity.
	
	Now suppose that $L(\lambda_i)\otimes L(\mu_i)$ is multiplicity-free for all $i\in\{0,\ldots,n\}$. By Lemma \ref{completely_reducible}\linebreak and Corollary \ref{composition_p}, for each $i\in\{0,\ldots,n\}$, there exist $\nu_i^1,\ldots,\nu_i^{m_i}\in X^+$ distinct and \linebreak$p$-restricted such that
	$$L(\lambda_i)\otimes L(\mu_i) \cong \bigoplus_{j=1}^{m_i}L(\nu_i^j).$$
	Therefore, we have 
	$$L(\lambda)\otimes L(\mu)\cong\bigotimes_{i=0}^n\bigoplus_{j=1}^{m_i}L(\nu_i^j)^{(p^i)}\cong\bigoplus_{\vec j}\bigotimes_{i=0}^nL(\nu_i^{j_i})^{(p^i)},$$
	where $\vec j=(j_0,\ldots,j_n)$ runs over $\bigtimes\limits_{i=0}^n\{1,\ldots,m_i\}$. Since all the weights $\nu_i^j$ are $p$-restricted, we can use Steinberg's tensor product theorem again to get 
	$$L(\lambda)\otimes L(\mu)\cong\bigoplus_{\vec j}L(\sum_{i=0}^np^i\nu_i^{j_i}).$$
	By uniqueness of the $p$-adic expansion of a weight, we conclude that $L(\lambda)\otimes L(\mu)$ is multiplicity-free.
\end{proof}

By this proposition, in order to classify multiplicity-free tensor products of simple \linebreak$G$-modules, we may restrict our attention to the study of the tensor products of simple modules with $p$-restricted highest weights.

\cleardoublepage
\section{Connections between characteristic $0$ and positive characteristic}
In this section, we show some links between multiplicity-free tensor products in characteristic $0$ and in positive characteristic. In particular, we show that in positive characteristic, multiplicity-free tensor products and completely reducible tensor products are closely related. This will allow us to classify multiplicity-free tensor products for $p=2$ and $G$ of type $A_n$ (see section \ref{An_p2}).

\begin{notation}
	For $\lambda \in X^+$, we denote by $L_{\C}(\lambda)$ the irreducible $G_{\C}$-module of highest weight $\lambda$ (over $\C$). Recall that $\ch L_{\C}(\lambda)=\chi(\lambda)$.
\end{notation}

\begin{prop}\label{char0}
	Let $\lambda,\mu \in X^+$ be dominant weights such that $\Delta(\lambda)\cong L(\lambda)$ and \linebreak$\Delta(\mu)\cong L(\mu)$. Suppose that $\Delta(\nu)\cong L(\nu)$ for all dominant weights $\nu\in X^+$ such that $L(\nu)$ is a composition factor of $L(\lambda)\otimes L(\mu)$. Then $L(\lambda)\otimes L(\mu)$ is multiplicity-free if and only if $L_{\C}(\lambda)\otimes L_{\C}(\mu)$ is multiplicity-free.
\end{prop}

\begin{proof}
	We show that the composition factors in characteristic $p$ and in characteristic $0$ are the same. Let $\nu_1,\ldots,\nu_m\in X^+$ be distinct weights and $k_1,\ldots,k_m\in \Z_{>0}$ be such that $\ch( L(\lambda)\otimes L(\mu))=k_1\ch L(\nu_1)+\ldots+k_m\ch L(\nu_m)$. By assumption, we have
	\begin{align*}
		\chi(\lambda)\chi(\mu)&=\ch L(\lambda)\ch L(\mu)=\ch( L(\lambda)\otimes L(\mu))=k_1\ch L(\nu_1)+\ldots+k_m\ch L(\nu_m)\\
		&=k_1\chi(\nu_1)+\ldots+k_m\chi(\nu_m).
	\end{align*}
	By uniqueness of the composition factors and linear independence of the Weyl characters (Lemma \ref{basis_character}), it follows that $L_{\C}(\lambda)\otimes L_{\C}(\mu)$ is multiplicity-free if and only if $k_i=1$ for all $i\in \{1,\ldots,m\}$ if and only if $L(\lambda)\otimes L(\mu)$ is multiplicity-free.
\end{proof}

\begin{cor}\label{char0_bis}
	Let $\lambda,\mu \in X^+$ be dominant weights such that $\Delta(\lambda)\cong L(\lambda)$ and $\Delta(\mu)\cong L(\mu)$. If $\Delta(\nu)\cong L(\nu)$ for all dominant weights $\nu\leq \lambda+\mu$, then $L(\lambda)\otimes L(\mu)$ is multiplicity-free if and only if $L_{\C}(\lambda)\otimes L_{\C}(\mu)$ is multiplicity-free.
\end{cor}

\begin{proof}
	If $L(\nu)$ is a composition factor of $L(\lambda)\otimes L(\nu)$, then $\nu$ is dominant and $\nu\leq \lambda+\mu$. Thus we can apply Proposition \ref{char0}
\end{proof}

\begin{cor}\label{sum_C1}
	Let $\lambda,\mu \in X^+$ be such that $\lambda+\mu\in \widehat{C_1}$. The module $L(\lambda)\otimes L(\mu)$ is multiplicity-free if and only if $L_{\C}(\lambda)\otimes L_{\C}(\mu)$ is multiplicity-free.
\end{cor}

\begin{proof}
	By Lemmas \ref{C1} and \ref{fundamental_alcove}, Corollary \ref{char0_bis} applies in case $\lambda+\mu\in \widehat{C_1}$.
\end{proof}

The next theorem allows us to find the explicit decomposition of some tensor products as a direct sum of irreducible modules. Then we prove a more general version which allows us to conclude that some tensor products of irreducible modules are multiplicity-free without computing the explicit decomposition.

\begin{thm}\label{complete_red_mult_free}
	Let $\lambda,\mu\in X^+$ be $p$-restricted dominant weights such that the following hold:
	\begin{enumerate}[label = \emph{(\arabic*)}] 
		\item $L(\lambda)\cong\Delta(\lambda)$,
		\item $L(\mu)\cong \Delta(\mu)$,
		\item $L(\lambda)\otimes L(\mu)$ is completely reducible and
		\item $L_{\C}(\lambda)\otimes L_{\C}(\mu)$ is multiplicity-free.
	\end{enumerate}
	Then $L(\lambda)\otimes L(\mu)$ is multiplicity-free. Moreover, if we have the decomposition 
	$$L_{\C}(\lambda)\otimes L_{\C}(\mu)\cong\bigoplus_{i=1}^mL_{\C}(\nu_m)$$ for distinct dominant weights $\nu_1,\ldots,\nu_m\in X^+$, then $L(\nu_i)\cong \Delta(\nu_i)$ for all $i\in \{1,\ldots,m\}$ and 
	$$L(\lambda)\otimes L(\mu)\cong \bigoplus_{i=1}^mL(\nu_m).$$
\end{thm}

\begin{proof}
	To simplify the notation, we set $M:=L(\lambda)\otimes L(\mu)$. 
	
	Since $L(\lambda)$ and $L(\mu)$ are tilting modules, $M$ must be a tilting module. Thus there exist $\nu_1,\ldots,\nu_m\in X^+$ such that $M\cong T(\nu_1)\oplus\ldots\oplus T(\nu_m)$ (Proposition \ref{tilting}). Since $M$ is completely reducible, $T(\nu_i)$ must be completely reducible for every $i\in\{1,\ldots,m\}$. Therefore, $T(\nu_i)\cong L(\nu_i)\cong \Delta(\nu_i)$ for every $i\in\{1,\ldots,m\}$, so all dominant weights $\nu$ such that $L(\nu)$ is a composition factors of $M$ satisfy $\Delta(\nu)\cong L(\nu)$. Therefore, we can conclude by Proposition \ref{char0}\linebreak (and its proof).
\end{proof}

\begin{thm}\label{comp_red_0}
	Let $\lambda,\mu\in X^+$ be $p$-restricted dominant weights such that the following hold:
	\begin{enumerate}[label = \emph{(\arabic*)}] 
		\item $L(\lambda)\otimes L(\mu)$ is completely reducible and
		\item $L_{\C}(\lambda)\otimes L_{\C}(\mu)$ is multiplicity-free.
	\end{enumerate}
	Then $L(\lambda)\otimes L(\mu)$ is multiplicity-free.
\end{thm}

\begin{proof}
	We know that we have a surjection $\Delta(\lambda)\to L(\lambda)$ and a surjection $\Delta(\mu)\to L(\mu)$. By right exactness of tensor products, we get a surjection 
	$$\phi :\Delta(\lambda)\otimes\Delta(\mu)\to L(\lambda)\otimes\Delta(\mu)\to L(\lambda)\otimes L(\mu).$$
	Using Theorem \ref{Weyl_filtration}, we fix a Weyl filtration
	$$0=V_0\subseteq V_1\subseteq\ldots\subseteq V_m=\Delta(\lambda)\otimes\Delta(\mu).$$
	Thus there exist $\nu_1,\ldots,\nu_m\in X^+$ such that $V_i/V_{i-1}\cong \Delta(\nu_i)$ for $i=1,\ldots, m$. In particular, we have 
	$$\chi(\lambda)\chi(\mu)=\sum_{i=1}^m\chi(\nu_i).$$
	Since $L_{\C}(\lambda)\otimes L_{\C}(\mu)$ is multiplicity-free, we deduce that $\nu_i\neq \nu_j$ for all $i\neq j$.
	
	For $i\in \{1,\ldots,m\}$, we set $W_i:=\phi(V_i)$ and we denote by $\phi_i:V_i\to W_i$ the restriction and corestriction of the map $\phi$. In particular, $\phi_i$ is surjective for all $i\in \{1,\ldots,m\}$. By construction, we have a filtration 
	$$0=W_0\subseteq W_1\subseteq\ldots \subseteq W_m=L(\lambda)\otimes L(\mu)$$
	of $L(\lambda)\otimes L(\mu)$.

	Now we identify the quotients $W_i/W_{i-1}$ for $i=1,\ldots,m$. Let $\iota_i:V_{i-1}\to V_i$ be the inclusion map. We have the following situation:
	$$\begin{tikzcd}
		0 & {V_{i-1}} & {V_i} & {\Delta(\nu_i)} & 0 \\
		0 & {W_{i-1}} & {W_i} & {W_i/W_{i-1}} & 0.
		\arrow[from=1-1, to=1-2]
		\arrow["{\iota_i}", hook, from=1-2, to=1-3]
		\arrow[two heads, from=1-3, to=1-4]
		\arrow[from=1-4, to=1-5]
		\arrow[from=2-1, to=2-2]
		\arrow[hook, from=2-2, to=2-3]
		\arrow["{\theta_i}",two heads, from=2-3, to=2-4]
		\arrow[from=2-4, to=2-5]
		\arrow["{\phi_{i-1}}"', two heads, from=1-2, to=2-2]
		\arrow["{\phi_i}"', two heads, from=1-3, to=2-3]
	\end{tikzcd}$$
	Since $\theta_i\circ \phi_i\circ \iota_i=0$, i.e. $V_{i-1}\subseteq \ker (\theta_i\circ \phi_i)$, there exists a unique map $\psi_i :V_i/V_{i-1}\to W_i/W_{i-1}$ such that the following diagram commutes:
	$$\begin{tikzcd}
		0 & {V_{i-1}} & {V_i} & {\Delta(\nu_i)} & 0 \\
		0 & {W_{i-1}} & {W_i} & {W_i/W_{i-1}} & 0.
		\arrow[from=1-1, to=1-2]
		\arrow["{\iota_i}", hook, from=1-2, to=1-3]
		\arrow[two heads, from=1-3, to=1-4]
		\arrow[from=1-4, to=1-5]
		\arrow[from=2-1, to=2-2]
		\arrow[hook, from=2-2, to=2-3]
		\arrow["{\theta_i}",two heads, from=2-3, to=2-4]
		\arrow[from=2-4, to=2-5]
		\arrow["{\phi_{i-1}}"', two heads, from=1-2, to=2-2]
		\arrow["{\phi_i}"', two heads, from=1-3, to=2-3]
		\arrow["{\psi_i}"', dashed, from=1-4, to=2-4]
	\end{tikzcd}$$
	In particular, $\psi_i$ is surjective because $\theta_i$ and $\phi_i$ are surjective.
	By Proposition \ref{submodule_comp_red}, $W_i$ is completely reducible and so $W_i/W_{i-1}$ is completely reducible. Therefore, $\rad \Delta(\nu_i)\subseteq \ker (\psi_i)$,\linebreak so $\psi_i$ factors as
	$$\begin{tikzcd}
		{\Delta(\nu_i)} & {W_i/W_{i-1}} \\
		{L(\nu_i)}
		\arrow["{\psi_i}", from=1-1, to=1-2]
		\arrow[from=1-1, to=2-1]
		\arrow["{\bar\psi_i}"', from=2-1, to=1-2]
	\end{tikzcd}$$
	Therefore, $W_i/W_{i-1}$ is a quotient of $L(\nu_i)$, so either $W_i/W_{i-1}\cong L(\nu_i)$ or $W_i/W_{i-1}=0$. We deduce that $L(\lambda)\otimes L(\nu)$ is isomorphic to a submodule of $\bigoplus\limits_{i=1}^mL(\nu_i)$, and in particular it is multiplicity-free since all the $\nu_i$'s are distinct.
\end{proof}

\cleardoublepage

\section{$\SL_2$}

In this section, we establish the classification of multiplicity-free tensor products of simple $\SL_2(k)$-modules. Recall that $\SL_2(k)$ has root system $\Phi$ of type $A_1$, so $\Phi=\{\alpha,-\alpha\}$. The fundamental weight $\omega$ satisfies $\alpha=2\omega$. Since all weights are integer multiples of $\omega$, we will identify the set of weights with $\Z$. Under this identification, $\omega$ corresponds to $1$, the positive root $\alpha$ to $2$, and dominant weights are in bijection with $\N$. Moreover, we have $\rho=\frac12\alpha=\omega$ and it corresponds to $1$.\\

Since there exists a unique positive root in $\Phi$, alcoves are in bijection with $\Z$, with 
$$C_n=\{\lambda\in X_{\R}|\; (n-1)p<(\lambda+\rho,\alpha^\vee)<np\}.$$
Using the identification previously described, we identify $X_{\R}$ with $\R$. Thus we get 
$$C_n=\{\lambda\in\R|\; (n-1)p<\lambda+1<np\}.$$
In particular, there exists a unique $p$-restricted alcove, the fundamental alcove $C_1$, and by Lemma \ref{fundamental_alcove}, $L(\lambda)\cong\Delta(\lambda)$ for all $p$-restricted dominant weights $\lambda\in X^+$.\\

We start by computing the Weyl characters and the decomposition of the product of two such characters.

\begin{lemma}\label{mult_sl2}
	Let $\lambda\in X^+$. Then 
	$$\chi(\lambda)=\sum_{i=0}^\lambda e^{\lambda-2i}.$$
\end{lemma}

\begin{proof}
	We show this result using Weyl's character formula (Theorem \ref{Weyl_character_formula}). We have 
	\begin{align*}
		(e^1-e^{-1})\sum_{i=0}^\lambda e^{\lambda-2i}&=\sum_{i=0}^\lambda e^{\lambda+1-2i}-\sum_{i=0}^\lambda e^{\lambda-1-2i}\\
		&=e^{\lambda+1}+\sum_{i=1}^\lambda e^{\lambda+1-2i}-e^{-\lambda-1}-\sum_{i=0}^{\lambda-1} e^{\lambda-1-2i}\\
		&=e^{\lambda+1}-e^{-\lambda-1}.
	\end{align*}
Therefore \begin{ceqn}
	\begin{equation*}
	\chi(\lambda)=\frac{e^{\lambda+1}-e^{-\lambda-1}}{e^1-e^{-1}}=\sum_{i=0}^\lambda e^{\lambda-2i}.\qedhere
	\end{equation*}
\end{ceqn}
\end{proof}

\begin{prop}[Clebsch-Gordan formula]\label{character_sl2}
	For $\lambda,\mu \in X^+$ with $\lambda\geq \mu$, we have $$\chi(\lambda)\chi(\mu)=\chi(\lambda+\mu)+\chi(\lambda+\mu-2)+\ldots+\chi(\lambda-\mu+2)+\chi(\lambda-\mu).$$
\end{prop}

\begin{proof}
	By Proposition \ref{product_characters} and Lemma \ref{mult_sl2}, we have
		\begin{ceqn}
		\begin{equation*}
			\chi(\lambda)\chi(\mu)=\chi(\mu)\chi(\lambda)=\sum_{i=0}^\mu \chi(\lambda+\mu-2i).\qedhere
		\end{equation*}
	\end{ceqn}
\end{proof}

We are now ready to state the main result of this section.

\begin{prop}\label{sl2}
	Let $\lambda,\mu\in X^+$ be $p$-restricted dominant weights. Then $L(\lambda)\otimes L(\mu)$ is multiplicity-free if and only if $\lambda+\mu$ is $p$-restricted.
\end{prop}

\begin{proof}
	The "only if" direction is a direct consequence of Corollary \ref{sum_p_res}. For the "if" direction, observe that $L_{\C}(\lambda)\otimes L_{\C}(\mu)$ is multiplicity-free for all $\lambda,\mu\in X^+$ by Proposition \ref{character_sl2}. If $\lambda+\mu<p$, then $\lambda+\mu\in \widehat{C_1}$, and $L(\lambda)\otimes L(\mu)$ is multiplicity-free by Corollary \ref{sum_C1}. 
\end{proof}

Finally, we state the classification theorem for $\SL_2(k)$.

\begin{thm}
	Let $\lambda=\lambda_0+p\lambda_1+\ldots+p^n\lambda_n,\;\mu=\mu_0+p\mu_1+\ldots+p^n\mu_n\in X^+$ be dominant weights with $\lambda_i,\mu_i$ $p$-restricted for all $i\in\{0,\ldots,n\}$. Then $L(\lambda)\otimes L(\mu)$ is multiplicity-free if and only if $\lambda_i+\mu_i$ is $p$-restricted for all $i\in\{0,\ldots,n\}$.
\end{thm}

\begin{proof}
	By Proposition \ref{padic}, $L(\lambda)\otimes L(\mu)$ is multiplicity-free if and only if $L(\lambda_i)\otimes L(\mu_i)$ is multiplicity-free for all $i\in\{0,\ldots,n\}$. We conclude by Proposition \ref{sl2}.
\end{proof}

\cleardoublepage

\section{$\SL_3$}

In this section, we establish the classification of multiplicity-free tensor products of simple $\SL_3$-modules with $p$-restricted highest weight.
We fix $G=\SL_3(k)$ with root system $\Phi$ of type $A_2$ and $\Pi=\{\alpha_1,\alpha_2\}$ a base of $\Phi$. With respect to this base, we have $\rho =\alpha_1+\alpha_2$.
For $\lambda=x\omega_1+y\omega_2\in X_{\R}$, we write $\lambda=(x,y)$. In particular, we have $\alpha_1=(2,-1)$, $\alpha_2=(-1,2)$ and $\rho=(1,1)$.

 We will prove the following theorem:

\begin{thm}\label{SL3}
	Let $\lambda=(a,b),\mu=(c,d)\in X^+$ be non-zero $p$-restricted dominant weights. Up to the reordering of $\lambda$ and $\mu$, the module $L(\lambda)\otimes L(\mu)$ is multiplicity-free if and only if one of the following holds:
	\begin{enumerate}[label = \emph{(\arabic*)}] 
		\item $b=d=0$ and $a+c<p$,
		\item $a=c=0$ and $b+d<p$,
		\item $b=c=0$ and $a+d<p-1$ or $(a,d)\in \{(p-1,1),(1,p-1)\}$,
		\item $b=0$ and $a+c+d<p-1$,
		\item $a=0$ and $b+c+d<p-1$,
		\item $b=0$, $c+d=p-1$, $a+c<p$ and $a<c+2$,
		\item $a=0$, $c+d=p-1$, $b+d<p$ and $b<d+2$,
		\item $b=0$, $c+d>p-1$, $a+c<p$ and $a+d<p$,
		\item $a=0$, $c+d>p-1$, $b+c<p$ and $b+d<p$ or
		\item $a+b<p-1$, $c+d=p-1$, $a+b+c<p$ and $a+b+d<p$.
	\end{enumerate}
\end{thm}

 We start by recalling some facts about the structure of those simple modules. Then we consider the relation between characters of simple modules and Weyl characters. Finally, we will establish a sequence of propositions which yield the classification.\\


\subsection{Alcoves}
In this subsection, we describe the $p$-restricted alcoves of a root system of type $A_2$. There are two such alcoves which we define to be the \emph{fundamental alcove}
$$C_1:=\{\lambda\in X_{\R}|\;(\lambda+\rho,\alpha_1^\vee)>0,\;(\lambda+\rho,\alpha_2^\vee)>0,\;(\lambda+\rho,(\alpha_1+\alpha_2)^\vee)<p\}$$
and the \emph{second alcove}
$$C_2:=\{\lambda\in X_{\R}|\;(\lambda+\rho,\alpha_1^\vee)<p,\;(\lambda+\rho,\alpha_2^\vee)<p,\;(\lambda+\rho,(\alpha_1+\alpha_2)^\vee)>p\}.$$
Therefore, all $p$-restricted dominant weights belong to $\widehat{C_1}\sqcup\widehat{C_2}$. Using the notation previously defined, we get
$$\widehat{C_1}\cap X^+=\{(a,b)\in \N^2| \; a+b\leq p-2\}$$
and
$$\widehat{C_2}\cap X^+=\{(a,b)\in \N^2|\; a<p,\; b<p, \; a+b> p-2\}.$$
Furthermore, we define the walls
\begin{ceqn}
\begin{align*}
	&F_{1,2}:=\overline{C_1}\cap \overline{C_2}=\{(x,y)\in [-1,p-1]^2|\; x+y=p-2\},\\
	&F_{2,3}:=\{\lambda\in \overline{C_2}|\; (\lambda+\rho,\alpha_1^\vee)=p\},\\
	&F_{2,3'}:=\{\lambda\in \overline{C_2}|\; (\lambda+\rho,\alpha_2^\vee)=p\}.
\end{align*}

Thus we get 
\begin{align*}
	&F_{1,2}\cap X^+=\{(a,b)\in \N^2|\; a+b=p-2\},\\
	&F_{2,3}\cap X^+=\{(p-1,b)\in \N^2|\; b\leq p-1\},\\
	&F_{2,3'}\cap X^+=\{(a,p-1)\in \N^2|\; a\leq p-1\}.
\end{align*}
\end{ceqn}

\begin{rem}
	Following the labelling of the alcoves from \cite{SL3_5}, we have $F_{2,i}=\overline{C_2}\cap \overline{C_i}$ for $i=3,3'$.
\end{rem}

Let us illustrate these alcoves with a picture. The blue arrows form the root system. The black arrows are the fundamental weights, generating the weight lattice (in black). The region containing the dominant weights is coloured in green. The red triangles are the walls of the alcoves.

\begin{center}
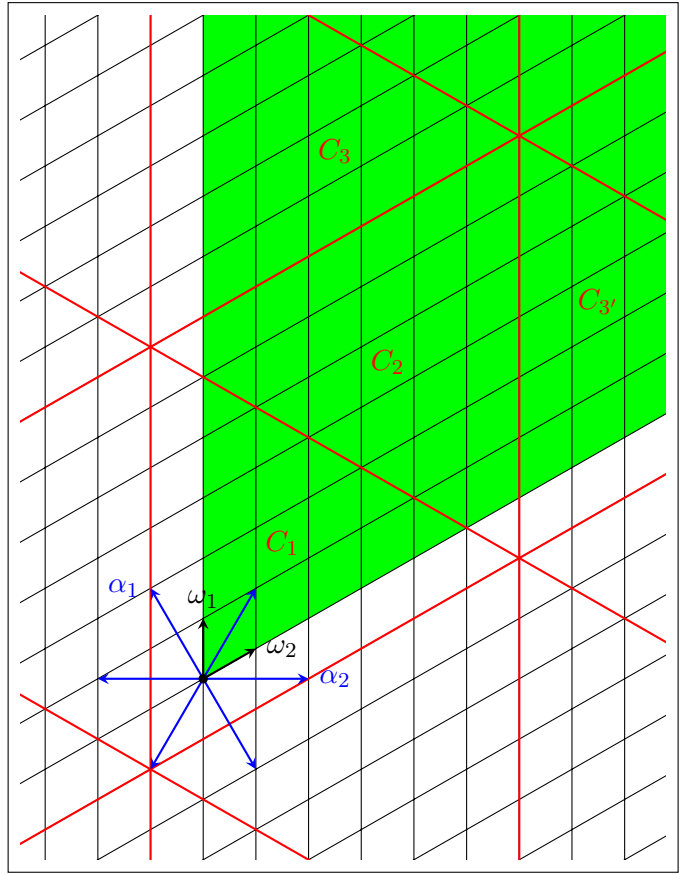

\begin{tikzpicture}[framed,baseline={(current bounding box.center)},scale=0.8]
	\clip (-3,-3) rectangle (7.6,11);
	
	\pgftransformcm{cos(30)}{sin(30)}{cos(90)}{sin(90)}{\pgfpoint{0cm}{0cm}}
	
	\fill[green,opacity=.3] (0,0) rectangle (24,24);
	
	\draw[black, thin] (-24,-24) grid[step=1cm] (24,24);
	
	\draw[red,thick, shift={(-1,-1)}, scale=7] (-16,-16) grid (24,24);
	\foreach \x in {-4,-3,...,5}{
		\foreach \y in {-4,-3,...,5}{
			\draw[red, thick]  (-1+7*\x,-1+7*\y) -- (-1+7*\x+7,-1+7*\y-7);
		}
	}
	
	\draw [blue,thick,-stealth](0,0) -- (2,-1) node[right]{$\alpha_2$};
	\draw [blue,thick,-stealth](0,0) -- (-1,2) node[left]{$\alpha_1$};
	\draw [blue,thick,-stealth](0,0) -- (1,1);
	\draw [blue,thick,-stealth](0,0) -- (-2,1);
	\draw [blue,thick,-stealth](0,0) -- (1,-2);
	\draw [blue,thick,-stealth](0,0) -- (-1,-1);
	
	\draw [thick,-stealth](0,0) -- (1,0) node[right]{$\omega_2$};
	\draw [thick,-stealth](0,0) -- (0,1) node[above]{$\omega_1$};
	
	\node[red]  at (3.5,3.5){$C_2$};
	\node[red]  at (1.5,1.5){$C_1$};
	\node[red]  at (2.5,7.5){$C_3$};
	\node[red]  at (7.5,2.5){$C_{3'}$};
	
	
	\node[draw,circle,fill,scale=.3] at (0,0){};
	
\end{tikzpicture}
\captionof{figure}{Alcoves for $A_2$ and $p=7$.}
\end{center}



\subsection{Structure of Weyl modules and weights in irreducible modules}
In this subsection, we compute the composition factors of Weyl modules with $p$-restricted highest weight and the multiplicity of certain weights in irreducible modules with $p$-restricted highest weight.

\begin{lemma}\label{Weyl_A2}
	Let $\lambda=(a,b)\in X^+$ be a $p$-restricted dominant weight. 
	\begin{itemize}
		\item If $a+b<p-1$ (i.e. $\lambda\in \widehat{C_1}$), then $\Delta(\lambda)\cong L(\lambda)$.
		\item If $a=p-1$ or $b=p-1$ (i.e. $\lambda\in F_{2,3}$ or $\lambda\in F_{2,3'}$), then $\Delta(\lambda)\cong L(\lambda)$.
		\item Else (i.e. if $\lambda\in C_2$), then $\Delta(\lambda)$ admits exactly two composition factors, $L(\lambda)$ and $L(\lambda-(a+b+2-p)\rho)$.
	\end{itemize}
In particular, $T(\lambda)$ is irreducible if and only if $\lambda\in \widehat{C_1}\cup F_{2,3}\cup F_{2,3'}$.
\end{lemma}

\begin{proof}
	We use the Jantzen $p$-sum formula (Proposition \ref{Jantzen_p_sum}) to show this result.
	\begin{itemize}
		\item If $\lambda \in \widehat{C_1}$, then $\Delta(\lambda)\cong L(\lambda)$ by Lemma \ref{fundamental_alcove}.
		\item If $\lambda \in \widehat{C_2}$, then
		$$\JSF(\lambda)=\chi(s_{\alpha_1+\alpha_2,p}\sbullet \lambda)=\chi(\lambda-(a+b+2-p)\rho)=\chi(p-b-2,p-a-2).$$
		If $a=p-1$ or $b=p-1$, then $\JSF(\lambda)=0$ because $s_{\alpha_1+\alpha_2,p}\sbullet \lambda\in D\setminus X^+$ and $\Delta(\lambda)$ is irreducible. Else, $\Delta(\lambda)$ admits the unique composition series 
		$$[L(\lambda),L(\lambda-(a+b+2-p)\rho)].$$
	\end{itemize}
The last claim follows directly from Lemma \ref{tilting_irreducible}.
\end{proof}
We can also prove Lemma \ref{Weyl_A2} using Proposition \ref{link_reverse}, see for example \cite[Lemma \nolinebreak2.1.4]{these_Nathan}.

To simplify the notation we define the map 
	$$\Theta:\begin{array}{rcl}
		X&\to &\Z\\
		(a,b)&\mapsto&a+b+2-p.
	\end{array}$$

\begin{rem}\label{characters_alcove2}
	By the proof of Lemma \ref{Weyl_A2}, we have
	$$\ch L(\lambda)=\chi(\lambda)-\chi(\lambda-\Theta(\lambda)\rho)$$
	for all dominant weights $\lambda\in \widehat{C_2}$.
\end{rem}

\begin{lemma}\label{C2_theta}
	Let $\lambda\in X^+$ be $p$-restricted. Then $\lambda\in \widehat{C_2}$ if and only if $\Theta(\lambda)\geq 1$.
\end{lemma}

\begin{proof}
	This is a direct consequence of the definitions of $\Theta$ and $\widehat{C_2}$.
\end{proof}

\begin{lemma}[{\cite[1.35]{Testerman_AMS}}]\label{-11}
	Let $\lambda=(a,b)\in X^+$ be $p$-restricted with $a\neq 0$ and $b\neq 0$. Then
	$$m_{L(\lambda)}(\lambda-\alpha_1-\alpha_2)=\begin{cases}
		1 &\text{if } a+b=p-1,\\
		2 &\text{otherwise.}
	\end{cases}$$
\end{lemma}

\begin{lemma}\label{weights_a0}
	Let $\lambda=(a,0)\in X^+$ with $a<p$. Then
	$$m_{L(\lambda)}(\lambda-i\alpha_1-j\alpha_2)=\begin{cases}
		1 \quad \text{ if } \; 0\leq j\leq i\leq a,\\
		0 \quad \text{ else.}
	\end{cases}$$
\end{lemma}

\begin{proof}
	First, observe that for $a<p$, we have $\Delta(\lambda)\cong L(\lambda)$ by Lemma \ref{Weyl_A2}. 
	Recall that $\rho=\omega_1+\omega_2$. Using Weyl's degree formula (Corollary \ref{Weyl_degree}) with $(\alpha_1,\alpha_1)=2$, we have 
	$$\dim L(\lambda)=\frac{(\lambda+\rho,\alpha_1)(\lambda+\rho,\alpha_2)(\lambda+\rho,\alpha_1+\alpha_2)}{(\rho,\alpha_1)(\rho,\alpha_2)(\rho,\alpha_1+\alpha_2)}=\frac{(a+1)\cdot 1\cdot (a+2)}{1\cdot 1\cdot 2}=\frac{(a+2)(a+1)}{2}.$$
	Observe that $A=\{\lambda-i\alpha_1-j\alpha_2\}_{0\leq j\leq i\leq a}$ is saturated with highest weight $\lambda$. Thus, by Proposition \ref{saturated}, we have $m_{L(\lambda)}(\nu)=m_{\Delta(\lambda)}(\nu)\geq 1$ for all $\nu\in A$. Moreover, $$|A| =\sum_{i=0}^a\sum_{j=0}^i1=\sum_{i=0}^a(i+1)=\sum_{i=1}^{a+1}i=\frac{(a+2)(a+1)}{2}=\dim L(\lambda).$$
	Therefore, $m_{L(\lambda)}(\nu)= 1$ for all $\nu\in A$, as claimed.
\end{proof}

\subsection{The classification}

In this subsection, we prove several propositions which yield the classification of \linebreak multiplicity-free tensor products of simple $\SL_3(k)$-modules with $p$-restricted highest weight \linebreak(Theorem \ref{SL3}). We start by stating a theorem from Stembridge (\cite[Theorem 1.1.A]{Stembridge}) which classifies multiplicity-free tensor products of simple $\SL_3(\C)$-modules.

\begin{thm}
	\label{char_0_A2}
	Let $(a,b),(c,d)\in X^+$ be dominant weights. Then $L_{\C}(a,b) \otimes L_{\C}(c,d)$ is multiplicity-free if and only if $a\cdot b\cdot c\cdot d=0$.
\end{thm}

Before moving on to the classification, recall that $L(a,b)^*\cong L(b,a)$ for all $a,b\in \N$ and observe that $L(\lambda)\otimes L(\mu)$ is multiplicity-free if and only if $(L(\lambda)\otimes L(\mu))^*\cong L(\lambda)^*\otimes L(\mu)^*$ is multiplicity-free. This allows us to treat several cases simultaneously.

\subsubsection{$L(a,0)\otimes L(c,0)$}

\begin{prop}\label{ac}
	Let $\lambda=(a,0),\mu=(c,0)\in X^+$ be $p$-restricted. Then $L(\lambda)\otimes L(\mu)$ is multiplicity-free if and only if $a+c<p$.
\end{prop}

\begin{proof}
	The "only if" direction is a direct consequence of Corollary \ref{sum_p_res}.
	
	Suppose that $a+c<p$. By Lemma \ref{Weyl_A2}, we have $\Delta(\lambda)\cong L(\lambda)$ and $\Delta(\mu)\cong L(\mu)$. If $\nu\leq \lambda+\mu$ is a dominant weight, then either $\nu=\lambda+\mu$ or $\nu \in \widehat{C_1}$. We have $\Delta(\lambda+\mu)\cong L(\lambda+\mu)$\linebreak by Lemma \ref{Weyl_A2} and $\Delta(\nu)\cong L(\nu)$ for all $\nu \in \widehat{C_1}$ by Lemma \ref{A2_fundamental}. Thus we can apply \linebreak Corollary \ref{char0_bis}. By Theorem \ref{char_0_A2}, $L_{\C}(\lambda)\otimes L_{\C}(\mu)$ is multiplicity-free, thus $L(\lambda)\otimes L(\mu)$ is multiplicity-free.
\end{proof}

\begin{rem}
	By duality, $L(0,b)\otimes L(0,d)$ is multiplicity-free if and only if $b+d<p$.
\end{rem}

\subsubsection{$L(a,0)\otimes L(0,d)$}
\begin{prop}\label{ad}
	Let $\lambda=(a,0),\mu=(0,d)\in X^+$ be $p$-restricted with $a,d\neq 0$. Then $L(\lambda)\otimes L(\mu)$ is multiplicity-free if and only if $a+d<p-1$, or $(a,d)\in\{(p-1,1),(1,p-1)\}.$
\end{prop}

\begin{proof}
	 To simplify the notation, we set $M:=L(\lambda)\otimes L(\mu)$. 
	 
	 By Lemma \ref{Weyl_A2}, $L(\lambda)$ and $L(\mu)$ are tilting modules. By Theorem \ref{tensor_tilting}, $M$ is a tilting module. We know that $L(\lambda+\mu)$ is a composition factor of $M$. If $\lambda+\mu \in C_2$, then we conclude that $M$ has multiplicity using Lemma \ref{argument_tilting}.\\

	 If $a+d<p-1$ (i.e. if $\lambda+\mu\in \widehat{C_1}$), we apply Corollary \ref{sum_C1} and Theorem \ref{char_0_A2} to conclude that $M$ is multiplicity-free.\\

	Suppose $a=p-1$ and $d>1$ (in particular, $\lambda+\mu \in F_{2,3}$). We use Lemmas \ref{-1} and \ref{weights_a0} and Argument \ref{argument1} to compute
	\begin{align*}
		&m_M(\lambda+\mu)=1, &&m_M(\lambda+\mu-\alpha_1)=1,&&m_M(\lambda+\mu-\alpha_2)=1,\\
		&m_M(\lambda+\mu-\alpha_1-\alpha_2)=3.
	\end{align*}
	Since $m_{L(\lambda+\mu)}(\lambda+\mu-\alpha_1)=m_{L(\lambda+\mu)}(\lambda+\mu-\alpha_2)=1$ (Lemma \ref{-1}), we conclude that\linebreak $L(\lambda+\mu-\alpha_1)$ and $L(\lambda+\mu-\alpha_2)$ are not composition factors of $M$.
 Moreover,\linebreak $m_{L(\lambda+\mu)}(\lambda+\mu-\alpha_1-\alpha_2)=2$ by Lemma \ref{-11}, thus $L(\lambda+\mu-\alpha_1-\alpha_2)$ is a composition factor of $M$. But $\lambda+\mu-\alpha_1-\alpha_2=(p-2,d-1)\in C_2$, so $T(\lambda+\mu-\alpha_1-\alpha_2)$ is not irreducible by Lemma \ref{Weyl_A2}. By Lemma \ref{argument_tilting}, we conclude that $M$ has multiplicity.
	The case $a>1,\; d=p-1$ is symmetric.\\
	
	Finally consider the case $a=p-1,\; d=1$ (in particular, $\lambda+\mu \in F_{2,3}$). Again, we have 
		\begin{align*}
		&m_M(\lambda+\mu)=1,&&m_M(\lambda+\mu-\alpha_1)=1,&&m_M(\lambda+\mu-\alpha_2)=1.
	\end{align*}
	Therefore, $L(\lambda+\mu)$ is a composition factor, and all other composition factors have highest weight $\nu$ with $\nu\leq \lambda+\mu-\alpha_1-\alpha_2$, $\nu\leq \lambda+\mu-2\alpha_1$ or $\nu\leq \lambda+\mu-2\alpha_2$. In particular, $\nu\in \widehat{C_1}$. By Lemma \ref{Weyl_A2}, this implies that for all dominant weights $\nu$ such that $L(\nu)$ is a composition factor of $M$, we have $L(\nu)\cong \Delta(\nu)$. Using Proposition \ref{char0} and Theorem \ref{char_0_A2}, we conclude that $M$ is multiplicity-free.
	The case $a=1,\; d=p-1$ is symmetric
\end{proof}

\subsubsection{$L(a,0)\otimes L(c,d)$}

\begin{prop}\label{acd_1}
		Let $\lambda=(a,0),\mu=(c,d)\in X^+$ be $p$-restricted with $a,c,d\neq 0$ and $a+c+d<p-1$ (i.e. $\lambda+\mu\in \widehat{C_1}$). Then $L(\lambda)\otimes L(\mu)$ is multiplicity-free.
\end{prop}

\begin{proof}
	This is a direct consequence of Corollary \ref{sum_C1} and Theorem \ref{char_0_A2}.
\end{proof}

\begin{prop}\label{acd_2}
	Let $\lambda=(a,0),\mu=(c,d)\in X^+$ be $p$-restricted with $a,c,d\neq 0$, \linebreak$a+c+d\geq p-1$ and $L(\mu)\cong \Delta(\mu)$ (i.e. $\lambda+\mu \notin \widehat{C_1}$ and $\mu\in \widehat{C_1}\cup F_{2,3}\cup F_{2,3'}$). Then $L(\lambda)\otimes L(\mu)$ has multiplicity.
\end{prop}

\begin{proof}
	To simplify the notation, we set $M:=L(\lambda)\otimes L(\mu)$. By assumption $\Delta(\lambda)\cong L(\lambda)$ and $\Delta(\mu)\cong L(\mu)$, in particular these two modules are tilting modules. By Theorem \ref{tensor_tilting}, $M$ is a tilting module.\\
	
	If $a+c\geq p$, we conclude directly by Corollary \ref{sum_p_res} that $M$ has multiplicity. So we can restrict our attention to the case $a+c < p$.
	
	Observe that for $d=p-2$, the condition $\Delta(\lambda)\cong L(\lambda)$ forces $c=p-1$ (Lemma \ref{Weyl_A2}), so $a+c\geq p$ and $M$ has multiplicity.\\

	Suppose that $d= p-1$. By Lemma \ref{-1} and Argument \ref{argument1}, we have
		\begin{align*}
		&m_M(\lambda+\mu)=1,
		&&m_M(\lambda+\mu-\alpha_1)=2.
	\end{align*}
	By Lemma \ref{-1}, we have
	$m_{L(\lambda+\mu)}(\lambda+\mu-\alpha_1)=1$, so we deduce that $L(\lambda+\mu-\alpha_1)$ is a composition factor of $M$. But $\lambda+\mu-\alpha_1=(a+c-2,p)$ is not $p$-restricted, thus $M$ has multiplicity by Corollary \ref{composition_p}.\\
	
	Suppose that $a+c<p-1$ and $d<p-1$. By assumption, $\lambda+\mu\in C_2$, and $T(\lambda+\mu)$ is not irreducible by Lemma \ref{Weyl_A2}. Therefore, by Lemma \ref{argument_tilting}, $M$ has multiplicity. \\
	
	Finally, consider the case $a+c=p-1,\; d< p-2$. Again, we have 
		\begin{align*}
		&m_M(\lambda+\mu)=1, &&m_M(\lambda+\mu-\alpha_1)=2,
	\end{align*}
	and $L(\lambda+\mu-\alpha_1)$ is a composition factor of $M$. Observe that $\lambda+\mu-\alpha_1=(p-3,d+1)\in C_2$, so $T(\lambda+\mu-\alpha_1)$ is not irreducible. We conclude by Lemma \ref{argument_tilting} that $M$ has multiplicity.
\end{proof}

\begin{prop}\label{acd_3}
	Let $\lambda=(a,0),\mu=(c,d)\in X^+$ be $p$-restricted with $a,c,d\neq 0$. If $d+\min(a,c)\geq p$, then $L(\lambda)\otimes L(\mu)$ has multiplicity.
\end{prop}

\begin{proof}
	To simplify the notation, we set $M:=L(\lambda)\otimes L(\mu)$. 
	
	If $a+c\geq p$, we conclude directly using Corollary \ref{sum_p_res}. Thus we can assume $a+c<p$.
	We show inductively that $[M:L(\lambda+\mu-s\alpha_1)]=1$ for $0\leq s \leq p-d\leq \min(a,c)$.
	For $0\leq s \leq \min(a,c)$, using Lemma \ref{-1}, we have
	$$m_M(\lambda+\mu-s\alpha_1)=\sum_{i=0}^sm_{L(\lambda)}(\lambda-i\alpha_1)\cdot m_{L(\mu)}(\mu-(s-i)\alpha_1)=s+1.$$
	Moreover, if $0\leq i<s\leq p-d$, then $\lambda+\mu-i\alpha_1=(a+c-2i,d+i)$ is $p$-restricted and by Lemma \ref{-1}, we have
	$$\sum_{i=0}^{s-1}m_{L(\lambda+\mu-i\alpha_1)}(\lambda+\mu-s\alpha_1)=s.$$
	We know that $[M:L(\lambda+\mu)]=1$, and combining the two previous equations, we conclude inductively that $[M:L(\lambda+\mu-s\alpha_1)]=1$ for $0\leq s \leq p-d$.
	In particular, $L(\lambda+\mu-(p-d)\alpha_1)$ is a composition factor of $M$. But $\lambda+\mu-(p-d)\alpha_1=(a+c-2(p-d),p)$ is not $p$-restricted. We conclude by Corollary \ref{composition_p} that $M$ has multiplicity.
\end{proof}

\begin{prop}\label{acd_p-1}
		Let $\lambda=(a,0),\mu=(c,d)\in X^+$ be $p$-restricted with $a,c,d\neq 0$,\linebreak $c+d=p-1$ and $a+c < p$ (in particular, $\mu\in \widehat{C_2}$). Then $L(\lambda)\otimes L(\mu)$ is multiplicity-free if and only if $a<c+2$.
\end{prop}

\begin{proof}
		We set $M:=L(\lambda)\otimes L(\mu)$. 
		
		By Remark \ref{characters_alcove2}, we have $\ch L(\mu)=\chi(\mu)-\chi(\mu-\rho)$ and $\ch L(\lambda)=\chi(\lambda)$. Therefore, by Proposition \ref{product_characters}, we have
		\begin{align*}
			\ch M&=\chi(\lambda)(\chi(\mu)-\chi(\mu-\rho))=\sum_{\nu\in X}m_{\Delta(\lambda)}(\nu)(\chi(\mu+\nu)-\chi(\mu+\nu-\rho)).
		\end{align*}
	
	Now we use Lemma \ref{weights_a0} to rewrite this sum. We get 
	\begin{align}
		\ch M&=\sum_{\nu\in X}m_{\Delta(\lambda)}(\nu)(\chi(\mu+\nu)-\chi(\mu+\nu-\rho))\nonumber\\
		&=\sum_{i=0}^a\sum_{j=0}^i\chi(\mu+\lambda-i\alpha_1-j\alpha_2)-\chi(\mu+\lambda-i\alpha_1-j\alpha_2-\rho)\nonumber\\
		&=\sum_{i=0}^a\sum_{j=0}^i\chi(\mu+\lambda-i\alpha_1-j\alpha_2)-\sum_{i=0}^a\sum_{j=0}^i\chi(\mu+\lambda-i\alpha_1-j\alpha_2-\rho)\nonumber\\
		&=\sum_{i=0}^a\sum_{j=0}^i\chi(\mu+\lambda-i\alpha_1-j\alpha_2)-\sum_{i=1}^{a+1}\sum_{j=1}^i\chi(\mu+\lambda-i\alpha_1-j\alpha_2)\nonumber\\
		&=\sum_{i=0}^a\chi(\mu+\lambda-i\alpha_1)-\sum_{j=1}^{a+1}\chi(\mu+\lambda-(a+1)\alpha_1-j\alpha_2)\nonumber\\
		&=\sum_{i=0}^a\chi(\mu+\lambda-i\alpha_1)-\sum_{j=0}^{a}\chi(\mu+\lambda-(a+1)\alpha_1-(a+1-j)\alpha_2)\nonumber\\
		&=\sum_{i=0}^a\chi(\mu+\lambda-i\alpha_1)-\chi(\mu+\lambda-(a+1)\alpha_1-(a+1-i)\alpha_2)\nonumber\displaybreak\\
		&=\sum_{i=0}^a\chi(\mu+\lambda-i\alpha_1)-\chi(\mu+\lambda-i\alpha_1-(a+1-i)\rho)\nonumber\\
		&=\sum_{i=0}^a\chi(\mu+\lambda-i\alpha_1)-\chi(\mu+\lambda-i\alpha_1-\Theta(\mu+\lambda-i\alpha_1)\rho),\label{eq200}
	\end{align}
where in the last equality, we used that, for $i\in\{0,\ldots,a\}$, we have $$\mu+\lambda-i\alpha_1=(a+c-2i,d+i),$$ so $\Theta(\mu+\lambda-i\alpha_1)=a+1-i$. Now we dissociate several cases.\\

First, suppose that $a\leq c$. In this case, for $i\in\{0,\ldots,a\}$ we have $0\leq a+c-2i<p$ and $0<d+i\leq d+c < p$, therefore $\mu+\lambda-i\alpha_1$ is dominant and $p$-restricted. Moreover, $$\Theta(\mu+\lambda-i\alpha_1)=a+1-i\geq 1.$$ 
By Lemma \ref{C2_theta}, we have $\mu+\lambda-i\alpha_1\in \widehat{C_2}$ for all $i\in\{0,\ldots,a\}$. Using Remark \ref{characters_alcove2}, we get
\begin{align}
		\ch M &=\sum_{i=0}^a\chi(\mu+\lambda-i\alpha_1)-\chi(\mu+\lambda-i\alpha_1-\Theta(\mu+\lambda-i\alpha_1)\rho)\nonumber\\
		&=\sum_{i=0}^a\ch L(\mu+\lambda-i\alpha_1).\label{eq203}
\end{align}
Thus we conclude that $M$ is multiplicity-free.\\

Suppose that $a=c+1$. Using line \eqref{eq200}, we have
\begin{align*}
	\ch M=&\sum_{i=0}^c\Big(\chi(\mu+\lambda-i\alpha_1)-\chi(\mu+\lambda-i\alpha_1-\Theta(\mu+\lambda-i\alpha_1)\rho)\Big)\\
	&\hspace{4em}+\chi(\mu+\lambda-a\alpha_1)-\chi(\mu+\lambda-a\alpha_1-\Theta(\mu+\lambda-a\alpha_1)\rho).
\end{align*}
In this case, for $i\in\{0,\ldots,c\}$ we have $0\leq a+c-2i<p$ and $0<d+i\leq d+c < p$ so $\mu+\lambda-i\alpha_1$ is dominant and $p$-restricted. Moreover, 
$$\Theta(\mu+\lambda-i\alpha_1)=a+1-i\geq 2,$$ 
so by Lemma \ref{C2_theta}, $\mu+\lambda-i\alpha_1\in \widehat{C_2}$ for all $i\in\{0,\ldots,c\}$. By Remark \ref{characters_alcove2}, we have
\begin{ceqn}
\begin{equation}
	\chi(\mu+\lambda-i\alpha_1)-\chi(\mu+\lambda-i\alpha_1-\Theta(\mu+\lambda-i\alpha_1)\rho)=\ch L(\mu+\lambda-i\alpha_1).\label{eq1}
\end{equation}
Moreover, $$\mu+\lambda-a\alpha_1=(-1,d+a)\in D\setminus X^+,$$ 
so by Lemma \ref{action_characters} we have 
\begin{equation}
	\chi(\mu+\lambda-a\alpha_1)=0.\label{eq2}
\end{equation} 
\end{ceqn}
Finally, $$\mu+\lambda-a\alpha_1-\Theta(\mu+\lambda-a\alpha_1)\rho =\mu+\lambda-a\alpha_1-\rho=(-2,d+a-1)=(-2,p-1).$$ 
Using again Lemma \ref{action_characters}, we get
\begin{align}
\chi(-2,p-1)&=-\chi(s_{\alpha_1}\sbullet(-2,p-1))=-\chi(0,p-2)=-\ch L(0,p-2)\nonumber\\&=-\ch L(\lambda+\mu-a\alpha_1-\alpha_2)\label{eq3}
\end{align}
where in the last equality, we used that
$\lambda+\mu-a\alpha_1-\alpha_2=(0,p-2)\in X^+$
is $p$-restricted and Lemma \ref{Weyl_A2}. Combining lines \eqref{eq1}, \eqref{eq2} and \eqref{eq3}, we conclude that
\begin{align}
\ch M=&\sum_{i=0}^c\Big(\chi(\mu+\lambda-i\alpha_1)-\chi(\mu+\lambda-i\alpha_1-\Theta(\mu+\lambda-i\alpha_1)\rho)\Big)\nonumber\\
&\hspace{4em}+\chi(\mu+\lambda-a\alpha_1)-\chi(\mu+\lambda-a\alpha_1-\Theta(\mu+\lambda-a\alpha_1)\rho)\nonumber\\
=&\ch L(\lambda+\mu-a\alpha_1-\alpha_2)+\sum_{i=0}^c\ch L(\mu+\lambda-i\alpha_1),\label{eq230}
\end{align}
and $M$ is multiplicity-free.\\

Finally, we consider the case $a\geq c+2$. We show that $M$ has multiplicity. Using line \eqref{eq200}, we have 
\begin{align}
	\ch M=&\sum_{i=0}^c\Big(\chi(\mu+\lambda-i\alpha_1)-\chi(\mu+\lambda-i\alpha_1-\Theta(\mu+\lambda-i\alpha_1)\rho)\Big)\label{eq4}\\
	&\hspace{4em}+\sum_{i=c+1}^a\Big(\chi(\mu+\lambda-i\alpha_1)-\chi(\mu+\lambda-i\alpha_1-\Theta(\mu+\lambda-i\alpha_1)\rho)\Big).\label{eq5}
\end{align}
By the same argument as in the previous case (see line \eqref{eq1}), we write \eqref{eq4} as 
\begin{equation}
	\sum_{i=0}^c\Big(\chi(\mu+\lambda-i\alpha_1)-\chi(\mu+\lambda-i\alpha_1-\Theta(\mu+\lambda-i\alpha_1)\rho)\Big)=\sum_{i=0}^c\ch L(\mu+\lambda-i\alpha_1).\label{eq6}
\end{equation}
We rewrite the sum of line \eqref{eq5} by making the expression of the weights explicit. We get
\begin{align}
	&\sum_{i=c+1}^a\Big(\chi(\mu+\lambda-i\alpha_1)-\chi(\mu+\lambda-i\alpha_1-\Theta(\mu+\lambda-i\alpha_1)\rho)\Big)\nonumber\\
	&=\sum_{i=c+1}^a\Big(\chi(a+c-2i,d+i)-\chi(c-i-1,d+2i-a-1)\Big)\nonumber\\
	&=\sum_{i=1}^{a-c}\Big(\chi(a-c-2i,d+c+i)-\chi(-i-1,d+2c+2i-a-1)\Big)\nonumber\\
	&=\sum_{i=1}^{a-c}\chi(a-c-2i,p-1+i)\label{eq7}\\
	&\hspace{4em}-\sum_{i=1}^{a-c}\chi(-i-1,d+2c+2i-a-1)\label{eq8}.
\end{align}
We show that the sum in line \eqref{eq7} is equal to zero. 
Using Lemma \ref{action_characters}, we have
\begin{align*}
	&\sum_{i=1}^{a-c}\chi(a-c-2i,p-1+i)\\
	&\qquad=-\sum_{i=1}^{a-c}\chi(s_{\alpha_1}\sbullet(a-c-2i,p-1+i))\\
	&\qquad=-\sum_{i=1}^{a-c}\chi(-a+c+2i-2,p-i+a-c)\displaybreak\\
	&\qquad=-\sum_{i=1}^{a-c}\chi(-a+c+2(a-c+1-i)-2,p-(a-c+1-i)+a-c)\\
	&\qquad=-\sum_{i=1}^{a-c}\chi(a-c-2i,p-1+i)
\end{align*}
Therefore, 
\begin{ceqn}
\begin{align}
	\sum_{i=1}^{a-c}\chi(a-c-2i,p-1+i)=0.\label{eq9}
\end{align}
\end{ceqn}\\

Now, let us study line \eqref{eq8}. Using Lemma \ref{action_characters}, we have
\begin{align*}
	&\sum_{i=1}^{a-c}\chi(-i-1,d+2c+2i-a-1)=-\sum_{i=1}^{a-c}\chi(s_{\alpha_1}\sbullet(-i-1,d+2c+2i-a-1))\\
	&=-\sum_{i=1}^{a-c}\chi(i-1,d+2c+i-a-1)=-\sum_{i=0}^{a-c-1}\chi(i,p-1+c-a+i).
\end{align*}
Observe that $(i,p-1+c-a+i)\in X^+$ is $p$-restricted for all $i\in\{0,\ldots,a-c-1\}$. We dissociate the cases $a-c$ even and odd.\\

If $a-c$ is even, we get
\begin{align*}
	&\sum_{i=0}^{a-c-1}\chi(i,p-1+c-a+i)=\sum_{i=0}^{\frac{a-c}2-1}\chi(a-c-1-i,p-2-i)+\chi(i,p-1+c-a+i)\\
&=\sum_{i=0}^{\frac{a-c}2-1}\chi(a-c-1-i,p-2-i)\\
&\qquad\qquad\qquad+\chi((a-c-1-i,p-2-i)-\Theta(a-c-1-i,p-2-i)\rho).
\end{align*}
For $i\in \{0,\ldots,\frac{a-c}{2}-1\}$, we have 
$$(a-c-1-i)+(p-2-i)\geq a-c-1+p-(a-c-2)=p+1,$$ thus $(a-c-1-i,p-2-i)\in C_2$. This implies that $$(a-c-1-i,p-2-i)-\Theta(a-c-1-i,p-2-i)\rho \in C_1,$$ 
and in particular, by Lemma \ref{Weyl_A2} we have \begin{align*}
	&\chi((a-c-1-i,p-2-i)-\Theta(a-c-1-i,p-2-i)\rho)\\
	&\qquad\qquad=\ch L((a-c-1-i,p-2-i)-\Theta(a-c-1-i,p-2-i)\rho).
\end{align*}
Therefore, by Remark \ref{characters_alcove2}, we have
\begin{align}
	&\sum_{i=0}^{\frac{a-c}2-1}\chi(a-c-1-i,p-2-i)+\chi((a-c-1-i,p-2-i)-\Theta(a-c-1-i,p-2-i)\rho)\nonumber\\
	&=\sum_{i=0}^{\frac{a-c}2-1}\ch L(a-c-1-i,p-2-i)\nonumber\\
	&\qquad\qquad\qquad+2\chi((a-c-1-i,p-2-i)-\Theta(a-c-1-i,p-2-i)\rho)\nonumber\displaybreak\\
	&=\sum_{i=0}^{\frac{a-c}2-1}\ch L(a-c-1-i,p-2-i)\nonumber\\
	&\qquad\qquad\qquad+2\ch L((a-c-1-i,p-2-i)-\Theta(a-c-1-i,p-2-i)\rho)\nonumber\\
	&=\sum_{i=0}^{\frac{a-c}2-1}\ch L(a-c-1-i,p-2-i)+2\ch L(i,p-1+c-a+i)\nonumber\\
	&=\sum_{i=1}^{\frac{a-c}2}\ch L(a-c-i,p-1-i)+2\ch L(i-1,p-2+c-a+i)\nonumber\\
	&=\sum_{i=1}^{\frac{a-c}2}\ch L(\lambda+\mu-c\alpha_1-i\rho)+2\ch L(\lambda+\mu-c\alpha_1-(a+d-i+2-p)\rho).\label{eq10}
\end{align}\\

Similarly, if $a-c$ is odd, we get 
\begin{align*}
	&\sum_{i=0}^{a-c-1}\chi(i,p-1+c-a+i)\\
	&=\sum_{i=0}^{\frac{a-c-1}2-1}\Big(\chi(i,p-1+c-a+i)+\chi(a-c-1-i,p-2-i)\Big)\\
	&\hspace{4em}+\chi(\frac{a-c-1}{2},p-1+\frac{c-a-1}2).
\end{align*}
Since $(\frac{a-c-1}{2},p-1+\frac{c-a-1}2)\in C_1$, we get, as in the even case,
\begin{align}
	&\sum_{i=0}^{\frac{a-c-1}2-1}\Big(\chi(i,p-1+c-a+i)+\chi(a-c-1-i,p-2-i)\Big)\nonumber\\
	&\qquad\qquad+\chi(\frac{a-c-1}{2},p-1+\frac{c-a-1}2)\nonumber\\
	&=\sum_{i=0}^{\frac{a-c-1}2-1}\ch L(a-c-1-i,p-2-i)+2\ch L(i,p-1+c-a+i)\nonumber\\
	&\qquad\qquad+\ch L(\frac{a-c-1}{2},p-1+\frac{c-a-1}2)\nonumber\\
	&=\sum_{i=1}^{\frac{a-c-1}2}\ch L(\lambda+\mu-c\alpha_1-i\rho)+2\ch L(\lambda+\mu-c\alpha_1-(a+d-i+2-p)\rho)\nonumber\\
	&\qquad\qquad+\ch L(\lambda+\mu-c\alpha_1-\frac{a-c+1}{2}\rho).\label{eq11}
\end{align}
Combining lines \eqref{eq6}, \eqref{eq9} and \eqref{eq10} (respectively \eqref{eq11}), we find, if $a-c$ is even:
\begin{align*}
	\ch M&=\sum_{i=0}^c\ch L(\mu+\lambda-i\alpha_1)\\
	&+\sum_{i=1}^{\frac{a-c}2}\ch L(\lambda+\mu-c\alpha_1-i\rho)+2\ch L(\lambda+\mu-c\alpha_1-(a+d-i+2-p)\rho), 
\end{align*}
and if $a-c$ is odd:
\begin{align*}
	\ch M&=\ch L(\lambda+\mu-c\alpha_1-\frac{a-c+1}{2}\rho)+\sum_{i=0}^c\ch L(\mu+\lambda-i\alpha_1)\\
	&+\sum_{i=1}^{\frac{a-c-1}2}\ch L(\lambda+\mu-c\alpha_1-i\rho)+2\ch L(\lambda+\mu-c\alpha_1-(a+d-i+2-p)\rho).
\end{align*}
In both cases, the second sum is non-empty since $a\geq c+2$ by assumption, and we deduce that $M$ has multiplicity.
\end{proof}

\begin{prop}\label{acd_p}
		Let $\lambda=(a,0),\mu=(c,d)\in X^+$ be $p$-restricted with $a,c,d\neq 0$,\linebreak $c+d>p-1$, $a+c < p$ and $a+d<p$ (in particular, $\mu\in C_2$). Then $L(\lambda)\otimes L(\mu)$ is multiplicity-free.
\end{prop}

\begin{proof}
	We set $M:=L(\lambda)\otimes L(\mu)$, and $m:=\Theta(\mu)=c+d+2-p\geq 2$. Observe that our hypotheses imply $a<c$ and $a<d$. 
	
		By Remark \ref{characters_alcove2} and Lemma \ref{Weyl_A2}, we have $\ch L(\mu)=\chi(\mu)-\chi(\mu-m\rho)$ and $\ch L(\lambda)=\chi(\lambda)$. Therefore, by Proposition \ref{product_characters}, we have
	\begin{align*}
		\ch M&=\ch L(\lambda)\cdot \ch L(\mu)=\chi(\lambda)\cdot(\chi(\mu)-\chi(\mu-m\rho))\\
		&=\sum_{\nu\in X}m_{\Delta(\lambda)}(\nu)(\chi(\mu+\nu)-\chi(\mu+\nu-m\rho)).
	\end{align*}
	
	We use Lemma \ref{weights_a0} to rewrite this sum. We get 
	\begin{align}
		\ch M&=\sum_{\nu\in X}m_{\Delta(\lambda)}(\nu)(\chi(\mu+\nu)-\chi(\mu+\nu-m\rho))\nonumber\\
		&=\sum_{i=0}^a\sum_{j=0}^i\chi(\mu+\lambda-i\alpha_1-j\alpha_2)-\chi(\mu+\lambda-i\alpha_1-j\alpha_2-m\rho)\nonumber\\
		&=\sum_{i=0}^a\sum_{j=0}^i\chi(\mu+\lambda-i\alpha_1-j\alpha_2)-\sum_{i=0}^a\sum_{j=0}^i\chi(\mu+\lambda-i\alpha_1-j\alpha_2-m\rho)\nonumber\\
		&=\sum_{i=0}^a\sum_{j=0}^i\chi(\mu+\lambda-i\alpha_1-j\alpha_2)-\sum_{i=m}^{a+m}\sum_{j=m}^i\chi(\mu+\lambda-i\alpha_1-j\alpha_2).\label{eq400}
	\end{align}

For $0\leq j\leq i\leq a$, we have 
\begin{ceqn}
\begin{align}
	\mu+\lambda-i\alpha_1-j\alpha_2=(a+c+j-2i,d+i-2j)\label{acd_poids}
\end{align} 
with 
$$0<c-a\leq c+a-2i\leq c+a+j-2i\leq c+a -i\leq c+a<p$$ 
and
$$0<d-a\leq d-j\leq d+i-2j\leq d+i\leq d+a <p.$$
Therefore, 
\begin{align}
	\mu+\lambda-i\alpha_1-j\alpha_2\in X^+ \text{ is } p\text{-restricted for all } 0\leq j\leq i\leq a.\label{acd_prestricted}\end{align} 
Moreover, using line \eqref{acd_poids}, we have 
\begin{align}
	\Theta(\lambda+\mu-i\alpha_1-j\alpha_2)=(a+c+j-2i)+(d+i-2j)+2-p=m+a-i-j,\label{acd_theta}
\end{align}
and if $0\leq j\leq i \leq \min\{a,m-1\}$, then 
$$\Theta(\lambda+\mu-i\alpha_1-j\alpha_2)=m+a-i-j\geq 1.$$
Using line \eqref{acd_prestricted} and Lemma \ref{C2_theta}, we deduce that
\begin{align}
	\lambda+\mu-i\alpha_1-j\alpha_2\in \widehat{C_2} \quad \text{ for all }0\leq j\leq i \leq \min\{a,m-1\}. \label{acd_secondalcove}
\end{align}
\end{ceqn}\\

We dissociate the cases $a<m$ and $a\geq m$. \\
If $a<m$, using line \eqref{eq400}, we have
	\begin{align*}
	\ch M&=\sum_{i=0}^a\sum_{j=0}^i\chi(\mu+\lambda-i\alpha_1-j\alpha_2)-\sum_{i=m}^{a+m}\sum_{j=m}^i\chi(\mu+\lambda-i\alpha_1-j\alpha_2)\\
	&=\sum_{i=0}^a\sum_{j=0}^i\chi(\mu+\lambda-i\alpha_1-j\alpha_2)-\sum_{j=m}^{a+m}\sum_{i=j}^{a+m}\chi(\mu+\lambda-i\alpha_1-j\alpha_2)\\
	&=\sum_{i=0}^a\sum_{j=0}^i\chi(\mu+\lambda-i\alpha_1-j\alpha_2)-\sum_{i=m}^{a+m}\sum_{j=i}^{a+m}\chi(\mu+\lambda-j\alpha_1-i\alpha_2)\\
	&=\sum_{i=0}^a\sum_{j=0}^i\chi(\mu+\lambda-i\alpha_1-j\alpha_2)-\sum_{i=0}^{a}\sum_{j=i}^{a}\chi(\mu+\lambda-j\alpha_1-i\alpha_2-m\rho)\\	
	&=\sum_{i=0}^a\sum_{j=0}^i\chi(\mu+\lambda-i\alpha_1-j\alpha_2)-\sum_{i=0}^{a}\sum_{j=0}^{i}\chi(\mu+\lambda-(a-j)\alpha_1-(a-i)\alpha_2-m\rho)\\
	&=\sum_{i=0}^a\sum_{j=0}^i\chi(\mu+\lambda-i\alpha_1-j\alpha_2)-\chi(\mu+\lambda-(a-j)\alpha_1-(a-i)\alpha_2-m\rho)\\
	&=\sum_{i=0}^a\sum_{j=0}^i\chi(\mu+\lambda-i\alpha_1-j\alpha_2)-\chi(\mu+\lambda-i\alpha_1-j\alpha_2-(m+a-i-j)\rho)\\
	&=\sum_{i=0}^a\sum_{j=0}^i\chi(\mu+\lambda-i\alpha_1-j\alpha_2)-\chi(\mu+\lambda-i\alpha_1-j\alpha_2-\Theta(\mu+\lambda-i\alpha_1-j\alpha_2)\rho),
\end{align*}
where in the last equality, we use line \eqref{acd_theta}. Since $a\leq m-1$, we can use line \eqref{acd_secondalcove} and \linebreak Remark \ref{characters_alcove2} to conclude that
	\begin{align}
	\ch M
	&=\sum_{i=0}^a\sum_{j=0}^i\chi(\mu+\lambda-i\alpha_1-j\alpha_2)-\chi(\mu+\lambda-i\alpha_1-j\alpha_2-\Theta(\mu+\lambda-i\alpha_1-j\alpha_2)\rho)\nonumber\\
	&=\sum_{i=0}^a\sum_{j=0}^i\ch L(\mu+\lambda-i\alpha_1-j\alpha_2).\label{eq201}
\end{align}
In particular, $M$ is multiplicity-free.\\

If $a\geq m$, using line \eqref{eq400}, we have
	\begin{align}
	\ch M=&\sum_{i=0}^a\sum_{j=0}^i\chi(\mu+\lambda-i\alpha_1-j\alpha_2)-\sum_{i=m}^{a+m}\sum_{j=m}^i\chi(\mu+\lambda-i\alpha_1-j\alpha_2)\nonumber\\
	=&\sum_{i=0}^{m-1}\sum_{j=0}^i\chi(\mu+\lambda-i\alpha_1-j\alpha_2)+\sum_{i=m}^a\sum_{j=0}^{m-1}\chi(\mu+\lambda-i\alpha_1-j\alpha_2)\nonumber\\
	&\hspace{1em}-\sum_{i=a+1}^{a+m}\sum_{j=m}^{i}\chi(\mu+\lambda-i\alpha_1-j\alpha_2)\nonumber\\
	=&\sum_{i=0}^{m-1}\sum_{j=0}^i\chi(\mu+\lambda-i\alpha_1-j\alpha_2)+\sum_{i=m}^a\sum_{j=0}^{m-1}\chi(\mu+\lambda-i\alpha_1-j\alpha_2)\nonumber\\
	&\hspace{1em}-\sum_{j=a+1}^{a+m}\sum_{i=j}^{a+m}\chi(\mu+\lambda-i\alpha_1-j\alpha_2)-\sum_{j=m}^{a}\sum_{i=a+1}^{a+m}\chi(\mu+\lambda-i\alpha_1-j\alpha_2)\nonumber\\
	=&\sum_{i=0}^{m-1}\sum_{j=0}^i\chi(\mu+\lambda-i\alpha_1-j\alpha_2)-\sum_{i=0}^{m-1}\sum_{j=i}^{m-1}\chi(\mu+\lambda-j\alpha_1-i\alpha_2-(a+1)\rho)\label{eq101}\\
	&\hspace{1em}+\sum_{i=m}^a\sum_{j=0}^{m-1}\chi(\mu+\lambda-i\alpha_1-j\alpha_2)-\sum_{i=m}^{a}\sum_{j=0}^{m-1}\chi(\mu+\lambda-(j+a+1)\alpha_1-i\alpha_2).\label{eq102}
\end{align}
Let us study line \eqref{eq101}. We have
\begin{align*}
	&\sum_{i=0}^{m-1}\sum_{j=0}^i\chi(\mu+\lambda-i\alpha_1-j\alpha_2)-\sum_{i=0}^{m-1}\sum_{j=i}^{m-1}\chi(\mu+\lambda-j\alpha_1-i\alpha_2-(a+1)\rho)\\
	&=\sum_{i=0}^{m-1}\sum_{j=0}^i\chi(\mu+\lambda-i\alpha_1-j\alpha_2)\\
	&\hspace{3em}-\sum_{i=0}^{m-1}\sum_{j=0}^{i}\chi(\mu+\lambda-(m-1-j)\alpha_1-(m-1-i)\alpha_2-(a+1)\rho)\\
	&=\sum_{i=0}^{m-1}\sum_{j=0}^i\chi(\mu+\lambda-i\alpha_1-j\alpha_2)-\chi(\mu+\lambda-i\alpha_1-j\alpha_2-(a+m-i-j)\rho)\\
	&=\sum_{i=0}^{m-1}\sum_{j=0}^i\chi(\mu+\lambda-i\alpha_1-j\alpha_2)-\chi(\mu+\lambda-i\alpha_1-j\alpha_2-\Theta(\mu+\lambda-i\alpha_1-j\alpha_2)\rho),
\end{align*}
where in the last equality, we use line \eqref{acd_theta}. Now we can use line \eqref{acd_secondalcove} and Remark \ref{characters_alcove2} to conclude that
\begin{align}
	&\sum_{i=0}^{m-1}\sum_{j=0}^i\chi(\mu+\lambda-i\alpha_1-j\alpha_2)-\chi(\mu+\lambda-i\alpha_1-j\alpha_2-\Theta(\mu+\lambda-i\alpha_1-j\alpha_2)\rho)\nonumber\\
	&=\sum_{i=0}^{m-1}\sum_{j=0}^i\ch L(\mu+\lambda-i\alpha_1-j\alpha_2).\label{eq103}
\end{align}

Now we study line \eqref{eq102}. We have
\begin{align*}
	&\sum_{i=m}^a\sum_{j=0}^{m-1}\chi(\mu+\lambda-i\alpha_1-j\alpha_2)-\sum_{i=m}^{a}\sum_{j=0}^{m-1}\chi(\mu+\lambda-(j+a+1)\alpha_1-i\alpha_2)\\
	&=\sum_{i=m}^a\sum_{j=0}^{m-1}\chi(\mu+\lambda-i\alpha_1-j\alpha_2)-\sum_{i=m}^{a}\sum_{j=0}^{m-1}\chi(\mu+\lambda-((m-1-j)+a+1)\alpha_1-i\alpha_2)\\
	&=\sum_{i=0}^{a-m}\sum_{j=0}^{m-1}\chi(\mu+\lambda-(i+m)\alpha_1-j\alpha_2)-\sum_{i=0}^{a-m}\sum_{j=0}^{m-1}\chi(\mu+\lambda-(m+a-j)\alpha_1-(i+m)\alpha_2)\\
	&=\sum_{i=0}^{a-m}\sum_{j=0}^{m-1}\chi(\mu+\lambda-(i+m)\alpha_1-j\alpha_2)\\
	&\qquad\qquad-\sum_{i=0}^{a-m}\sum_{j=0}^{m-1}\chi(\mu+\lambda-(m+a-j)\alpha_1-((a-m-i)+m)\alpha_2)\\
	&=\sum_{i=0}^{a-m}\sum_{j=0}^{m-1}\chi(\mu+\lambda-(i+m)\alpha_1-j\alpha_2)-\sum_{i=0}^{a-m}\sum_{j=0}^{m-1}\chi(\mu+\lambda-(m+a-j)\alpha_1-(a-i)\alpha_2)\\
	&=\sum_{i=0}^{a-m}\sum_{j=0}^{m-1}\chi(\mu+\lambda-(i+m)\alpha_1-j\alpha_2)-\chi(\mu+\lambda-(m+a-j)\alpha_1-(a-i)\alpha_2)\\
	&=\sum_{i=0}^{a-m}\sum_{j=0}^{m-1}\chi(\mu+\lambda-(i+m)\alpha_1-j\alpha_2)-\chi(\mu+\lambda-(i+m)\alpha_1-j\alpha_2-(a-i-j)\rho)\\
	&=\sum_{i=0}^{a-m}\sum_{j=0}^{m-1}\chi(\mu+\lambda-(i+m)\alpha_1-j\alpha_2)\\
	&\qquad\qquad-\chi(\mu+\lambda-(i+m)\alpha_1-j\alpha_2-\Theta(\mu+\lambda-(i+m)\alpha_1-j\alpha_2)\rho),
\end{align*}
where in the last equality, we use line \eqref{acd_theta}.
For $0\leq i\leq a-m,\; 0\leq j\leq m-1$, we have
$$\Theta(\mu+\lambda-(i+m)\alpha_1-j\alpha_2)=a-i-j\geq 1.$$
Therefore, using line \eqref{acd_prestricted} and Lemma \ref{C2_theta}, we deduce that $\mu+\lambda-(i+m)\alpha_1-j\alpha_2\in \widehat{C_2}$ for all $0\leq i\leq a-m,\; 0\leq j\leq m-1$.
By Remark \ref{characters_alcove2}, we conclude that 
\begin{align}
	&\sum_{i=0}^{a-m}\sum_{j=0}^{m-1}\chi(\mu+\lambda-(i+m)\alpha_1-j\alpha_2)\nonumber\\
	&\hspace{4em}-\chi(\mu+\lambda-(i+m)\alpha_1-j\alpha_2-\Theta(\mu+\lambda-(i+m)\alpha_1-j\alpha_2)\rho)\nonumber\\
	&=\sum_{i=0}^{a-m}\sum_{j=0}^{m-1}\ch L(\mu+\lambda-(i+m)\alpha_1-j\alpha_2).\label{eq104}
\end{align}

We conclude by combining lines \eqref{eq103} and \eqref{eq104} to obtain
\begin{align}
	\ch M& =\sum_{i=0}^{m-1}\sum_{j=0}^i\ch L(\mu+\lambda-i\alpha_1-j\alpha_2)+\sum_{i=0}^{a-m}\sum_{j=0}^{m-1}\ch L(\mu+\lambda-(i+m)\alpha_1-j\alpha_2)\nonumber\\
	&=\sum_{i=0}^{m-1}\sum_{j=0}^i\ch L(\mu+\lambda-i\alpha_1-j\alpha_2)+\sum_{i=m}^{a}\sum_{j=0}^{m-1}\ch L(\mu+\lambda-i\alpha_1-j\alpha_2).\label{eq202}
\end{align}
In  particular, $M$ is multiplicity-free.
\end{proof}

\begin{rem}
	By duality, for $\lambda=(0,b),\mu=(c,d)\in X^+$ with $0<b,c,d<p$, the tensor product $L(\lambda)\otimes L(\mu)$ is multiplicity-free if and only if one of the following holds:
	\begin{enumerate}[label = (\arabic*)] 
		\item $b+c+d<p-1$,
		\item $c+d=p-1$, $b+d<p$ and $b<d+2$ or
		\item $c+d>p-1$, $b+c<p$ and $b+d<p$. 
	\end{enumerate}
\end{rem}

At this step, let us make the following observation, which will be useful in the next proposition.

\begin{cor}\label{character_acd}
	Let $\lambda=(a,0),\mu=(c,d)\in X^+$ be $p$-restricted with $a,c,d\neq 0$, $c+d\geq p-1$, $a+c < p$ and $a+d<p$ (in particular $\mu\in C_2$). Let $m:=\Theta(\mu)$. Then 
	$$\ch(L(\lambda)\otimes L(\mu))=\sum_{i=0}^a\sum_{j=0}^i\ch L(\lambda+\mu-i\alpha_1-j\alpha_2)-\sum_{i=m}^a\sum_{j=m}^i\ch L(\lambda+\mu-i\alpha_1-j\alpha_2).$$
\end{cor}

\begin{proof}\
	\begin{itemize}
		\item If $c+d>p-1$ and $a<m$, then the second double sum is empty and we find the same result as in line \eqref{eq201} (in the proof of Proposition \ref{acd_p}).
		\item If $c+d>p-1$ and $a\geq m$, the result is a consequence of line \eqref{eq202} in the proof of Proposition \ref{acd_p} since we have
		\begin{align*}
			\ch(L(\lambda)\otimes L(\mu))&=\sum_{i=0}^{m-1}\sum_{j=0}^i\ch L(\mu+\lambda-i\alpha_1-j\alpha_2)+\sum_{i=m}^{a}\sum_{j=0}^{m-1}\ch L(\mu+\lambda-i\alpha_1-j\alpha_2)\\
			&=\sum_{i=0}^a\sum_{j=0}^i\ch L(\lambda+\mu-i\alpha_1-j\alpha_2)-\sum_{i=m}^a\sum_{j=m}^i\ch L(\lambda+\mu-i\alpha_1-j\alpha_2).
		\end{align*}
	\item If $c+d=p-1$, the condition $a+d<p$ is equivalent to $a\leq c$. Moreover, in this case $m=\Theta(\mu)=1$. Therefore, by line \eqref{eq203} in the proof of Proposition \ref{acd_p-1}, we get
	\begin{align*}
		\ch(L(\lambda)\otimes L(\mu))&=\sum_{i=0}^a\ch L(\mu+\lambda-i\alpha_1)\\
		&=\sum_{i=0}^a\sum_{j=0}^i\ch L(\lambda+\mu-i\alpha_1-j\alpha_2)-\sum_{i=1}^a\sum_{j=1}^i\ch L(\lambda+\mu-i\alpha_1-j\alpha_2)\\
		&=\sum_{i=0}^a\sum_{j=0}^i\ch L(\lambda+\mu-i\alpha_1-j\alpha_2)-\sum_{i=m}^a\sum_{j=m}^i\ch L(\lambda+\mu-i\alpha_1-j\alpha_2).
	\end{align*}
	\end{itemize}
Hence in all cases we are done.
\end{proof}

\subsubsection{$L(a,b)\otimes L(c,d)$}
\begin{lemma}\label{diff_tensor}
	Let $\lambda=(a,b)\in X^+$ be $p$-restricted with $a,b\neq 0$ and $a+b<p-1$ (i.e. $\lambda\in \widehat{C_1}$). Then $$\ch(L(a,0)\otimes L(0,b))-\ch(L(a-1,0)\otimes L(0,b-1))=\ch L(a,b).$$
\end{lemma}

\begin{proof}
	First, note that $(a-1,0)+(0,b-1)=\lambda-\rho$. By Lemma \ref{Weyl_A2}, we have \linebreak$L(a-1,0)\cong\Delta(a-1,0)$, $L(a,0)\cong \Delta(a,0)$, $L(0,b-1)\cong\Delta(0,b-1)$, $L(0,b)\cong\Delta(0,b)$ and $L(\lambda)\cong\Delta(\lambda)$. By Proposition \ref{product_characters} and Lemma \ref{weights_a0}, and using Lemma \ref{action_characters} in the seventh equality below, we get
	\begin{align*}
		&\ch(L(a,0)\otimes L(0,b))-\ch(L(a-1,0)\otimes L(0,b-1))=\chi(a,0)\chi(0,b)-\chi(a-1,0)\chi(0,b-1)\\
		&\qquad\qquad=\sum_{i=0}^a\sum_{j=0}^i\chi(\lambda-i\alpha_1-j\alpha_2)-\sum_{i=0}^{a-1}\sum_{j=0}^i\chi(\lambda-\rho-i\alpha_1-j\alpha_2)\\
		&\qquad\qquad=\sum_{i=0}^a\sum_{j=0}^i\chi(\lambda-i\alpha_1-j\alpha_2)-\sum_{i=1}^a\sum_{j=1}^i\chi(\lambda-i\alpha_1-j\alpha_2)\\
		&\qquad\qquad=\sum_{i=0}^a\chi(\lambda-i\alpha_1)\\
		&\qquad\qquad=\chi(\lambda)+\sum_{i=1}^a\chi(\lambda-i\alpha_1)\\
		&\qquad\qquad=\chi(\lambda)+\frac12\sum_{i=1}^a\chi(\lambda-i\alpha_1)+\chi(\lambda-(a+1-i)\alpha_1)\\
		&\qquad\qquad=\chi(\lambda)+\frac12\sum_{i=1}^a\chi(\lambda-i\alpha_1)-\chi(s_{\alpha_1}\sbullet(\lambda-(a+1-i)\alpha_1))\\
		&\qquad\qquad=\chi(\lambda)+\frac12\sum_{i=1}^a\chi(\lambda-i\alpha_1)-\chi(\lambda-i\alpha_1)\\
		&\qquad\qquad=\chi(\lambda)\\
		&\qquad\qquad=\ch L(a,b).\qedhere
	\end{align*}
\end{proof}

\begin{prop}\label{abcd}
	Let $\lambda=(a,b),\;\mu=(c,d)\in X^+$ be two $p$-restricted weights with \linebreak$a,b,c,d\neq 0$. 
	Then $L(\lambda)\otimes L(\mu)$ is multiplicity-free if and only if one of the following holds:
	\begin{enumerate}[label = \emph{(\arabic*)}] 
		\item $a+b=p-1,\; c+d<p-1,\; a+c+d<p$ and $b+c+d<p$ or
		\item $a+b<p-1,\; c+d=p-1,\; a+b+c<p$ and $a+b+d<p$.
	\end{enumerate}
\end{prop}


\begin{proof}
	We set $M:=L(\lambda)\otimes L(\mu)$. \\
	
	If $a+c\geq p$ or $b+d\geq p$, then $M$ has multiplicity by Corollary \ref{sum_p_res}. Therefore, we can assume $a+c<p$ and $b+d<p$. In particular, it cannot happen that $a+b\geq p-1$ and $c+d>p-1$.\\
	
	Suppose that $a+b\neq p-1$ and $c+d\neq p-1$. Using Argument \ref{argument1}, we show that $L(\lambda+\mu-\alpha_1-\alpha_2)$ has multiplicity at least 2 in $M$. Using Lemmas \ref{-1} and \ref{-11}, we have
	\begin{align*}
		&m_M(\lambda+\mu)=1,
		&&m_M(\lambda+\mu-\alpha_1)=2,
		&&m_M(\lambda+\mu-\alpha_2)=2,\\
		&m_M(\lambda+\mu-\alpha_1-\alpha_2)=6.
	\end{align*}
	
	Using Lemmas \ref{-1} and \ref{-11} again, we have
	\begin{align*}
		&m_{L(\lambda+\mu)}(\lambda+\mu-\alpha_1)=1,&&m_{L(\lambda+\mu)}(\lambda+\mu-\alpha_2)=1.
	\end{align*}
	
	Therefore, $L(\lambda+\mu-\alpha_1)$ and $L(\lambda+\mu-\alpha_2)$ are composition factors of $M$.
	Using \linebreak Lemmas \ref{-1} and \ref{-11} a third time, we get 
	\begin{align*}
		&m_{L(\lambda+\mu)}(\lambda+\mu-\alpha_1-\alpha_2)\leq 2, &&m_{L(\lambda+\mu-\alpha_1)}(\lambda+\mu-\alpha_1-\alpha_2)=1,\\
		&m_{L(\lambda+\mu-\alpha_2)}(\lambda+\mu-\alpha_1-\alpha_2)=1.
	\end{align*}
	
	We conclude that $L(\lambda+\mu-\alpha_1-\alpha_2)$ is a composition factor with multiplicity 
	$$[M:L(\lambda+\mu-\alpha_1-\alpha_2)]\geq 6-1-1-2=2.$$ 
	Therefore, $M=L(\lambda)\otimes L(\mu)$ has multiplicity.\\
	
	Suppose that $a+b=c+d=p-1$. Since we assume $a+c<p$ and $b+d<p$, we only need to consider the case $\lambda+\mu=(p-1,p-1)$. Using the same reasoning as in the previous case, we deduce that $L(\lambda+\mu-\alpha_1)$ is a composition factor of $M$. But $\lambda+\mu-\alpha_1=(p-3,p)$ is not $p$-restricted. Therefore, by Corollary \ref{composition_p}, $M$ has multiplicity.\\ 
	
	Let us consider the case $a+b<p-1,\; c+d=p-1,\; a+b+d\geq p$ (the other remaining cases where we claim that $M$ has multiplicity are symmetric). Note that $b+c+d\geq p$, therefore  $b+d+\min(a,c)\geq p$.
	We show inductively that $[M:L(\lambda+\mu-s\alpha_1)]=1$ for $0\leq s \leq p-b-d\leq \min(a,c)$.
	For $0\leq s \leq \min(a,c)$, using Lemma \ref{-1}, we have:
	$$m_M(\lambda+\mu-s\alpha_1)=\sum_{i=0}^sm_{L(\lambda)}(\lambda-i\alpha_1)\cdot m_{L(\mu)}(\mu-(s-i)\alpha_1)=s+1.$$
	Moreover, if $i< p-b-d$, then $\lambda+\mu-i\alpha_1=(a+c-2i,b+d+i)$ is $p$-restricted and we have
	$$\sum_{i=0}^{s-1}m_{L(\lambda+\mu-i\alpha_1)}(\lambda+\mu-s\alpha_1)=s.$$
	
	We know that $[M:L(\lambda+\mu)]=1$, and combining the two previous equations, we \linebreak conclude inductively that $[M:L(\lambda+\mu-s\alpha_1)]=1$ for $0\leq s \leq p-b-d$.
	\linebreak In particular, $L(\lambda+\mu-(p-b-d)\alpha_1)$ is a composition factor of $M$. But \linebreak$\lambda+\mu-(p-b-d)\alpha_1=(a+c-2(p-b-d),p)$ is not $p$-restricted. We conclude by Corollary \ref{composition_p} that $M$ has multiplicity.\\

Now we consider the cases where we claim that $M$ is multiplicity-free. Up to the reordering of the weights and up to symmetry, we can suppose $a+b<p-1, \;c+d=p-1$ and $c\geq d$. Thus the conditions in the statement of the proposition are equivalent to a single one: $a+b+c<p$. In particular we have $a<d\leq c$ and $b<d\leq c$.
We show that $M$ is multiplicity-free by showing the following equality of characters:
$$\ch M=\sum_{i=0}^a\sum_{j=0}^b\ch L(\lambda+\mu-i\alpha_1-j\alpha_2).$$

Let $0\leq i\leq a,\; 0\leq j\leq b$. We claim that $\lambda+\mu-i\alpha_1-j\alpha_2\in \widehat{C_2}$. We have
$$\lambda+\mu-i\alpha_1-j\alpha_2=(a+c-2i+j,b+d-2j+i)$$
with
\begin{ceqn}
	\begin{align*}
		&a+c-2i+j\leq a+c+b<p,\\
		&b+d-2j+i\leq b+d+a\leq a+b+c< p\qquad\text{and}\\
		&(a+c-2i+j)+(b+d-2j+i)=a+b+c+d-i-j\geq c+d =p-1.
	\end{align*}
\end{ceqn}
Thus $\lambda+\mu-i\alpha_1-j\alpha_2\in \widehat{C_2}$ for all $0\leq i\leq a,\; 0\leq j\leq b$.

By Lemma \ref{diff_tensor}, we have
\begin{align*}
	\ch M&= \ch L(a,b)\ch L(c,d)\\
	&=\left(\ch(L(a,0)\otimes L(0,b))-\ch(L(a-1,0)\otimes L(0,b-1))\right)\ch L(c,d)\\
	&=\ch(L(a,0)\otimes L(0,b)\otimes L(c,d))-\ch(L(a-1,0)\otimes L(0,b-1)\otimes L(c,d))\\
	&=\ch L(a,0)\ch(L(0,b)\otimes L(c,d))-\ch L(a-1,0)\ch(L(0,b-1)\otimes L(c,d)).
\end{align*}
Since $c+d=p-1,\;b+d<p$ and $b\leq d$, we can use line \eqref{eq203} in the proof of Lemma \ref{acd_p-1}\linebreak to express $\ch(L(0,b)\otimes L(c,d))$ and $\ch(L(0,b-1)\otimes L(c,d))$. We set $\lambda_b:=(0,b)$ and \linebreak$\lambda_{b-1}:=(0,b-1)$. We get
$$\ch(L(0,b)\otimes L(c,d))=\sum_{k=0}^b\ch L(\lambda_b+\mu-k\alpha_2),$$
and 
$$\ch(L(0,b-1)\otimes L(c,d))=\sum_{k=0}^{b-1}\ch L(\lambda_{b-1}+\mu-k\alpha_2).$$
Observe that $\lambda_b+\mu-k\alpha_2=(c+k,b+d-2k)$. For $0\leq k\leq b$, we have 
\begin{ceqn}
	\begin{align}
		&0<c\leq c+k\leq c+b<p-1,\nonumber\\
		 &0<d-b\leq b+d-2k\leq b+d<p-1 \qquad \text{and}\nonumber\\ 
		 &(c+k)+(b+d-2k)\geq c+d=p-1.\label{eq700}
	\end{align}

Therefore $\lambda_b+\mu-k\alpha_2\in C_2$ for all $k\in \{0,\ldots,b\}$.
Moreover, we have
\begin{align}
	&a+c+k\leq a+b+c<p\qquad\text{and}\label{eq701}\\
	&a+b+d-2k\leq a+b+d\leq a+b+c<p\label{eq702}.
\end{align}
Similarly, for all $0\leq k\leq b-1$, we have $\lambda_{b-1}+\mu-k\alpha_2=(c+k,b-1+d-2k)\in C_2$ and
\begin{align}
	&(c+k)+(b-1+d-2k)\geq p-1, \label{eq703}\\
	&(a-1)+c+k\leq a+b+c-2<p\qquad\text{and}\label{eq704}\\
	&(a-1)+b-1+d-2k\leq a+b+d-2\leq a+b+c-2<p\label{eq705}.
\end{align}
\end{ceqn}


Let $m:=\Theta(\lambda_b+\mu)=b+c+d+2-p=b+1$. Observe that $\Theta(\lambda_b+\mu-k\alpha_2)=m-k$ and $\Theta(\lambda_{b-1}+\mu-k\alpha_2)=m-k-1$. Lines \eqref{eq700}, \eqref{eq701}, \eqref{eq702}, \eqref{eq703}, \eqref{eq704} and \eqref{eq705} allow us to use Corollary \ref{character_acd} in the fourth equality below, to get
\begin{align*}
	\ch M=&\ch L(a,0)\ch(L(0,b)\otimes L(c,d))-\ch L(a-1,0)\ch(L(0,b-1)\otimes L(c,d))\\
	=&\ch L(a,0)\left(\sum_{k=0}^b\ch L(\lambda_b+\mu-k\alpha_2)\right)-\ch L(a-1,0)\left(\sum_{k=0}^{b-1}\ch L(\lambda_{b-1}+\mu-k\alpha_2)\right)\\
	=&\left(\sum_{k=0}^b\ch L(a,0)\ch L(\lambda_b+\mu-k\alpha_2)\right)-\left(\sum_{k=0}^{b-1}\ch L(a-1,0)\ch L(\lambda_{b-1}+\mu-k\alpha_2)\right)\\
	=&\sum_{k=0}^b\sum_{i=0}^a\sum_{j=0}^i\ch L(\lambda+\mu-i\alpha_1-(j+k)\alpha_2)\\
	&\qquad-\sum_{k=0}^b\sum_{i=m-k}^a\sum_{j=m-k}^i\ch L(\lambda+\mu-i\alpha_1-(j+k)\alpha_2)\\
	&\qquad-\sum_{k=0}^{b-1}\sum_{i=0}^{a-1}\sum_{j=0}^i\ch L(\lambda+\mu-\rho-i\alpha_1-(j+k)\alpha_2)\\
	&\qquad+\sum_{k=0}^{b-1}\sum_{i=m-k-1}^{a-1}\sum_{j=m-k-1}^i\ch L(\lambda+\mu-\rho-i\alpha_1-(j+k)\alpha_2)\\
	=&\sum_{k=0}^b\sum_{i=0}^a\sum_{j=0}^i\ch L(\lambda+\mu-i\alpha_1-(j+k)\alpha_2)\\
	&\qquad-\sum_{k=0}^b\sum_{i=m-k}^a\sum_{j=m-k}^i\ch L(\lambda+\mu-i\alpha_1-(j+k)\alpha_2)\\
	&\qquad-\sum_{k=0}^{b-1}\sum_{i=1}^{a}\sum_{j=1}^i\ch L(\lambda+\mu-i\alpha_1-(j+k)\alpha_2)\\
	&\qquad+\sum_{k=0}^{b-1}\sum_{i=m-k}^{a}\sum_{j=m-k}^i\ch L(\lambda+\mu-i\alpha_1-(j+k)\alpha_2).
\end{align*}
To simplify the notation, we set $Z(i,j):=\ch L(\lambda+\mu-i\alpha_1-j\alpha_2)$. Thus we get
\begin{align*}
	\ch M=&\sum_{k=0}^b\sum_{i=0}^a\sum_{j=0}^i Z(i,j+k)-\sum_{k=0}^b\sum_{i=m-k}^a\sum_{j=m-k}^i Z(i,j+k)\\
	&\hspace{3em}-\sum_{k=0}^{b-1}\sum_{i=1}^{a}\sum_{j=1}^iZ(i,j+k)+\sum_{k=0}^{b-1}\sum_{i=m-k}^{a}\sum_{j=m-k}^iZ(i,j+k)\\
	=&\sum_{k=0}^{b-1}Z(0,k)+\sum_{k=0}^{b-1}\sum_{i=1}^aZ(i,k)+\sum_{i=0}^a\sum_{j=0}^iZ(i,j+b)-\sum_{i=m-b}^a\sum_{j=m-b}^i Z(i,j+b)\displaybreak\\
	=&\sum_{k=0}^{b-1}\sum_{i=0}^aZ(i,k)+\sum_{i=0}^aZ(i,b)+\sum_{i=1}^a\sum_{j=1}^iZ(i,j+b)-\sum_{i=m-b}^a\sum_{j=m-b}^i Z(i,j+b)\\
	=&\sum_{k=0}^{b}\sum_{i=0}^aZ(i,k)+\sum_{i=1}^a\sum_{j=1}^iZ(i,j+b)-\sum_{i=m-b}^a\sum_{j=m-b}^i Z(i,j+b).
\end{align*}

Finally, observe that $m-b=1$. This allows us to conclude that 
\begin{align}
	\ch M &=\sum_{k=0}^{b}\sum_{i=0}^aZ(i,k)+\sum_{i=1}^a\sum_{j=1}^iZ(i,j+b)-\sum_{i=1}^a\sum_{j=1}^i Z(i,j+b)\nonumber\\
	&=\sum_{k=0}^{b}\sum_{i=0}^aZ(i,k)\nonumber\\
	&=\sum_{k=0}^{b}\sum_{i=0}^a \ch L(\lambda+\mu -i\alpha_1-k\alpha_2).\label{eqabcd}
\end{align}
In particular, $M=L(\lambda)\otimes L(\mu)$ is multiplicity-free.
\end{proof}

Now Theorem \ref{SL3} follows from the previous sequence of propositions. 

\subsection{Decomposition of multiplicity-free tensor products}
Since multiplicity-free tensor products of simple modules are completely reducible \linebreak(Lemma \ref{complete_red_mult_free}), we can specify the structure of those modules. For some of the cases, we still need to compute the decomposition in characteristic zero, which we do in the next lemma.

\begin{lemma}\label{A2_acd_char0}
	Let $\lambda=(a,0),\mu=(c,d)\in X^+$. Then
	$$\chi(\lambda)\chi(\mu)=\sum_{i=0}^{a}\sum_{j=\max\{0,i-c\}}^{\min\{i,d\}}\chi(\lambda+\mu-i\alpha_1-j\alpha_2).$$
	Moreover, all weights appearing in the sum are dominant.
\end{lemma}

\begin{proof}
	By Proposition \ref{product_characters} and Lemma \ref{weights_a0}, we get
	\begin{align}
		\chi(\lambda)\chi(\mu)=&\sum_{i=0}^a\sum_{j=0}^i\chi(\lambda+\mu-i\alpha_1-j\alpha_2)\nonumber\\
		=&\sum_{i=0}^{a}\sum_{j=\max\{0,i-c\}}^{\min\{i,d\}}\chi(\lambda+\mu-i\alpha_1-j\alpha_2)\label{eqC}\\
		&\qquad+ \sum_{i=0}^a\sum_{j=0}^{\min\{i-c-1,d\}}\chi(\lambda+\mu-i\alpha_1-j\alpha_2)\label{eqA}\\
		&\qquad+\sum_{i=0}^a\sum_{j=d+1}^i\chi(\lambda+\mu-i\alpha_1-j\alpha_2)\label{eqB}.
	\end{align}
Observe that 
\begin{ceqn}
	\begin{align}
		\lambda+\mu-i\alpha_1-j\alpha_2=(a+c-2i+j,d+i-2j).\label{eq600}
	\end{align}
\end{ceqn}
In particular, for $i\leq a$ and $j\geq i-c$, we have
$$a+c-2i+j\geq a-i\geq 0,$$
and for $i\geq 0$ and $j\leq \min\{i,d\}$, we have
$$d+i-2j\geq (d-j)+(i-j)\geq 0.$$
Therefore, all weights appearing in line \ref{eqC} are dominant.\\

We show that lines \eqref{eqA} and \eqref{eqB} are equal to zero.
We start by line \eqref{eqB}. Using \linebreak Lemma \ref{action_characters}, and line \eqref{eq600} in the second equality below, we have
\begin{align*}
	\sum_{j=d+1}^i\chi(\lambda+\mu-i\alpha_1-j\alpha_2)&=-\sum_{j=d+1}^i\chi(s_{\alpha_2}\sbullet(\lambda+\mu-i\alpha_1-j\alpha_2))\\
	&=-\sum_{j=d+1}^i\chi(\lambda+\mu-i\alpha_1-j\alpha_2-(d+i-2j+1)\alpha_2)\\
	&=-\sum_{j=d+1}^i\chi(\lambda+\mu-i\alpha_1-(d+i-j+1)\alpha_2)\\
	&=-\sum_{j=d+1}^i\chi(\lambda+\mu-i\alpha_1-(d+i-(d+1+i-j)+1)\alpha_2)\\
	&=-\sum_{j=d+1}^i\chi(\lambda+\mu-i\alpha_1-j\alpha_2).
\end{align*}
Therefore, 
$$\sum_{j=d+1}^i\chi(\lambda+\mu-i\alpha_1-j\alpha_2)=0$$
and 
$$\sum_{i=0}^a\sum_{j=d+1}^i\chi(\lambda+\mu-i\alpha_1-j\alpha_2)=0.$$

Now we work on line \eqref{eqA}. If $i<c+1$, then $i-c-1<0$. Thus we have
\begin{align}
	\sum_{i=0}^a\sum_{j=0}^{\min\{i-c-1,d\}}\chi(\lambda+\mu-i\alpha_1-j\alpha_2)&=\sum_{i=c+1}^a\sum_{j=0}^{\min\{i-c-1,d\}}\chi(\lambda+\mu-i\alpha_1-j\alpha_2)\nonumber\\
	&=\sum_{j=0}^{\min\{a-c-1,d\}}\sum_{i=c+1+j}^a\chi(\lambda+\mu-i\alpha_1-j\alpha_2)\label{eqD}.
\end{align}
We fix $j$ and work on the second sum in line \eqref{eqD}. Using Lemma \ref{action_characters}, and line \eqref{eq600} in the second equality below, we get
\begin{align*}
	&\sum_{i=c+1+j}^a\chi(\lambda+\mu-i\alpha_1-j\alpha_2)\\
	&\hspace{6em}=-\sum_{i=c+1+j}^a\chi(s_{\alpha_1}\sbullet(\lambda+\mu-i\alpha_1-j\alpha_2))\displaybreak\\
	&\hspace{6em}=-\sum_{i=c+1+j}^a\chi(\lambda+\mu-i\alpha_1-j\alpha_2-(a+c-2i+j+1)\alpha_1)\\
	&\hspace{6em}=-\sum_{i=c+1+j}^a\chi(\lambda+\mu-(a+c-i+j+1)\alpha_1-j\alpha_2)\\
	&\hspace{6em}=-\sum_{i=c+1+j}^a\chi(\lambda+\mu-(a+c-(a+c+1+j-i)+j+1)\alpha_1-j\alpha_2)\\
	&\hspace{6em}=-\sum_{i=c+1+j}^a\chi(\lambda+\mu-i\alpha_1-j\alpha_2).
\end{align*}
Therefore, 
$$\sum_{i=c+1+j}^a\chi(\lambda+\mu-i\alpha_1-j\alpha_2)=0$$
so
$$\sum_{j=0}^{\min\{a-c-1,d\}}\sum_{i=c+1+j}^a\chi(\lambda+\mu-i\alpha_1-j\alpha_2)=0$$
and 
\begin{ceqn}
	\begin{equation*}
		\sum_{i=0}^a\sum_{j=0}^{\min\{i-c-1,d\}}\chi(\lambda+\mu-i\alpha_1-j\alpha_2)=0.\qedhere
	\end{equation*}
\end{ceqn}
\end{proof}

\begin{rem}
	Let $(a,b),\;(c,d)\in X^+$ be dominant weights. Lemma \ref{A2_acd_char0} shows that \linebreak$L_{\C}(\lambda)\otimes L_{\C}(\mu)$ is multiplicity-free if $a\cdot b\cdot c\cdot d=0$. Moreover, using Argument \ref{argument1} like in the proof of Proposition \ref{abcd}, one can prove that $[L_{\C}(\lambda)\otimes L_{\C}(\mu):L_{\C}(\lambda+\mu-\alpha_1-\alpha_2)]=2$ if $a\cdot b\cdot c\cdot d\neq 0$. These two facts provide a proof of Theorem \ref{char_0_A2}.
\end{rem}

\begin{rem}
	The computations in Lemma \ref{A2_acd_char0} can be done using the Littlewood-Richardson rule, see \cite[Theorem III]{Littlewood} or \cite[Proposition 15.25]{Fulton_Harris} for this specific case.
\end{rem}

\begin{cor}
	Let $\lambda=(a,b),\mu=(c,d) \in X^+$ be non-zero and $p$-restricted such that $L(\lambda)\otimes L(\mu)$ is multiplicity-free. We have the following decompositions:
		\begin{enumerate}[label = \emph{(\arabic*)}] 
		\item If $b=d=0$, $a+c<p$ and $a\leq c$, then
		$$L(\lambda)\otimes L(\mu)\cong \bigoplus_{i=0}^a L(\mu+\lambda-i\alpha_1).$$
		\item If $b=c=0$, $a\leq d$ and $a+d<p-1$ or $(a,d)=(1,p-1)$, then
		$$L(\lambda)\otimes L(\mu)\cong \bigoplus_{i=0}^a L(\mu+\lambda-i\rho).$$
		\item If $b=0$ and $a+c+d<p-1$, then
		$$L(\lambda)\otimes L(\mu)\cong \bigoplus_{i=0}^a\bigoplus_{j=\max\{0,i-c\}}^{\min\{i,d\}} L(\mu+\lambda-i\alpha_1-j\alpha_2).$$
		\item If $b=0$, $c+d=p-1$, $a+c<p$ and $a\leq c$, then
		$$L(\lambda)\otimes L(\mu)\cong \bigoplus_{i=0}^a L(\mu+\lambda-i\alpha_1).$$
		\item If $b=0$, $c+d=p-1$, $a+c<p$ and $a=c+1$, then
		$$L(\lambda)\otimes L(\mu)\cong L(\lambda+\mu-a\alpha_1-\alpha_2)\oplus\bigoplus_{i=0}^c L(\mu+\lambda-i\alpha_1).$$
		\item If $b=0$, $c+d>p-1$, $a+c<p$ and $a+d<p$, then
		$$L(\lambda)\otimes L(\mu)\cong \bigoplus_{i=0}^{a}\bigoplus_{j=0}^{\min\{i,\Theta(\mu)-1\}} L(\mu+\lambda-i\alpha_1-j\alpha_2).$$
		\item If $a\cdot b\cdot c\cdot d\neq 0$, $a+b<p-1$, $c+d=p-1$, $a+b+c<p$ and $a+b+d<p$, then
		$$L(\lambda)\otimes L(\mu)\cong \bigoplus_{k=0}^{b}\bigoplus_{i=0}^a  L(\lambda+\mu -i\alpha_1-k\alpha_2).$$
	\end{enumerate}
\end{cor}

\begin{proof}\
	\begin{enumerate}[label = (\arabic*)] 
		\item This follows from Lemma \ref{A2_acd_char0} and Proposition \ref{ac}.
		\item This follows from Lemma \ref{A2_acd_char0} and Proposition \ref{ad}.
		\item This follows from Lemma \ref{A2_acd_char0} and Proposition \ref{acd_1}.
		\item This follows from line \eqref{eq203} in the proof of Proposition \ref{acd_p-1}.
		\item This follows from line \eqref{eq230} in the proof of Proposition \ref{acd_p-1}.
		\item This follows from lines \eqref{eq201} and \eqref{eq202} in the proof of Proposition \ref{acd_p}.
		\item This follows from line \eqref{eqabcd} in the proof of Proposition \ref{abcd}.\qedhere
	\end{enumerate}
\end{proof}


\cleardoublepage

\section{$\Sp_4$}
In this section, we establish a number of results for the classification of multiplicity-free tensor products of simple $\Sp_4$-modules with $p$-restricted highest weight.
We fix $G=\Sp_4(k)$ with root system $\Phi$ of type $B_2=C_2$ and $\Pi=\{\alpha_1,\alpha_2\}$ a base of $\Phi$ with $\frac{(\alpha_1,\alpha_1)}{(\alpha_2,\alpha_2)}=2$. With respect to this base, we have $\rho =\frac32\alpha_1+2\alpha_2$.
We assume $p\geq 5$, so that there exist weights inside the alcoves. Since the Coxeter number of a root system $\Phi$ of type $B_2$ is $h=4$, we have in particular $p\geq h$.
For $\lambda=x\omega_1+y\omega_2\in X$, we write $\lambda=(x,y)$. In particular, we have $\alpha_1=(2,-2),\;\alpha_2=(-1,2)$ and $\rho=(1,1)$.

In this section, we will also use Euclidean coordinates several times. We fix an orthogonal basis $(\epsilon_1,\epsilon_2)$ of $\R^2$ with $\epsilon_1=\omega_1=\alpha_1+\alpha_2$ and $\epsilon_2=\alpha_2=2\omega_2-\omega_1$. With respect to this basis, we have $\alpha_1=(1,-1),\;\alpha_2=(0,1),\;\alpha_1+\alpha_2=(1,0),\;\alpha_1+2\alpha_2=(1,1)$ and $\rho=(\frac 32, \frac 12)$.
	Since the notation $(\;,\;)$ might be confusing, we will always explicitly state when we use Euclidean coordinates. If it is not mentioned, then it means that we use coordinates with respect to the fundamental weights basis. In particular, coordinates with respect to the fundamental weights are used in all statements of propositions in the classification (section \ref{B2_classification}).

\subsection{Alcoves}
In this subsection, we describe the four $p$-restricted alcoves of a root system of type $B_2$. We start by defining a numeration of those alcoves.

\begin{defn}
	We set
	
	\begin{align*}
		&C_1:=\{\lambda\in X_{\R}|\; (\lambda+\rho,\alpha_1^\vee)>0, (\lambda+\rho,\alpha_2^\vee)>0,(\lambda+\rho,(\alpha_1+\alpha_2)^\vee)<p\},\\
		&C_2:=\{\lambda\in X_{\R}|\; (\lambda+\rho,\alpha_2^\vee)>0, (\lambda+\rho,(\alpha_1+\alpha_2)^\vee)>p,(\lambda+\rho,(\alpha_1+2\alpha_2)^\vee)<p\},\\
		&C_3:=\{\lambda\in X_{\R}|\; (\lambda+\rho,\alpha_2^\vee)<p, (\lambda+\rho,(\alpha_1+\alpha_2)^\vee)<2p,(\lambda+\rho,(\alpha_1+2\alpha_2)^\vee)>p\},\\
		&C_4:=\{\lambda\in X_{\R}|\; (\lambda+\rho,\alpha_1^\vee)<p, (\lambda+\rho,\alpha_2^\vee)<p,(\lambda+\rho,(\alpha_1+\alpha_2)^\vee)>2p\}.
	\end{align*}
We called $C_1$ the \emph{fundamental alcove}, $C_2$ the \emph{second alcove}, $C_3$ the \emph{third alcove} and $C_4$ the \emph{fourth alcove}.
	
	We also set 
	\begin{ceqn}
	
	\begin{align*}
		F_{i,i+1}:=\overline{C_{i}}\cap \overline{C_{i+1}}
	\end{align*}
	for $i=1,2,3$, i.e. $F_{i,i+1}$ is the wall between the alcove $C_i$ and $C_{i+1}$.
	Finally, we set
	 \begin{align*}
		&F_{3,5}:=\{\lambda\in \overline{C_3}|\; (\lambda+\rho,\alpha_2^\vee)=p\},\\
		&F_{4,6}:=\{\lambda\in \overline{C_4}|\; (\lambda+\rho,\alpha_1^\vee)=p\} \text{ and }\\
		&F_{4,7}:=\{\lambda\in \overline{C_4}|\; (\lambda+\rho,\alpha_2^\vee)=p\}.
	\end{align*}
\end{ceqn}

\end{defn}
We set $C_5:=s_{\alpha_2,p}\sbullet C_3$, $C_6:=s_{\alpha_1,p}\sbullet C_4$ and $C_7:=s_{\alpha_2,p}\sbullet C_4$. Then $F_{i,j}=\overline{C_{i}}\cap \overline{C_{j}}$ for $(i,j)\in \{(3,5),(4,6),(4,7)\}$.\\

Using coordinates with respect to the fundamental weights, we have
\begin{align*}
	&\widehat{C_1}\cap X^+=\{(a,b)\in \N^2| \; 2a+b\leq p-3\},\\
	&\widehat{C_2}\cap X^+=\{(a,b)\in \N^2| \; 2a+b> p-3,\; a+b\leq p-2\},\\
	&\widehat{C_3}\cap X^+=\{(a,b)\in \N^2| \; a+b> p-2,\; 2a+b\leq 2p-3,\; b\leq p-1\},\\
	&\widehat{C_4}\cap X^+=\{(a,b)\in \N^2| \; 2a+b> 2p-3,\;a\leq p-1,\; b\leq p-1\}.
\end{align*}

Using Euclidean coordinates, we have
\begingroup
\addtolength{\jot}{-0.49em}
\begin{align*}
	&X=\{(a,b)\in (\tfrac12\Z)^2|\; a+b\in \Z\},\\
	&\widehat{C_1}\cap X^+=\{(a,b)\in (\tfrac12\N)^2|\; a+b\in \N,\; a\leq \frac{p-3}{2},\; b\leq a\},\\
	&\widehat{C_2}\cap X^+=\{(a,b)\in (\tfrac12\N)^2|\; a+b\in \N,\; a>\frac{p-3}{2},\; a+b\leq p-2\},\\
	&\widehat{C_3}\cap X^+=\{(a,b)\in (\tfrac12\N)^2|\; a+b\in \N, \; a+b>p-2,\; a\leq p-\frac32,\; b\leq \frac{p-1}{2}\},\\
	&\widehat{C_4}\cap X^+=\{(a,b)\in (\tfrac12\N)^2|\; a+b\in \N,\; a> p-\frac32,\; b\leq \frac{p-1}{2},\; a-b\leq p-1\}.
\end{align*}
\endgroup
Moreover, a weight $(a,b)\in X$ is dominant if and only if $0\leq b\leq a$, and $(a,b)\in D$ if and only if $-\frac{1}{2}\leq b\leq a+1$.\\

Let us illustrate these alcoves with a picture. The blue arrows form the root system. The black arrows are the fundamental weights, generating the weight lattice (in black). The region containing the dominant weights is coloured in green. The red triangles are the walls of the alcoves.

\begin{center}
\begin{tikzpicture}[framed,baseline={(current bounding box.center)},scale=0.8]
	\clip (-3,-5) rectangle (7,14);
	
	\pgftransformcm{cos(45)}{sin(45)}{sqrt(2)*cos(90)}{sqrt(2)*sin(90)}{\pgfpoint{0cm}{0cm}}
	
	\fill[green,opacity=.3] (0,0) rectangle (24,24);
	
	\draw[black, thin] (-24,-24) grid[step=1cm] (24,24);
	
	\draw[red,thick, shift={(-1,-1)}, scale=7] (-16,-16) grid (24,24);
	\foreach \x in {-4,-3,...,5}{
		\foreach \y in {-4,-3,...,5}{
			\draw[red, thick]  (-1+7*\x,-1+7*\y) -- (-1+7*\x+7,-1+7*\y-7);
		}
	}
	
	\foreach \x in {-4,-3,...,5}{
		\foreach \y in {-4,-3,...,5}{
			\draw[red, thick]  (-1+7*\x,-1+7*\y) -- (-1+7*\x+14,-1+7*\y-7);
		}
	}
	\node[draw,circle,fill,scale=.4] at (0,0){};
	\draw [blue,thick,-stealth](0,0) -- (0,1);
	\draw [blue,thick,-stealth](0,0) -- (0,-1);
	\draw [blue,thick,-stealth](0,0) -- (2,-1) node[right]{$\alpha_2$};
	\draw [blue,thick,-stealth](0,0) -- (-2,1);
	\draw [blue,thick,-stealth](0,0) -- (2,-2);
	\draw [blue,thick,-stealth](0,0) -- (-2,2) node[left]{$\alpha_1$};
	\draw [blue,thick,-stealth](0,0) -- (-2,0);
	\draw [blue,thick,-stealth](0,0) -- (2,0);
	
	\node[red] at (1.5,2.5){$C_2$};
	\node[red]  at (1.5,0.5){$C_1$};
	\node[red]  at (3.5,2.5){$C_3$};
	\node[red]  at (4.5,4.5){$C_4$};
	\node[red]  at (7.5,0.5){$C_5$};
	\node[red]  at (0.5,7.5){$C_6$};
	\node[red]  at (7.5,3.5){$C_7$};
	
	\draw [thick,-stealth](0,0) -- (1,0) node[right]{$\omega_2$};
	\draw [thick,-stealth](0,0) -- (0,1) node[above]{$\omega_1$};
\end{tikzpicture}

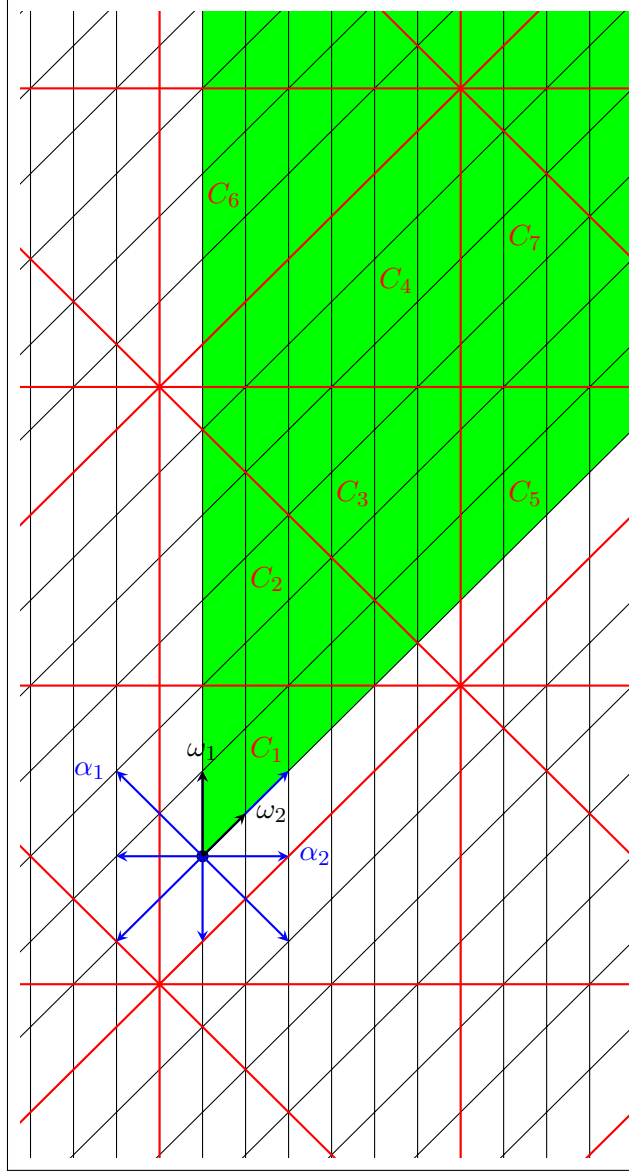
\captionof{figure}{Alcoves for $B_2$ and $p=7$.}\pagebreak
\end{center}

\begin{rem}\label{B2_symmetry}
	In Euclidean coordinates, we have
	\begin{enumerate}[label = (\arabic*)] 
		\item $s_{\alpha_1}\sbullet (a,b)=(b-1,a+1)$,
		\item $s_{\alpha_2}\sbullet (a,b)=(a,-b-1)$,
		\item $s_{\alpha_1+\alpha_2,p}\sbullet(a,b)=(p-3-a,b)$,
		\item $s_{\alpha_1+2\alpha_2,p}\sbullet(a,b)=(a-m,b-m)$ with $m=(a+b+2-p)$.
	\end{enumerate}
\end{rem}

\begin{lemma}\label{B2_order}
	Let $\lambda=(a,b),\mu=(c,d)\in X$ with $\mu\leq \lambda$. Then $c+d\leq a+b$.
\end{lemma}

\begin{proof}
	Let $s,t\in \N$ be such that $\mu=\lambda-s\alpha_1-t\alpha_2$. We have $\mu=(a-2s+t,b-2t+2s)=:(c,d)$ and $c+d=a+b-t\leq a+b$.
\end{proof}

\subsection{Structure of Weyl modules and weights in irreducible modules}
In this subsection, we compute the composition factors of Weyl modules with $p$-restricted highest weight and the multiplicity of certain weights in irreducible modules with $p$-restricted highest weight.

\begin{lemma}\label{Weyl_B2}
	Let $\lambda=(a,b)\in X^+$ be a $p$-restricted dominant weight.
	\begin{itemize}
	\item If $\lambda\in C_1$, then $\Delta(\lambda)\cong L(\lambda)$. 
	\item If $\lambda\in C_i$ for $i\in \{2,3,4\}$, then $\Delta(\lambda)$ admits the unique composition series $[L(\lambda),L(\mu)]$ where $\mu\in C_{i-1}$ is the unique weight linked to $\lambda$.
	\item If $\lambda\in F_{i,j}$ with $(i,j)\neq (4,7)$ then $\Delta(\lambda)\cong L(\lambda)$. 
	\item If $\lambda\in F_{4,7}\setminus F_{4,6}$, then $\Delta(\lambda)$ admits the unique composition series $[L(\lambda),L(\mu)]$ where $\mu\in F_{3,5}$ is the unique weight linked to $\lambda$.
	\end{itemize}
In particular, $T(\lambda)$ is irreducible if and only if $\lambda\in \widehat{C_1}\cup F_{2,3}\cup F_{3,4}\cup F_{3,5}\cup F_{4,6}$.
\end{lemma}

\begin{proof}
	We prove this lemma using Jantzen $p$-sum formula (Proposition \ref{Jantzen_p_sum}). For a weight $\lambda\in C_i$, we write $\lambda_j$ for the unique weight in $(W_p\sbullet \lambda) \cap C_j$. We will use Remark \ref{Jantzen_2_comp} and Lemma \ref{action_characters} several times without further reference.

	\begin{itemize}
		\item If $\lambda\in \widehat{C_1}$ then $\Delta(\lambda)\cong L(\lambda)$ by Lemma \ref{fundamental_alcove}.
		\item If $\lambda \in C_2$, then 
		$$\JSF(\lambda)=\chi(s_{\alpha_1+\alpha_2,p}\sbullet \lambda)=\chi(\lambda_1)=\ch L(\lambda_1).$$
		Therefore, $\chi(\lambda_2)=\ch L(\lambda_2)+\ch L(\lambda_1)$ and $\ch L(\lambda_2)=\chi(\lambda_2)-\chi(\lambda_1)$.
		\item If $\lambda \in C_3$, then
		$$\JSF(\lambda)=\chi(s_{\alpha_1+\alpha_2,p}\sbullet \lambda)+\chi(s_{\alpha_1+2\alpha_2,p}\sbullet\lambda)=-\chi(\lambda_1)+\chi(\lambda_2)=\ch L(\lambda_2).$$
		Therefore, $\chi(\lambda_3)=\ch L(\lambda_3)+\ch L(\lambda_2)$ and $\ch L(\lambda_3)=\chi(\lambda_3)-\chi(\lambda_2)+\chi(\lambda_1)$.
		\item If $\lambda\in C_4$, then
		\begin{align*}
			\JSF(\lambda)&=\chi(s_{\alpha_1+\alpha_2,p}\sbullet \lambda)+\chi(s_{\alpha_1+\alpha_2,2p}\sbullet\lambda)+\chi(s_{\alpha_1+2\alpha_2,p}\sbullet \lambda)=\chi(\lambda_1)+\chi(\lambda_3)-\chi(\lambda_2)\\
			&=\ch L(\lambda_3).
		\end{align*}
		\item If $\lambda\in F_{2,3}$, then $$\JSF(\lambda)=\chi(s_{\alpha_1+\alpha_2,p}\sbullet \lambda)=0$$
		because $s_{\alpha_1+\alpha_2,p}\sbullet \lambda\in D\setminus X^+$. Thus $\Delta(\lambda)$ is irreducible.
		\item If $\lambda\in F_{3,4}$, then 
		$$\JSF(\lambda)=\chi(s_{\alpha_1+\alpha_2,p}\sbullet \lambda)+\chi(s_{\alpha_1+2\alpha_2,p}\sbullet\lambda)=0$$
		because $s_{\alpha_1+2\alpha_2,p}\sbullet\lambda\in D\setminus X^+$ and $s_{\alpha_1}s_{\alpha_1+\alpha_2,p}\sbullet \lambda\in D\setminus X^+$. Thus $\Delta(\lambda)$ is irreducible.
		\item If $\lambda\in F_{3,5}$, then 
		$$\JSF(\lambda)=\chi(s_{\alpha_1+\alpha_2,p}\sbullet \lambda)+\chi(s_{\alpha_1+2\alpha_2,p}\sbullet\lambda)=-\chi(s_{\alpha_1}s_{\alpha_1+\alpha_2,p}\sbullet \lambda)+\chi(s_{\alpha_1+2\alpha_2,p}\sbullet\lambda)=0$$
		because $s_{\alpha_1}s_{\alpha_1+\alpha_2,p}\sbullet \lambda=s_{\alpha_1+2\alpha_2,p}\sbullet\lambda$. Thus $\Delta(\lambda)$ is irreducible.
		\item If $\lambda\in F_{4,6}$, then
		\begin{align*}
			\JSF(\lambda)&=\chi(s_{\alpha_1+\alpha_2,p}\sbullet \lambda)+\chi(s_{\alpha_1+\alpha_2,2p}\sbullet\lambda)+\chi(s_{\alpha_1+2\alpha_2,p}\sbullet \lambda)\\
			&=\chi(s_{\alpha_1+\alpha_2,2p}\sbullet\lambda)-\chi(s_{\alpha_2}s_{\alpha_1+2\alpha_2,p}\sbullet \lambda)=0
		\end{align*}
		where we used that $s_{\alpha_1+\alpha_2}s_{\alpha_1+\alpha_2,p}\sbullet \lambda\in D\setminus X^+$ in the second equality and \linebreak$s_{\alpha_1+\alpha_2,2p}\sbullet\lambda=s_{\alpha_2}s_{\alpha_1+2\alpha_2,p}\sbullet \lambda$ in the last equality. Thus $\Delta(\lambda)$ is irreducible.
		\item If $\lambda\in F_{4,7}\setminus F_{4,6}$, then
		\begin{align*}
			\JSF(\lambda)&=\chi(s_{\alpha_1+\alpha_2,p}\sbullet \lambda)+\chi(s_{\alpha_1+\alpha_2,2p}\sbullet\lambda)+\chi(s_{\alpha_1+2\alpha_2,p}\sbullet \lambda)\\
			&=\chi(s_{\alpha_1}s_{\alpha_1+\alpha_2}s_{\alpha_1+\alpha_2,p}\sbullet \lambda)+\chi(s_{\alpha_1+\alpha_2,2p}\sbullet\lambda)-\chi(s_{\alpha_2}s_{\alpha_1+2\alpha_2,p}\sbullet \lambda)\\
			&=\chi(s_{\alpha_1+\alpha_2,2p}\sbullet\lambda)
	\end{align*}
		where we used that $s_{\alpha_1}s_{\alpha_1+\alpha_2}s_{\alpha_1+\alpha_2,p}\sbullet \lambda=s_{\alpha_2}s_{\alpha_1+2\alpha_2,p}\sbullet \lambda$ in the last equality. Now observe that $s_{\alpha_1+\alpha_2,2p}\sbullet\lambda\in F_{3,5}$, so $\chi(s_{\alpha_1+\alpha_2,2p}\sbullet\lambda)=\ch L(s_{\alpha_1+\alpha_2,2p}\sbullet\lambda)$ by one of the previous cases and we are done.
	\end{itemize}
The last claim follows directly from Lemma \ref{tilting_irreducible}.
\end{proof}

Combining Lemma \ref{Weyl_B2} and Remark \ref{B2_symmetry}, we get the following remark.

\begin{rem}\label{B2_character_symmetry}
	In this remark, we use Euclidean coordinates. Let $\lambda=(a,b)\in X^+$. We set $m:=a+b+2-p$.
	\begin{itemize}
		\item If $\lambda \in \widehat{C_2}$, then $\ch L(\lambda)=\chi(a,b)-\chi(p-3-a,b)$.
		\item If $\lambda \in C_3$, then $\chi(a,b)=\ch L(a,b)+\ch L(a-m,b-m)$.
	\end{itemize}
\end{rem}

\begin{lemma}[{\cite[1.35]{Testerman_AMS}.}]\label{B2_-11}
	Let $\lambda=(a,b)\in X^+$ be $p$-restricted with $a\neq 0$ and $b\neq 0$. Then
	$$m_{L(\lambda)}(\lambda-\alpha_1-\alpha_2)=\begin{cases}
		1 &\text{if } \;2a+b+2\equiv 0\bmod p,\\
		2 &\text{otherwise.}
	\end{cases}$$
\end{lemma}

\begin{lemma}\label{table_B2}
	Let $\lambda=(a,b)\in X^+$. We have the following weight multiplicities:\\
	\\
	\begin{tabular}{|c|c|c|c|c|c|}
		\hline
		$\lambda=(a,b)$ & $a\geq 1,b\geq 2$ &$a\geq 1,b=1$  & $a=0,b\geq 2$  & $a\geq 1,b=0$ & $a=0,b=1$ \\
		\hline
		$m_{\Delta(\lambda)}(\lambda)$& $1$ & $1$ & $1$ & $1$ & $1$ \\
		\hline
		$m_{\Delta(\lambda)}(\lambda-\alpha_1)$& $1$ & $1$ & $0$ & $1$ & $0$ \\
		\hline
		$m_{\Delta(\lambda)}(\lambda-\alpha_2)$& $1$ & $1$ & $1$ & $0$ & $1$ \\
		\hline
		$m_{\Delta(\lambda)}(\lambda-\alpha_1-\alpha_2)$& $2$ & $2$ & $1$ & $1$ & $1$ \\
		\hline
		$m_{\Delta(\lambda)}(\lambda-2\alpha_2)$& $1$ & $0$ & $1$ & $0$ & $0$ \\
		\hline
		$m_{\Delta(\lambda)}(\lambda-\alpha_1-2\alpha_2)$& $3$ & $2$ & $2$ & $1$ & $1$ \\
		\hline
	\end{tabular}
\vspace{1em}
\\
	Moreover, for $\lambda=(a,0)\in X^+$ with $a\geq 2$, we have
	\begin{align*}
		&m_{\Delta(a,0)}(\lambda-\alpha_1-\alpha_2)=1, &&m_{\Delta(a,0)}(\lambda-\alpha_1-2\alpha_2)=1,
		&&m_{\Delta(a,0)}(\lambda-2\alpha_1-\alpha_2)=1,\\
		&m_{\Delta(a,0)}(\lambda-2\alpha_1-2\alpha_2)=2.
	\end{align*}
\end{lemma}
\begin{proof}
	This follows from Proposition \ref{Cavallin} and tables of dominant weights (\cite{tables_weight}).
\end{proof}

\begin{lemma}\label{B2_-12}
	Let $\lambda=(a,b)\in X^+$ with $1\leq a<p$ and $2\leq b<p$. Then
	$$m_{L(\lambda)}(\lambda-\alpha_1-2\alpha_2)=\begin{cases}
		2\quad &\text{if } \; 2a+b+2\equiv 0\bmod p,\\
		2\quad &\text{if } \; a+b=p-1,\\
		3\quad & \text{otherwise}.
	\end{cases}$$
\end{lemma}

\begin{proof}\
	\begin{itemize}
		
	\item If $2a+b+2=p$, then $\lambda\in C_2$ and by Lemma \ref{Weyl_B2}, $\Delta(\lambda)$ admits the unique composition series $[L(\lambda),L(\lambda-\alpha_1-\alpha_2)]$. By Lemma \ref{table_B2}, we get
	$$m_{L(\lambda)}(\lambda-\alpha_1-2\alpha_2)=m_{\Delta(\lambda)}(\lambda-\alpha_1-2\alpha_2)-m_{\Delta(\lambda-\alpha_1-\alpha_2)}(\lambda-\alpha_1-2\alpha_2)=3-1=2.$$
	
	\item If $2a+b+2=2p$, then $\lambda\in C_4\cup (F_{4,7}\setminus F_{4,6})$ and by Lemma \ref{Weyl_B2}, $\Delta(\lambda)$ admits the unique composition series 
	$[L(\lambda),L(\lambda-\alpha_1-\alpha_2)]$. By Lemmas \ref{table_B2} and \ref{-1}, we get
	$$m_{L(\lambda)}(\lambda-\alpha_1-2\alpha_2)=m_{\Delta(\lambda)}(\lambda-\alpha_1-2\alpha_2)-m_{L(\lambda-\alpha_1-\alpha_2)}(\lambda-\alpha_1-2\alpha_2)=3-1=2.$$
	
	\item If $a+b=p-1$, then $\lambda\in C_3$ and by Lemma \ref{Weyl_B2}, $\Delta(\lambda)$ admits the unique composition series $[L(\lambda),L(\lambda-\alpha_1-2\alpha_2)]$. By Lemma \ref{table_B2}, we get
	$$m_{L(\lambda)}(\lambda-\alpha_1-2\alpha_2)=m_{\Delta(\lambda)}(\lambda-\alpha_1-2\alpha_2)-m_{\Delta(\lambda-\alpha_1-2\alpha_2)}(\lambda-\alpha_1-2\alpha_2)=3-1=2.$$
	
	\item In all other cases, either $\Delta(\lambda)\cong L(\lambda)$ or all composition factors of $\Delta(\lambda)$ except $L(\lambda)$ have highest weight $\nu\ngeq\lambda-\alpha_1-2\alpha_2$, thus 
	$$m_{L(\lambda)}(\lambda-\alpha_1-2\alpha_2)=m_{\Delta(\lambda)}(\lambda-\alpha_1-2\alpha_2)=3$$
	by Lemma \ref{table_B2}.\qedhere
\end{itemize}
\end{proof}

\begin{lemma}
	Let $\lambda=(a,1)\in X^+$ with $1\leq a<p$. Then
	$$m_{L(\lambda)}(\lambda-\alpha_1-2\alpha_2)=\begin{cases}
		1\quad &\text{if } \; 2a+3=p,\\
		2\quad & \text{otherwise}.
	\end{cases}$$
\end{lemma}
\pagebreak
\begin{proof}\
	\begin{itemize}
		\item 	If $2a+3=p$, then $\lambda\in C_2$ and by Lemma \ref{Weyl_B2}, $\Delta(\lambda)$ admits the unique composition series $[L(\lambda),L(\lambda-\alpha_1-\alpha_2)]$. By Lemma \ref{table_B2}, we get
		$$m_{L(\lambda)}(\lambda-\alpha_1-2\alpha_2)=m_{\Delta(\lambda)}(\lambda-\alpha_1-2\alpha_2)-m_{\Delta(\lambda-\alpha_1-\alpha_2)}(\lambda-\alpha_1-2\alpha_2)=2-1=1.$$
		\item In all other cases, either $\Delta(\lambda)\cong L(\lambda)$ or all composition factors of $\Delta(\lambda)$ except $L(\lambda)$ have highest weight $\nu\ngeq\lambda-\alpha_1-2\alpha_2$, thus 
		$$m_{L(\lambda)}(\lambda-\alpha_1-2\alpha_2)=m_{\Delta(\lambda)}(\lambda-\alpha_1-2\alpha_2)=2$$
		by Lemma \ref{table_B2}.\qedhere
	\end{itemize}
\end{proof}

\begin{lemma}\label{S_Weyl_a0}
	Let $2\leq a<p$ and set $W:=\Delta(1,0)$. Then $\chi(a,0)=\ch S^a(W)-\ch S^{a-2}(W)$ and $\Delta(0,a) \cong S^a(\Delta(0,1))$. In particular $\chi(0,a)=\ch S^a(\Delta(0,1))$.
\end{lemma}

\begin{proof}
	For $\chi(a,0)$, see \cite[§19.5]{Fulton_Harris}.
	For $\Delta(0,a)$, observe that $S^a(\Delta(0,1))$ admits a maximal vector of weight $(0,a)$. By Lemma \ref{Weyl_univ_prop}, there exists a non-zero morphism \linebreak$\theta:\Delta(0,a)\to S^a(\Delta(0,1))$. By Lemma \ref{Weyl_B2}, $\Delta(0,a)$ is irreducible, thus $\theta$ is injective. To show that $\theta$ is surjective and thus an isomorphism, it is enough to show that \linebreak$\dim \Delta(0,a)=\dim S^a(\Delta(0,1))$.
	Using usual multilinear algebra, we have \linebreak$\dim S^a(\Delta(0,1))=\binom{a+3}{3}$.
	On the other hand, using Weyl's degree formula (Corollary \ref{Weyl_degree}), with the choice $(\alpha_2,\alpha_2)=2$ (and hence $(\alpha_1,\alpha_1)=4$), we have
	\begin{align*}
		\dim \Delta(0,a)&=\frac{(\lambda+\rho,\alpha_1)(\lambda+\rho,\alpha_2)(\lambda+\rho,\alpha_1+\alpha_2)(\lambda+\rho,\alpha_1+2\alpha_2)}{(\rho,\alpha_1)(\rho,\alpha_2)(\rho,\alpha_1+\alpha_2)(\rho,\alpha_1+2\alpha_2)}\\
		&=\frac{2(a+1)(a+3)(2a+4)}{2\cdot 1\cdot 3\cdot 4}=\binom{a+3}{3},
	\end{align*}
	hence we are done.
\end{proof}


\begin{notation}
	Let $\lambda=x\epsilon_1+y\epsilon_2\in X$. We define its $1-$norm  
	$$\lVert \lambda\rVert:= \lvert x\rvert+ \lvert y\rvert$$
	and its $\infty$-norm 
	$$\lvert \lambda\rvert_\infty:=\max\{\lvert x\rvert ,\lvert y\rvert\}.$$
\end{notation}

\begin{lemma}\label{B2_weight_a0}
	Let $a\in \N$ and $\lambda\in X$. We have 
	$$m_{\Delta(a,0)}(\lambda)=\begin{cases}
		\lfloor\frac{a-\lVert \lambda\rVert}{2}\rfloor+1 \quad &\text{ if }\lVert \lambda\rVert\in \N \text{ and }\lVert \lambda\rVert\leq a,\\
		0 \quad &\text{ else.}
	\end{cases}$$
\end{lemma}

\begin{proof}
	In this proof, we use Euclidean coordinates.
	We set $W:=\Delta(1,0)$. Let $k\in \N$. We compute $m_{S^kW}(\lambda)$ and then use Lemma \ref{S_Weyl_a0}.
	
	Let $i:=\lfloor\frac{k-\lVert \lambda\rVert}{2}\rfloor$. We claim that 
	$$m_{S^kW}(\lambda)=\begin{cases}
		\binom{2+i}{2} \quad&\text{ if } \lVert \lambda\rVert\in \N \text{ and }\lVert \lambda\rVert\leq k,\\
		0 \quad & \text{ else.}
	\end{cases}$$

We know that $\dim W=5$ and $W$ admits the five weights $0,\pm\epsilon_1,\pm\epsilon_2$, all of them with multiplicity 1. We fix $(v_{-2},v_{-1},v_0,v_1,v_2)$ an ordered basis of $W$ with $v_0$ a weight vector associated to $0$ and $v_{\pm i}$ a weight vector associated to $\pm\epsilon_i$ for $i=1,2$. By multilinear algebra, $S^kW$ admits the basis $\{v_{i_1}\otimes \cdots\otimes v_{i_k}\}_{i_1\leq i_2\leq \ldots\leq i_k}$ and there exists a natural bijection between this basis and the set $\{(x,y,z,s,t)\in \N^5|\; x+y+z+s+t=k\}$ where $x$ is the number of $v_{i_j}$ with $i_j=-2$, $y$ is the number of $v_{i_j}$ with $i_j=-1$, and so on. Moreover, under this bijection, a basis vector associated to $(x,y,z,s,t)$ is a weight vector with weight $(t-x)\epsilon_2+(s-y)\epsilon_1$. Thus, to compute $m_{S^kW}(\lambda)$, we will count the number of $5$-tuples $(x,y,z,s,t)$ with associated weight $\lambda$. Let us write $\lambda=f\epsilon_1+g\epsilon_2$. Then $(x,y,z,s,t)$ is associated to $\lambda$ if $t-x=g$ and $s-y=f$. In particular, we need $\lVert \lambda\rVert\in \N$ and $\lVert \lambda\rVert\leq k$.
In this case, we get 
$$m_{S^kW}(\lambda)=|\{(x,y,z,s,t)\in \N^5|\; x+y+z+s+t=k,\; t-x=g,\; s-y=f\}|.$$
Up to symmetry we can assume that both $f$ and $g$ are non-negative.
We get 
\begin{align*}
	m_{S^kW}(\lambda)&=|\{(x,y,z,s,t)\in \N^5|\; x+y+z+s+t=k,\; t-x=g,\; s-y=f\}|\\
	&=|\{(x,y,z)\in \N^3|\; 2x+g+2y+f+z=k\}|\\
	&=|\{(x,y,z)\in \N^3|\; 2x+2y+z=k-\lVert \lambda\rVert\}|.
\end{align*}
If $k-\lVert \lambda\rVert$ is odd, then so is $z$, and we have 
$$|\{(x,y,z)\in \N^3|\; 2x+2y+z=k-\lVert \lambda\rVert\}|=|\{(x,y,z)\in \N^3|\; 2x+2y+(z-1)=k-\lVert \lambda\rVert-1\}|.$$
In this case, let $z':=(z-1)/2$. We get 
\begin{align*}
	m_{S^kW}(\lambda)&=|\{(x,y,z')\in \N^3|\; 2x+2y+2z'=k-\lVert \lambda\rVert-1\}|\\
	&=|\{(x,y,z')\in \N^3|\; x+y+z'=(k-\lVert \lambda\rVert-1)/2\}|\\
	&=\binom{2+(k-\lVert \lambda\rVert-1)/2}{2}\\
	&=\binom{2+\lfloor(k-\lVert \lambda\rVert)/2\rfloor}{2}
\end{align*}
where the third equality is a well-known combinatorial result, see for example \cite[3.3]{combinatorics}.
Now suppose that $k-\lVert \lambda\rVert$ and $z$ are even, and set $z':=z/2$. Using the same reasoning we get
\begin{align*}
	m_{S^kW}(\lambda)&=|\{(x,y,z')\in \N^3|\; 2x+2y+2z'=k-\lVert \lambda\rVert\}|\\
	&=|\{(x,y,z')\in \N^3|\; x+y+z'=(k-\lVert \lambda\rVert)/2\}|\\
	&=\binom{2+(k-\lVert \lambda\rVert)/2}{2}\\
	&=\binom{2+\lfloor(k-\lVert \lambda\rVert)/2\rfloor}{2}.
\end{align*}
Hence in both cases we are done with our claim.\\

	Now let $i:=\lfloor\frac{a-\lVert \lambda\rVert}{2}\rfloor$ and observe that $\lfloor\frac{(a-2)-\lVert \lambda\rVert}{2}\rfloor=i-1$. For $\lVert \lambda\rVert\in \N$ and $\lVert \lambda\rVert\leq a-2$, by Lemma \ref{S_Weyl_a0}, we have
	\begin{align*}
		m_{\Delta(a,0)}(\lambda)&=m_{S^aW}(\lambda)-m_{S^{a-2}W}(\lambda)=\binom{2+i}{2}-\binom{1+i}{2}=\frac{(2+i)!}{2!i!}-\frac{(1+i)!}{2!(i-1)!}\\
		&=\frac{(2+i)!-i(1+i)!}{2!i!}=\frac{2(i+1)!}{2!i!}=i+1=\lfloor\frac{a-\lVert \lambda\rVert}{2}\rfloor+1.
	\end{align*}
If $\lVert \lambda\rVert\in \{a,a-1\}$, then $i=0$ and $$m_{\Delta(a,0)}(\lambda)=m_{S^aW}(\lambda)=\binom{2}{2}=1.$$
If $\lVert \lambda\rVert\notin \N$ or if $\lVert \lambda\rVert> a$, we conclude that $m_{\Delta(a,0)}(\lambda)=0$.
\end{proof}


\begin{cor}\label{Weyl_a0}
	Let $a\in \N$. We set 
	$$Y(i,j):=\lfloor\frac{a-|i|-|j|}{2}\rfloor+1.$$
	We have
	$$\chi(a,0)=\sum_{i=0}^a\sum_{j=i-a}^{a-i}Y(i,j)e^{j\epsilon_1+i\epsilon_2}+\sum_{i=1}^a\sum_{j=i-a}^{a-i}Y(i,j)e^{j\epsilon_1-i\epsilon_2}.$$
\end{cor}

\begin{proof}
	This is a direct consequence of Lemma \ref{B2_weight_a0}.
\end{proof}

\begin{cor}\label{irred_a0}
	Let $a=\frac{p-1}{2}$ and $\lambda\in X$. Then 
	$$m_{L(a,0)}(\lambda)=\begin{cases}
		1 \quad &\text{ if }\lVert \lambda\rVert\in \N \text{ and }\lVert \lambda\rVert\leq a,\\
		0 \quad &\text{ else.}
	\end{cases}$$
In particular, we have
$$\ch L(a,0)=\sum_{i=0}^a\sum_{j=i-a}^{a-i}e^{j\epsilon_1+i\epsilon_2}+\sum_{i=1}^a\sum_{j=i-a}^{a-i}e^{j\epsilon_1-i\epsilon_2}.$$
\end{cor}

\begin{proof}
	Observe that $(a,0)\in C_2$. Moreover, $s_{\alpha_1+\alpha_2,p}\sbullet(a,0)=(a-2,0)$, so $(a-2,0)\in C_1$ is the only weight in the first alcove linked to $(a,0)$. By Lemmas \ref{Weyl_B2} and \ref{Weyl_a0}, we have
	\begin{align*}
		m_{L(a,0)}(\lambda)&=m_{\Delta(a,0)}(\lambda)-m_{\Delta(a-2,0)}(\lambda)\\
		&=\begin{cases}
			\lfloor\frac{a-\lVert\lambda\rVert}{2}\rfloor+1-\left(	\lfloor\frac{a-2-\lVert\lambda\rVert}{2}\rfloor+1\right)\quad &\text{ if }\lVert \lambda\rVert\in \N \text{ and }\lVert \lambda\rVert\leq a-2,\\
			\lfloor\frac{a-\lVert\lambda\rVert}{2}\rfloor+1\quad &\text{ if } \lVert \lambda\rVert\in\{a,a-1\},\\
			0 \quad &\text{ else.}
		\end{cases}
	\end{align*}
Observe that $\lfloor\frac{a-\lVert\lambda\rVert}{2}\rfloor+1=1$ for $\lVert \lambda\rVert\in\{a,a-1\}$ and 
$$\lfloor\frac{a-\lVert\lambda\rVert}{2}\rfloor+1-\left(	\lfloor\frac{a-2-\lVert\lambda\rVert}{2}\rfloor+1\right)=\lfloor1+\frac{a-2-\lVert\lambda\rVert}{2}\rfloor-\lfloor\frac{a-2-\lVert\lambda\rVert}{2}\rfloor=1,$$
which allows us to conclude.
\end{proof}

\begin{lemma}\label{B2_weight_0b}
	Let $0\leq b<p$ and $\lambda\in X$. Then
	$$m_{L(0,b)}(\lambda)=m_{\Delta(0,b)}(\lambda)=\begin{cases}
		(\frac b2-\vert\lambda\rvert_\infty)+1 \quad &\text{ if }\frac b2-\lvert\lambda\rvert_\infty\in \N,\\
		0 \quad &\text{ else.}
	\end{cases}$$
\end{lemma}

\begin{proof} 
	Let $V:=\Delta(0,1)$. First, observe that $L(0,b)=\Delta(0,b)=S^bV$ by Lemmas \ref{Weyl_B2} and \ref{S_Weyl_a0}.  We know that $\dim V=4$ and $V$ admits the four weights $\frac12(\pm\epsilon_1\pm\epsilon_2)$, all of them with multiplicity 1. We fix a basis of weight vectors $(v_{-2},v_{-1},v_1,v_2)$ with $v_{\pm 2}$ associated to $\pm\frac12(\epsilon_1+\epsilon_2)$ and $v_{\pm 1}$ to $\pm\frac{1}{2}(\epsilon_1-\epsilon_2)$. By multilinear algebra, $S^bV$ admits the basis \linebreak$\{v_{i_1}\otimes \cdots\otimes v_{i_k}\}_{i_1\leq i_2\leq \ldots\leq i_k}$ and there exists a natural bijection between this basis and the set $\{(x,y,s,t)\in \N^4|\; x+y+s+t=b\}$ where $x$ is the number of $v_{i_j}$ with $i_j=-2$, $y$ is the number of $v_{i_j}$ with $i_j=-1$, and so on. Moreover, under this bijection, a basis vector associated to $(x,y,s,t)$ is a weight vector with weight $$(t-x)\frac12(\epsilon_1+\epsilon_2)+(s-y)\frac12(\epsilon_1-\epsilon_2)=\frac12(t+s-x-y)\epsilon_1+(t+y-x-s)\epsilon_2.$$ 
	Thus, to compute $m_{S^bV}(\lambda)$, we will count the number of $4$-tuples $(x,y,s,t)$ with associated weight $\lambda$. Let us write $\lambda=f\frac12(\epsilon_1+\epsilon_2)+g\frac12(\epsilon_1-\epsilon_2)$. Then $(x,y,s,t)$ is associated to $\lambda$ if $t-x=f$ and $s-y=g$. Therefore we get 
	$$m_{S^bV}(\lambda)=|\{(x,y,s,t)\in \N^4|\; x+y+s+t=b,\; t-x=f,\; s-y=g\}|.$$
	
	Up to symmetry we can assume that both $f$ and $g$ are non-negative. In particular, this implies that $\vert\lambda\rvert_\infty=\frac{1}{2}(f+g)$, so we get 
	\begin{align*}
		m_{S^bV}(\lambda)&=|\{(x,y,s,t)\in \N^4|\; x+y+s+t=b,\; t-x=f,\; s-y=g\}|\\
		&=|\{(x,y)\in \N^2|\; 2x+f+2y+g=b\}|\\
		&=|\{(x,y)\in \N^2|\; 2x+2y=b-2\vert\lambda\rvert_\infty\}|\\
		&=|\{(x,y)\in \N^2|\; x+y=\tfrac12b-\vert\lambda\rvert_\infty\}|.
	\end{align*}
Clearly, the equality $x+y=\frac12b-\vert\lambda\rvert_\infty$ cannot be satisfied for $x,y\in \N$ if $\frac12b-\vert\lambda\rvert_\infty\notin \N$. Thus we can restrict our attention to the case $\frac12b-\vert\lambda\rvert_\infty\in \N$, and using the combinatorial result (\cite[3.3]{combinatorics}) again, we get
\begin{align*}
	m_{S^bV}(\lambda)&=|\{(x,y)\in \N^2|\; x+y=\tfrac12b-\vert\lambda\rvert_\infty\}=\binom{\frac12b-\vert\lambda\rvert_\infty+1}{1}=\tfrac12b-\vert\lambda\rvert_\infty+1,
\end{align*}
thus we are done.
\end{proof}

\subsection{Classification results}\label{B2_classification}

In all statements of this section, unless stated otherwise, we use coordinates with respect to the fundamental weights.

We start by stating a theorem from Stembridge (\cite[Theorem 1.1.B]{Stembridge}) which classifies multiplicity-free tensor products of simple $\Sp_4(\C)$-modules.

\begin{thm}\label{B2_p0}
	Let $\lambda=(a,b),\mu=(c,d)\in X^+$ be dominant weights. Up to the reordering of $\lambda$ and $\mu$, $L_{\C}(\lambda) \otimes L_{\C}(\mu)$ is multiplicity-free if and only if one of the following holds:
	\begin{enumerate}[label = \emph{(\arabic*)}] 
		\item $a=0$ and $b=1$,
		\item $a=1$ and $b=0$,
		\item $a=d=0$,
		\item $a=c=0$,
		\item $b=0$ and $d=1$, or
		\item $b=d=0$.
	\end{enumerate}
\end{thm}

\begin{prop}
	Let $\lambda\in C_2$ and $\mu\in \widehat{C_1}$. If $\lambda+\mu\in C_3$, then $L(\lambda)\otimes L(\mu)$ has multiplicity.
\end{prop}

\begin{proof}
	Let $\lambda':=s_{\alpha_1+\alpha_2,p}\sbullet \lambda\in C_1$, so that $\Delta(\lambda)$ admits the unique composition series $[L(\lambda),L(\lambda')]$ and let $\eta:= s_{\alpha_1+	2\alpha_2,p}\sbullet(\lambda+\mu)\in C_2$, so that $\Delta(\lambda+\mu)$ admits the unique composition series $[L(\lambda+\mu),L(\eta)]$ (Lemma \ref{Weyl_B2}). 
	We have a short exact sequence 
	$$\begin{tikzcd}
		0 & {L(\lambda')} & {\Delta(\lambda)} & {L(\lambda)} & 0
		\arrow[from=1-1, to=1-2]
		\arrow[from=1-2, to=1-3]
		\arrow[from=1-3, to=1-4]
		\arrow[from=1-4, to=1-5].
	\end{tikzcd}$$
	By Lemma \ref{Weyl_B2}, we have $L(\mu)\cong\Delta(\mu)$ and $L(\lambda')\cong \Delta(\lambda')$. Taking the tensor product with $L(\mu)$, we get the short exact sequence
	$$\begin{tikzcd}
		0 & {\Delta(\lambda')\otimes \Delta(\mu)} & {\Delta(\lambda)\otimes \Delta(\mu)} & {L(\lambda)\otimes L(\mu)} & 0
		\arrow[from=1-1, to=1-2]
		\arrow["\phi",from=1-2, to=1-3]
		\arrow["\psi",from=1-3, to=1-4]
		\arrow[from=1-4, to=1-5].
	\end{tikzcd}$$
	
	By Lemma \ref{Weyl_tensor_Weyl}, $\Delta(\lambda)\otimes \Delta(\mu)$ admits a submodule isomorphic to $\Delta(\lambda+\mu)$. We will abuse the notation and denote it by $\Delta(\lambda+\mu)$. Thus, we can restrict our exact sequence to
	$$\begin{tikzcd}
		0 & {\phi^{-1}(\Delta(\lambda+\mu))} & {\Delta(\lambda+\mu)} & {\psi(\Delta(\lambda+\mu))} & 0
		\arrow[from=1-1, to=1-2]
		\arrow["\phi", from=1-2, to=1-3]
		\arrow["\psi", from=1-3, to=1-4]
		\arrow[from=1-4, to=1-5].
	\end{tikzcd}$$
	
	Suppose for contradiction that $L(\lambda)\otimes L(\mu)$ is multiplicity-free. Then in particular it is completely reducible, and $\psi(\Delta(\lambda+\mu))$ is completely reducible (Proposition \ref{submodule_comp_red}). Therefore, $\rad \Delta(\lambda+\mu)\subseteq \ker(\psi)$ and $\psi(\Delta(\lambda+\mu))\cong L(\lambda+\mu)$ or $\psi(\Delta(\lambda+\mu))=0$. We claim that the second case is impossible. By exactness, it would imply $\phi^{-1}(\Delta(\lambda+\mu))\cong \Delta(\lambda+\mu)$, \linebreak but $\lambda +\mu \nleq \lambda'+\mu$, so it cannot appear as a submodule of $\Delta(\lambda')\otimes \Delta(\mu)$.
	Thus, \linebreak$\phi^{-1}(\Delta(\lambda+\mu))\cong L(\eta)$ and it is a submodule of $\Delta(\lambda')\otimes \Delta(\mu)$.
	
	Using Theorem \ref{Weyl_filtration}, we fix 
	$$0=V_0\subseteq V_1\subseteq\ldots\subseteq V_m=\Delta(\lambda')\otimes\Delta(\mu)$$
	a Weyl filtration. Thus there exist $\nu_1,\ldots,\nu_m\in X^+$ such that $V_i/V_{i-1}\cong \Delta(\nu_i)$ for \linebreak$i\in\{1,\ldots, m\}$. We set $W_i:=V_i\cap \phi^{-1}(\Delta(\lambda+\mu))$. Since $\phi^{-1}(\Delta(\lambda+\mu))$ is irreducible, then $W_i=0$ or $W_i\cong L(\eta)$. Let $j$ be minimal such that $W_j\cong L(\eta)$ (in particular $W_{j-1}=0$). We have the following situation
	$$\begin{tikzcd}
		0 & {V_{j-1}} & {V_j} & {\Delta(\nu_j)} & 0 \\
		&& {L(\eta)}
		\arrow[from=1-1, to=1-2]
		\arrow[hook, from=1-2, to=1-3]
		\arrow["\pi_j",two heads, from=1-3, to=1-4]
		\arrow[from=1-4, to=1-5]
		\arrow[hook, from=2-3, to=1-3]
		\arrow[dashed, hook, from=2-3, to=1-4]
	\end{tikzcd}$$
	Since $\ker \pi_j=V_{j-1}$ and $W_j\cap V_{j-1}=W_{j-1}=0$, we have an injective map $L(\eta)\to \Delta(\nu_j)$, so $L(\eta)$ is a submodule of $\Delta(\nu_j)$. In particular, $L(\eta)$ is a composition factor of $\Delta(\nu_j)$, so $\eta \uparrow \nu_j$ by the Strong Linkage Principle (Proposition \ref{strong_linkage}). Recall at this step that $\eta \uparrow \lambda+\mu$.
	Now, observe that $\nu_j\leq \lambda'+\mu < \lambda+\mu$, so $\nu_j\neq \lambda +\mu$. 
	By the geometry of alcoves, it follows that $\nu_j\in C_2$, so $\eta=\nu_j$. But $L(\eta)$ is not a submodule of $\Delta(\eta)$ (Lemma \ref{Weyl_B2}), so we get a contradiction. Therefore, $L(\lambda)\otimes L(\mu)$ has multiplicity.
\end{proof}

\subsubsection{$L(0,b)\otimes L(0,d)$}

\begin{prop}\label{B2_bd}
	Let $\lambda=(0,b),\mu=(0,d)\in X^+$ be $p$-restricted dominant weights with {$0<\nolinebreak b,d<p$}. Then $L(\lambda)\otimes L(\mu)$ is multiplicity-free if and only if one of the following holds:
	\begin{enumerate}[label = \emph{(\arabic*)}] 
		\item $b+d\leq p-3$ (i.e. $\lambda+\mu\in \widehat{C_1}$) or
		\item $(b,d)\in \{(1,p-2),(p-2,1)\}$.
	\end{enumerate}
\end{prop}

\begin{proof}
	We set $M:=L(\lambda)\otimes L(\mu)$. By Lemma \ref{Weyl_B2}, $L(\lambda)$ and $L(\mu)$ are tilting modules, so $M$ is a tilting module.
	
	If $b+d\geq p$, we conclude directly from Corollary \ref{sum_p_res} that $M$ has multiplicity.\\
	
	If $b+d\leq p-3$ (i.e. $\lambda+\mu\in \widehat{C_1}$), we apply Corollary \ref{sum_C1} and Theorem \ref{B2_p0} to conclude that $M$ is multiplicity-free.\\
	
	Suppose that $b+d=p-2$. By Lemma \ref{-1}, we have
	\begin{align*}
		&m_{L(\lambda)}(\lambda-\alpha_2)=1,
		&&m_{L(\mu)}(\mu-\alpha_2)=1,
		&&m_{L(\lambda+\mu)}(\lambda+\mu-\alpha_2)=1.
	\end{align*}
Using Argument \ref{argument1}, we have 
$$m_{M}(\lambda+\mu-\alpha_2)=1+1=2,$$
and we deduce that $L(\lambda+\mu-\alpha_2)$ is a composition factor of $M$. Observe that \linebreak$\lambda+\mu-\alpha_2=(1,p-4)\in C_2$, so $T(\lambda+\mu-\alpha_2)$ is not irreducible by Lemma \ref{Weyl_B2}. We can thus conclude by Lemma \ref{argument_tilting} that $M$ has multiplicity.\\

If $b+d=p-1$ and $b\neq1, d\neq 1$, we use Argument \ref{argument1}. By Lemma \ref{-1}, we have
\begin{align*}
	&m_{L(\lambda)}(\lambda-\alpha_2)=1,
	&&m_{L(\lambda)}(\lambda-2\alpha_2)=1,\\
	&m_{L(\mu)}(\mu-\alpha_2)=1,
	&&m_{L(\mu)}(\mu-2\alpha_2)=1,\\
	&m_{L(\lambda+\mu)}(\lambda+\mu-\alpha_2)=1,
	&&m_{L(\lambda+\mu)}(\lambda+\mu-2\alpha_2)=1,\\
	&m_{L(\lambda+\mu-\alpha_2)}(\lambda+\mu-\alpha_2)=1,
	&&m_{L(\lambda+\mu-\alpha_2)}(\lambda+\mu-2\alpha_2)=1.
\end{align*}
Therefore, we get 
$$m_M(\lambda+\mu-\alpha_2)=2,\qquad m_M(\lambda+\mu-2\alpha_2)=3.$$
We deduce that $L(\lambda+\mu-2\alpha_2)$ is a composition factor of $M$. Observe that \linebreak$\lambda+\mu-2\alpha_2=(2,p-5)\in C_2$, so $T(\lambda+\mu-\alpha_2)$ is not irreducible by Lemma \ref{Weyl_B2}. As in the previous case, we conclude by Lemma \ref{argument_tilting} that $M$ has multiplicity. \\

Finally, consider the case $b=1,d=p-2$ (the case $b=p-2,d=1$ is symmetric). By Proposition \ref{product_characters} and Lemmas \ref{Weyl_B2} and \ref{action_characters}, we have
\begin{align*}
	\ch M&=\chi(\lambda)\chi(\mu)=\chi(0,p-1)+\chi(1,p-3)+\chi(0,p-3)+\chi(-1,p-1)\\
	&=\ch L(0,p-1)+\ch L(1,p-3)+\ch L(0,p-3).
\end{align*}
Therefore, $M$ is multiplicity-free.
\end{proof}

\subsubsection{$L(a,0)\otimes L(c,0)$}

\begin{lemma}\label{B2_ac_technique}
	Let $a,c\in \N$. We use Euclidean coordinates. For $i,j\in \Z$, we set $$\delta(a,i,j):=\begin{cases}
		1&\quad \text{if } a- i-j \text{ is even,}\\
		0&\quad \text{if } a-i-j \text{ is odd.}
	\end{cases}$$ 
Then 
	$$\chi(a\omega_1)\chi(c\omega_1)=\sum_{i=0}^a\sum_{j=i-a}^{a-i}\delta(a,i,j)\chi(c+j,i).$$
	Moreover, if $a\leq c$, then all the weights on the right hand side of the equality are dominant.
\end{lemma}

\begin{proof}
	In this proof we use Euclidean coordinates. Like in Corollary \ref{Weyl_a0}, we set 
	$$Y(i,j):=\lfloor\frac{a-|i|-|j|}{2}\rfloor+1.$$
	Using Proposition \ref{product_characters} and Corollary \ref{Weyl_a0} in the first equality below, Lemma \ref{action_characters} in the second one and Remark \ref{B2_symmetry} in the third one, we get
	\begin{align*}
		\chi(a\omega_1)\chi(c\omega_1)=&\sum_{i=0}^a\sum_{j=i-a}^{a-i}Y(i,j)\chi(c+j,i)+\sum_{i=1}^a\sum_{j=i-a}^{a-i}Y(i,j)\chi(c+j,-i)\\
		=&\sum_{i=0}^a\sum_{j=i-a}^{a-i}Y(i,j)\chi(c+j,i)-\sum_{i=1}^a\sum_{j=i-a}^{a-i}Y(i,j)\chi(s_{\alpha_2}\sbullet(c+j,-i))\\
		=&\sum_{i=0}^a\sum_{j=i-a}^{a-i}Y(i,j)\chi(c+j,i)-\sum_{i=1}^a\sum_{j=i-a}^{a-i}Y(i,j)\chi(c+j,i-1)\\
		=&\sum_{i=0}^a\sum_{j=i-a}^{a-i}Y(i,j)\chi(c+j,i)-\sum_{i=0}^{a-1}\sum_{j=i+1-a}^{a-i-1}Y(i+1,j)\chi(c+j,i)\\
		=&\sum_{i=0}^{a-1}\sum_{j=i+1-a}^{a-i-1}(Y(i,j)-Y(i+1,j))\chi(c+j,i)\\
		&\hspace{4em}+\sum_{i=0}^{a-1}Y(i,a-i)\chi(c+a-i,i)+Y(i,i-a)\chi(c+i-a,i)\\
		&\hspace{4em}+Y(a,0)\chi(c,a).
	\end{align*}
For $i\in\{0,\ldots, a-1\}$, we have $Y(i,a-i)=Y(i,i-a)=1=\delta(a,i,\pm(a-i))$, and $Y(a,0)=1=\delta(a,a,0)$. Moreover 
$$Y(i,j)-Y(i+1,j)=(\lfloor\frac{a-|i|-|j|}{2}\rfloor+1)-(\lfloor\frac{a-|i+1|-|j|}{2}\rfloor+1)=\delta(a,i,j).$$
Therefore, we get 
\begin{align}
	\chi(a\omega_1)\chi(c\omega_1)=&\sum_{i=0}^{a-1}\sum_{j=i+1-a}^{a-i-1}\delta(a,i,j)\chi(c+j,i)\nonumber\\
	&\hspace{4em}+\sum_{i=0}^{a-1}\delta(a,i,a-i)\chi(c+a-i,i)+\delta(a,i,i-a)\chi(c+i-a,i)\nonumber\\
	&\hspace{4em}+\delta(a,a,0)\chi(c,a)\nonumber\\
	=&\sum_{i=0}^a\sum_{j=i-a}^{a-i}\delta(a,i,j)\chi(c+j,i).\label{eq50}
\end{align}
Finally, if $0\leq i$, $a\leq c$ and $i-a\leq j$, then $0\leq i \leq i+(c-a)\leq c+j$, hence all the weights appearing in line \eqref{eq50} are dominant.
\end{proof}

\begin{prop}\label{B2_ac_bis}
	Let $\lambda=(a,0),\mu=(c,0)\in X^+$ be two $p$-restricted dominant weights with $0<a< c$, $c\geq \frac{p-1}{2}$ and $a+c<p-1$ (i.e. $\lambda\in \widehat{C_1}, \mu \in C_2$ and $\lambda+\mu \in \widehat{C_2}$). Then $L(\lambda)\otimes L(\mu)$ is multiplicity-free.
\end{prop}

\begin{proof}
	In this proof we use Euclidean coordinates. We set $M:=L(\lambda)\otimes L(\mu)$ and \linebreak$m:=2c-p+2$.
	Like in Lemma \ref{B2_ac_technique}, for $i,j\in \Z$, we set $$\delta(a,i,j):=\begin{cases}
		1&\quad \text{if } a- i-j \text{ is even,}\\
		0&\quad \text{if } a-i-j \text{ is odd.}
	\end{cases}$$
	Note that $\delta(a,i,j)=\delta(a,i,-j)$.\\
	
	By Remark \ref{B2_character_symmetry} and Proposition \ref{product_characters}, and using Lemma \ref{B2_ac_technique} in the second equality below, we have
	\begin{align}
		\ch M=& \chi(a,0)(\chi(c,0)-\chi(p-3-c,0))\nonumber\\
		=&\sum_{i=0}^a\sum_{j=i-a}^{a-i}\delta(a,i,j)\chi(c+j,i)-\sum_{i=0}^a\sum_{j=i-a}^{a-i}\delta(a,i,j)\chi(p-3-c+j,i)\nonumber\\
		=&\sum_{i=0}^a\sum_{j=i-a}^{a-i}\delta(a,i,j)\chi(c+j,i)-\sum_{i=0}^a\sum_{j=i-a}^{a-i}\delta(a,i,-j)\chi(p-3-c-j,i)\nonumber\\
		=&\sum_{i=0}^a\sum_{j=\max\{i-a,a-i-m\}}^{a-i}\delta(a,i,j)\chi(c+j,i)+\sum_{i=0}^a\sum_{j=i-a}^{a-i-m-1}\delta(a,i,j)\chi(c+j,i)\nonumber\\
		&\hspace{4em}-\sum_{i=0}^a\sum_{j=\max\{i-a,a-i-m\}}^{a-i}\delta(a,i,-j)\chi(p-3-c-j,i)\nonumber\\
		&\hspace{4em}-\sum_{i=0}^a\sum_{j=i-a}^{a-i-m-1}\delta(a,i,-j)\chi(p-3-c-j,i)\nonumber\\
		=&\sum_{i=0}^a\sum_{j=\max\{i-a,a-i-m\}}^{a-i}\delta(a,i,j)(\chi(c+j,i)-\chi(p-3-c-j,i))\label{eq60}\\
		&\hspace{2em}+\sum_{i=0}^a\left(\sum_{j=i-a}^{a-i-m-1}\delta(a,i,j)\chi(c+j,i)-\sum_{j=i-a}^{a-i-m-1}\delta(a,i,j)\chi(p-3-c-j,i)\right)\label{eq61}
	\end{align}
We show that line \eqref{eq61} is equal to zero. We have
\begin{align*}
	&\sum_{j=i-a}^{a-i-m-1}\delta(a,i,j)\chi(c+j,i)-\sum_{j=i-a}^{a-i-m-1}\delta(a,i,j)\chi(p-3-c-j,i)\\
	&\hspace{2em}=\sum_{j=i-a}^{a-i-m-1}\delta(a,i,j)\chi(c+j,i)\\
	&\hspace{7em}-\sum_{j=i-a}^{a-i-m-1}\delta(a,i,-m-1-j)\chi(p-3-c-(-m-1-j),i)\\
	&\hspace{2em}=\sum_{j=i-a}^{a-i-m-1}\delta(a,i,j)\chi(c+j,i)-\sum_{j=i-a}^{a-i-m-1}\delta(a,i,-m-1-j)\chi(c+j,i)\\
	&\hspace{2em}=\sum_{j=i-a}^{a-i-m-1}(\delta(a,i,j)-\delta(a,i,-m-1-j))\chi(c+j,i)=0,
\end{align*}
where in the last equality, we use that $m$ is odd, thus $-m-1$ is even and \linebreak$\delta(a,i,-m-1-j)=\delta(a,i,j)$. Therefore
$$\sum_{i=0}^a\left(\sum_{j=i-a}^{a-i-m-1}\delta(a,i,j)\chi(c+j,i)-\sum_{j=i-a}^{a-i-m-1}\delta(a,i,j)\chi(p-3-c-j,i)\right)=\sum_{i=0}^a0=0.$$

Now we work on line \eqref{eq60}. For $0\leq i\leq a$ and $\max{\{i-a,a-i-m\}}\leq j\leq a-i$, we claim that $(c+j,i)\in \widehat{C_2}$. Indeed, we have $i\geq 0$ and 
$$(c+j)+i\leq c+(a-i)+i\leq p-2$$ by assumption. Thus it remains to show that $c+j>\frac{p-3}{2}$. 
If $i-a\geq \frac{p-1}{2}-c$, then
$$c+j\geq c+i-a\geq \frac{p-1}{2}>\frac{p-3}{2},$$
hence we are done.
If $i-a< \frac{p-1}{2}-c$, then $a-i>c-\frac{p-1}{2}$ and
$$c+j\geq c+a-i-m>2c-\frac{p-1}{2}-m=2c-\frac{p-1}{2}-(2c-p+2)=\frac{p-3}{2}.$$
so we are done. Using Remark \ref{B2_character_symmetry}, we get
\begin{align*}
	\ch M=&\sum_{i=0}^a\sum_{j=\max\{i-a,a-i-m\}}^{a-i}\delta(a,i,j)(\chi(c+j,i)-\chi(p-3-c-j,i))\\
	=&\sum_{i=0}^a\sum_{j=\max\{i-a,a-i-m\}}^{a-i}\delta(a,i,j)\ch L(c+j,i).
\end{align*}
We conclude that $M$ is multiplicity-free.
\end{proof}

\begin{prop}\label{B2_ac}
	Let $\lambda=(a,0),\mu=(c,0)\in X^+$ be $p$-restricted dominant weights with $0<a\leq c$. Then $L(\lambda)\otimes L(\mu)$ is multiplicity-free if and only if one of the following holds:
	\begin{enumerate}[label = \emph{(\arabic*)}] 
		\item $a+c\leq \frac{p-3}{2}$ (i.e. $\lambda+\mu\in \widehat{C_1}$),
		\item $c\geq \frac{p-1}{2}$ and $a+c<p-1$ or
		\item $a=c=\frac{p-1}{2}$.
	\end{enumerate}
\end{prop}

\begin{proof}
	We set $M:=L(\lambda)\otimes L(\mu)$. 
	
	If $a+c\geq p$, we use Corollary \ref{sum_p_res} to conclude that $M$ has multiplicity.\\
	
	If $a+c\leq \frac{p-3}{2}$ (i.e. $\lambda+\mu\in \widehat{C_1}$), we apply Corollary \ref{sum_C1} and Theorem \ref{B2_p0} to conclude that $M$ is multiplicity-free.\\
	
	Suppose that $c\leq \frac{p-3}{2}$ and $a+c>\frac{p-3}{2}$ (i.e. $\mu\in \widehat{C_1}$ and $\lambda+\mu \in C_2$). In this case, $L(\lambda)$ and $L(\mu)$ are tilting modules by Lemma \ref{Weyl_B2}, hence $M$ is a tilting module. Since $T(\lambda+\mu)$ is not irreducible by Lemma \ref{Weyl_B2}, we conclude that $M$ has multiplicity by Lemma \ref{argument_tilting}.\\
	
	Consider the case $a=1, c=p-2$ (in particular $\mu\in F_{2,3}$, $\lambda\in \widehat{C_1}$ and $\lambda+\mu\in F_{4,6}$). By Lemma \ref{Weyl_B2}, we have $L(\lambda)$ and $L(\mu)$ are tilting modules, therefore $M$ is a tilting module. We use Argument \ref{argument1}. By Lemma \ref{-1}, we have
	\begin{align*}
		&m_{L(\lambda)}(\lambda-\alpha_1)=1,
		&&m_{L(\mu)}(\mu-\alpha_1)=1,
		&&m_{L(\lambda+\mu)}(\lambda+\mu-\alpha_1)=1.
	\end{align*}
Since $m_M(\lambda+\mu-\alpha_1)=2$, we deduce that $L(\lambda+\mu-\alpha_1)$ is a composition factor of $M$. Observe that $\lambda+\mu-\alpha_1=(p-3,2)\in C_3$, so $T(\lambda+\mu-\alpha_1)$ is not irreducible by Lemma \ref{Weyl_B2}. We conclude by Lemma \ref{argument_tilting} that $M$ has multiplicity. \\

If $a+c=p-1$ and $a\neq 1, \frac{p-1}{2}$, then $\frac{p-1}{2}< c<p-2$ (i.e. $\lambda\in \widehat{C_1}$ and $\mu\in C_2$) and by Lemma \ref{Weyl_B2}, we have $L(\lambda)\cong \Delta(\lambda)$. Moreover, $\Delta(\mu)$ admits the unique composition series $[L(\mu),L(\eta)]$ with $\eta=s_{\alpha_1+\alpha_2,p}\sbullet \mu\in C_1$. Since $c\geq \frac{p+1}{2}$, we have in particular \linebreak$\eta <\mu-2\alpha_1-2\alpha_2$, thus $m_{L(\mu)}(\nu)=m_{\Delta(\mu)}(\nu)$ for all weights $\nu\geq \lambda-2\alpha_1-2\alpha_2$.
We use Argument \ref{argument1} to show that $[M:L(\lambda+\mu-2\alpha_1-2\alpha_2)]=2$. Using Lemmas \ref{table_B2} and \ref{-1} (note that $a\geq 2$), we have
\begin{align*}
	&m_M(\lambda+\mu)=1,&&m_M(\lambda+\mu-\alpha_1)=2,\\
	&m_M(\lambda+\mu-2\alpha_1)=3, &&m_M(\lambda+\mu-\alpha_2)=m_M(\lambda+\mu-2\alpha_2)=0,\\
	&m_M(\lambda+\mu-\alpha_1-\alpha_2)=2,&&m_M(\lambda+\mu-\alpha_1-2\alpha_2)=2,\\
	&m_M(\lambda+\mu-2\alpha_1-\alpha_2)=4,&&m_M(\lambda+\mu-2\alpha_1-2\alpha_2)=7.
\end{align*}
We deduce that 
\begin{align*}
	&[M:L(\lambda+\mu)]=[M:L(\lambda+\mu-\alpha_1)]=[M:L(\lambda+\mu-2\alpha_1)]=1,\\
	&[M:L(\lambda+\mu-\alpha_2)]=[M:L(\lambda+\mu-2\alpha_2)]=0.
\end{align*}
By Lemma \ref{Weyl_B2}, $L(\lambda+\mu)\cong \Delta(\lambda+\mu)$. Moreover, $\lambda+\mu-\alpha_1=(p-3,2)$. Thus, using \linebreak Lemmas \ref{-1} and \ref{table_B2}, we get
\begin{align*}
	&[M:L(\lambda+\mu-\alpha_1-\alpha_2)]=0, \qquad\text{and} \qquad[M:L(\lambda+\mu-\alpha_1-2\alpha_2)]=0.
\end{align*} 
Since $2(p-3)+2+2\not\equiv 0\bmod p$, we have $m_{L(\lambda+\mu-\alpha_1)}(\lambda+\mu-2\alpha_1-\alpha_2)=2$ by Lemma \ref{B2_-11}.\linebreak By Lemma \ref{-1}, we have $m_{L(\lambda+\mu-2\alpha_1)}(\lambda+\mu-2\alpha_1-\alpha_2)=1$. Moreover, by Lemma \ref{table_B2}, we have $m_{L(\lambda+\mu)}(\lambda+\mu-2\alpha_1-\alpha_2)=1$. Therefore, 
$$[M:L(\lambda+\mu-2\alpha_1-\alpha_2)]=0.$$
Finally observe that $(p-3)+2=p-1$. Thus, using Lemmas \ref{table_B2}, \ref{B2_-12} and \ref{-1}, we have
\begin{align*}
	&m_{L(\lambda+\mu)}(\lambda+\mu-2\alpha_1-2\alpha_2)=2,&&m_{L(\lambda+\mu-\alpha_1)}(\lambda+\mu-2\alpha_1-2\alpha_2)=2,\\
	&m_{L(\lambda+\mu-2\alpha_1)}(\lambda+\mu-2\alpha_1-2\alpha_2)=1.
\end{align*}
We conclude that 
$$[M:L(\lambda+\mu-2\alpha_1-2\alpha_2)]=7-2-2-1=2.$$
In particular, $M$ has multiplicity.\\

If $c\geq \frac{p-1}{2}$ and $a+c<p-1$ then $M$ is multiplicity-free by Proposition \ref{B2_ac_bis}.\\

Finally, suppose that $a=c=\frac{p-1}{2}$ (i.e. $\lambda,\mu\in C_2$ and $\lambda+\mu\in F_{4,6}$). We show that $M$ is multiplicity-free. We have $$s_{\alpha_1+\alpha_2,p}\sbullet \mu=\mu-2(\alpha_1+\alpha_2).$$ 
For the rest of this proof we use Euclidean coordinates. Using Corollary \ref{products_characters_bis} and Lemma \ref{Weyl_B2}\linebreak in the first equality below, Corollary \ref{irred_a0} in the second one and Lemma \ref{action_characters} in the third one, we get
\begin{align*}
	\ch M=&\sum_{\nu\in X}m_{L(\lambda)}(\nu)(\chi(\mu+\nu)-\chi(\mu-2(\alpha_1+\alpha_2)+\nu))\\
	=&\sum_{i=0}^a\sum_{j=i-a}^{a-i}\chi(a+j,i)-\chi(a-2+j,i)+\sum_{i=1}^a\sum_{j=i-a}^{a-i}\chi(a+j,-i)-\chi(a-2+j,-i)\\
	=&\sum_{i=0}^a\sum_{j=i-a}^{a-i}\chi(a+j,i)-\sum_{i=1}^a\sum_{j=i-a}^{a-i}\chi(s_{\alpha_2}\sbullet(a+j,-i))\\
	&-\sum_{i=0}^a\sum_{j=i-a}^{a-i}\chi(a-2+j,i)+\sum_{i=1}^a\sum_{j=i-a}^{a-i}\chi(s_{\alpha_2}\sbullet(a-2+j,-i))\\
	=&\sum_{i=0}^a\sum_{j=i-a}^{a-i}\chi(a+j,i)-\sum_{i=1}^a\sum_{j=i-a}^{a-i}\chi(a+j,i-1)\\
	&-\sum_{i=0}^a\sum_{j=i-a}^{a-i}\chi(a-2+j,i)+\sum_{i=1}^a\sum_{j=i-a}^{a-i}\chi(a-2+j,i-1)\\
	=&\sum_{i=0}^a\sum_{j=i-a}^{a-i}\chi(a+j,i)-\sum_{i=0}^{a-1}\sum_{j=i+1-a}^{a-i-1}\chi(a+j,i)\\
	&-\sum_{i=0}^a\sum_{j=i-a}^{a-i}\chi(a-2+j,i)+\sum_{i=0}^{a-1}\sum_{j=i+1-a}^{a-i-1}\chi(a-2+j,i)\\
	=&\left(\sum_{i=0}^{a-1}\chi(i,i)+\chi(2a-i,i)\right)+\chi(a,a)\\
	&-\left(\sum_{i=0}^{a-1}\chi(i-2,i)+\chi(2a-2-i,i)\right)-\chi(a-2,a)\\
	=&\sum_{i=0}^a\chi(i,i)+\chi(2a,0)+\sum_{i=1}^{a-1}\chi(2a-i,i)\\
	&-\sum_{i=0}^{a-2}\chi(2a-2-i,i)-\chi(a-1,a-1)-\chi(-2,0)-\sum_{i=1}^{a-1}\chi(i-2,i)-\chi(a-2,a).
\end{align*}
At this step, recall that we use Euclidean coordinates and observe that $(i,i)\in \widehat{C_1}\cup F_{2,3}\cup F_{3,5}$\linebreak for $0\leq i\leq a=\frac{p-1}{2}$. Thus, by Lemma \ref{Weyl_B2}, we have $\chi(i,i)=\ch L(i,i)$. Similarly, \linebreak$(2a,0)=(p-1,0)\in F_{4,6}$, hence $\chi(2a,0)=\ch L(2a,0)$ by Lemma \ref{Weyl_B2}. Using those facts and Lemma \ref{action_characters}, we get
\begin{align*}
	\ch M=&\sum_{i=0}^a\ch L(i,i)+\ch L(2a,0)+\sum_{i=1}^{a-1}\chi(2a-i,i)-\sum_{i=1}^{a-1}\chi(2a-1-i,i-1)\\
	&+\sum_{i=1}^{a-1}\chi(s_{\alpha_1}\sbullet(i-2,i))+\chi(s_{\alpha_1+\alpha_2}\sbullet(-2,0))-\chi(a-1,a-1)+\chi(s_{\alpha_1}\sbullet(a-2,a))\displaybreak\\
	=&\sum_{i=0}^a\ch L(i,i)+\ch L(2a,0)+\sum_{i=1}^{a-1}\chi(2a-i,i)-\sum_{i=1}^{a-1}\chi(2a-1-i,i-1)\\
	&+\sum_{i=1}^{a-1}\chi(i-1,i-1)+\chi(-1,0)-\chi(a-1,a-1)+\chi(a-1,a-1)\\
	=&\sum_{i=0}^a\ch L(i,i)+\ch L(2a,0)+\sum_{i=1}^{a-1}\chi(2a-i,i)\\
	&-\left(\sum_{i=1}^{a-1}\chi(2a-1-i,i-1)-\chi(i-1,i-1)\right),
\end{align*}
where in the last equality we use that $(-1,0)\in D\setminus X^+$ so $\chi(-1,0)=0$ by Lemma \ref{action_characters}. If $i\in \{1,\ldots,a-1\}$, then $(2a-1-i,i-1)\in C_2$. Thus, by Remark \ref{B2_symmetry}, we have 
$$s_{\alpha_1+\alpha_2,p}\sbullet (2a-1-i,i-1)=(p-3-(2a-1-i),i-1)=(i-1,i-1),$$
and by Remark \ref{B2_character_symmetry} we get
$$\chi(2a-1-i,i-1)-\chi(i-1,i-1)=\ch L(2a-1-i,i-1).$$
Moreover, $(2a-i,i)\in C_3$ and $2a-i+i-(p-2)=1$, so by Remark \ref{B2_character_symmetry}, we have
$$\chi(2a-i,i)=\ch L(2a-i,i)+\ch L(2a-i-1,i-1).$$
We can thus conclude that
\begin{align*}
	\ch M=&\sum_{i=0}^a\ch L(i,i)+\ch L(2a,0)+\sum_{i=1}^{a-1}\chi(2a-i,i)-\sum_{i=1}^{a-1}\ch L(2a-1-i,i-1)\\
	=&\sum_{i=0}^a\ch L(i,i)+\ch L(2a,0)+\sum_{i=1}^{a-1}\ch L(2a-i,i)\\
	=&\sum_{i=0}^a\ch L(i,i)+\sum_{i=0}^{a-1}\ch L(2a-i,i).
\end{align*}
In particular, $M$ is multiplicity-free.
\end{proof}

\subsubsection{$L(a,0)\otimes L(0,d)$}

\begin{lemma}\label{B2_ad}
	Let $a,b\in \N$. In Euclidean coordinates, we have
	$$\chi(a\omega_1)\chi(b\omega_2)= \sum_{i=0}^b\sum_{j=0}^{\min\{i,b-i\}}\chi(a+\frac b2-i,\frac b2-j).$$
	Moreover, if $b<p$ and $a<p$, then all weights appearing on the right hand side of the equality belong to $D\cup W\sbullet\overline{C_1}$ and the dominant ones are $p$-restricted.
\end{lemma}

\begin{proof}
	In this proof, we use Euclidean coordinates. Let $\lambda=a\omega_1$ and $\mu=b\omega_2$. We set $X(i,j):=m_{\Delta(\mu)}(\frac b2-i,\frac b2-j)=\frac b2 -\max\{|\frac b2-i|,|\frac b2-j|\}+1$ (Lemma \ref{B2_weight_0b}). We get
	\begin{align*}
		\chi(\lambda)\chi(\mu)=&\sum_{\nu\in X}m_{\Delta(\mu)}(\nu)\chi(\lambda+\nu)
		=\sum_{i=0}^b\sum_{j=0}^bX(i,j)\chi(a+\frac b2-i,\frac b2-j).
	\end{align*}
	We dissociate the cases $b$ even and $b$ odd. First suppose that $b$ is even. Using Lemma \ref{action_characters} in the second equality below and Remark \ref{B2_symmetry} in the third one, we get 
	\begin{align}
		\chi(\lambda)\chi(\mu)=&\sum_{i=0}^b\sum_{j=0}^{\frac b2}X(i,j)\chi(a+\frac b2-i,\frac b2-j)+\sum_{i=0}^b\sum_{j=\frac b2 +1}^bX(i,j)\chi(a+\frac b2-i,\frac b2-j)\nonumber\\
		=&\sum_{i=0}^b\sum_{j=0}^{\frac b2}X(i,j)\chi(a+\frac b2-i,\frac b2-j)-\sum_{i=0}^b\sum_{j=\frac b2 +1}^bX(i,j)\chi(s_{\alpha_2}\sbullet(a+\frac b2-i,\frac b2-j))\nonumber\\
		=&\sum_{i=0}^b\sum_{j=0}^{\frac b2}X(i,j)\chi(a+\frac b2-i,\frac b2-j)-\sum_{i=0}^b\sum_{j=\frac b2 +1}^bX(i,j)\chi(a+\frac b2-i,-\frac b2+j-1)\nonumber\\
		=&\sum_{i=0}^b\sum_{j=0}^{\frac b2}X(i,j)\chi(a+\frac b2-i,\frac b2-j)\nonumber\\
		&\hspace{6em}-\sum_{i=0}^b\sum_{j=1}^{\frac b2}X(i,b-j+1)\chi(a+\frac b2-i,-\frac b2+(b-j+1)-1)\nonumber\\
		=&\sum_{i=0}^b\sum_{j=0}^{\frac b2}X(i,j)\chi(a+\frac b2-i,\frac b2-j)-\sum_{i=0}^b\sum_{j=1}^{\frac b2}X(i,b-j+1)\chi(a+\frac b2-i,\frac b2-j)\nonumber\\
		=&\sum_{i=0}^bX(i,0)\chi(a+\frac b2 -i,\frac b2)\nonumber\\
		&\hspace{6em}+\sum_{i=0}^b\sum_{j=1}^{\frac b2}(X(i,j)-X(i,b-j+1))\chi(a+\frac b2-i,\frac b2-j).\label{eqeven}
	\end{align}
	
	If $b$ is odd, then for any $i$, we have 
	$$s_{\alpha_2}\sbullet(a+\frac b2-i,\frac b2- \frac{b+1}{2})=s_{\alpha_2}\sbullet(a+\frac b2-i,-\frac 12)=(a+\frac b2-i,-\frac 12),$$
	so $\chi(a+\frac b2-i,\frac b2- \frac{b+1}{2})=0$ by Lemma \ref{action_characters}. Similarly to the case where $b$ is even, we get
	\begin{align}
		\chi(\lambda)\chi(\mu)=&\sum_{i=0}^b\sum_{j=0}^{\frac{b-1}2}X(i,j)\chi(a+\frac b2-i,\frac b2-j)+\sum_{i=0}^b\sum_{j=\frac{b+3}2 }^bX(i,j)\chi(a+\frac b2-i,\frac b2-j)\nonumber\\
		=&\sum_{i=0}^b\sum_{j=0}^{\frac{b-1}2}X(i,j)\chi(a+\frac b2-i,\frac b2-j)-\sum_{i=0}^b\sum_{j=\frac{b+3}2}^bX(i,j)\chi(s_{\alpha_2}\sbullet(a+\frac b2-i,\frac b2-j))\nonumber\\
		=&\sum_{i=0}^b\sum_{j=0}^{\frac{b-1}2}X(i,j)\chi(a+\frac b2-i,\frac b2-j)-\sum_{i=0}^b\sum_{j=\frac{b+3}2}^bX(i,j)\chi(a+\frac b2-i,-\frac b2+j-1)\nonumber\displaybreak\\
		=&\sum_{i=0}^b\sum_{j=0}^{\frac{b-1}2}X(i,j)\chi(a+\frac b2-i,\frac b2-j)\nonumber\\
		&\hspace{6em}-\sum_{i=0}^b\sum_{j=1}^{\frac{b-1}2}X(i,b-j+1)\chi(a+\frac b2-i,-\frac b2+(b-j+1)-1)\nonumber\\
		=&\sum_{i=0}^b\sum_{j=0}^{\frac{b-1}2}X(i,j)\chi(a+\frac b2-i,\frac b2-j)-\sum_{i=0}^b\sum_{j=1}^{\frac{b-1}2}X(i,b-j+1)\chi(a+\frac b2-i,\frac b2-j)\nonumber\\
		=&\sum_{i=0}^bX(i,0)\chi(a+\frac b2 -i,\frac b2)\nonumber\\
		&\hspace{6em}+\sum_{i=0}^b\sum_{j=1}^{\frac{b-1}2}(X(i,j)-X(i,b-j+1))\chi(a+\frac b2-i,\frac b2-j).\label{eqodd}
	\end{align}
	At this step, observe that $X(i,0)=1$, and set $Y(i,j):=X(i,j)-X(i,b-j+1)$. Using \linebreak lines \eqref{eqeven} and \eqref{eqodd}, for all $b$, we have
	\begin{align}
		\chi(\lambda)\chi(\mu)=&\sum_{i=0}^b\chi(a+\frac b2 -i,\frac b2)+\sum_{i=0}^b\sum_{j=1}^{\lfloor\frac b2\rfloor}Y(i,j)\chi(a+\frac b2-i,\frac b2-j)\nonumber\\
		=&\sum_{i=0}^b\chi(a+\frac b2 -i,\frac b2)\nonumber\\
		&\hspace{5em}+\sum_{i=0}^{\lfloor\frac b2\rfloor}\sum_{j=1}^{\lfloor\frac b2\rfloor}Y(i,j)\chi(a+\frac b2-i,\frac b2-j)\label{eq42}\\
		&\hspace{5em}+\sum_{i=\lfloor\frac b2\rfloor+1}^{b}\sum_{j=1}^{\lfloor\frac b2\rfloor}Y(i,j)\chi(a+\frac b2-i,\frac b2-j).\label{eq43}
	\end{align}
	We compute the value of $Y(i,j)$ in lines \eqref{eq42} and \eqref{eq43}. Recall that \linebreak$Y(i,j)=X(i,j)-X(i,b-j+1)$ and $X(i,j)=\frac b2 -\max\{|\frac b2-i|,|\frac b2-j|\}+1$. Therefore, for $j\leq\frac b2$, we have 
	\begin{align*}
		Y(i,j)&=(\frac b2 -\max\{|\frac b2-i|,|\frac b2-j|\}+1)-(\frac b2 -\max\{|\frac b2-i|,|\frac b2-(b-j+1)|\}+1)\\
		&=\max\{|\frac b2-i|,|j-\frac b2-1|\}-\max\{|\frac b2-i|,|\frac b2-j|\}\\
		&=\max\{|\frac b2-i|,|\frac b2-j+1|\}-\max\{|\frac b2-i|,|\frac b2-j|\}\\
		&=\max\{|\frac b2-i|,\frac b2-j+1\}-\max\{|\frac b2-i|,\frac b2-j\}.
	\end{align*}
	
	If $|\frac b2-i|\geq\frac b2-j+1$, then 
	$$\max\{|\frac b2-i|,\frac b2-j+1\}=\max\{|\frac b2-i|,\frac b2-j\}=|\frac b2-i|$$
	so $Y(i,j)=0$.
	
	If $|\frac b2-i|<\frac b2-j+1$, then $|\frac b2-i|\leq\frac b2-j$ and we have
	$$\max\{|\frac b2-i|,\frac b2-j+1\}=\frac b2-j+1\qquad \text{ and }\qquad\max\{|\frac b2-i|,\frac b2-j\}=\frac b2-j,$$
	so $Y(i,j)=1$. Therefore, we have 
	\begin{ceqn}
		\begin{align}
			Y(i,j)=\begin{cases}
				0 \quad &\text{if  }\;|\frac b2-i|\geq\frac b2-j+1\\
				1\quad &\text{if  }\;|\frac b2-i|<\frac b2-j+1.
			\end{cases} \label{eq44}
		\end{align}
		If $i\leq\frac b2$, as in line \eqref{eq42}, then
		\begin{align}
			|\frac b2-i|=\frac b2-i\geq\frac b2-j+1 \iff j\geq i+1 \label{eq45}.
		\end{align}
		If $i>\frac b2$, as in line \eqref{eq43}, then
		\begin{align}
			|\frac b2-i|=i-\frac b2\geq\frac b2-j+1 \iff j\geq b-i+1 \label{eq46}.
		\end{align}
	\end{ceqn}
	Combining lines \eqref{eq44}, \eqref{eq45} and \eqref{eq46} with lines \eqref{eq42} and \eqref{eq43}, we get
	
	\begin{align}
		\chi(\lambda)\chi(\mu)=&\sum_{i=0}^b\chi(a+\frac b2 -i,\frac b2)+\sum_{i=0}^{\lfloor\frac b2\rfloor}\sum_{j=1}^{i}\chi(a+\frac b2-i,\frac b2-j)\nonumber\\
		&\hspace{4em}+\sum_{i=\lfloor\frac b2\rfloor+1}^{b}\;\sum_{j=1}^{b-i}\chi(a+\frac b2-i,\frac b2-j)\nonumber\\
		=&\sum_{i=0}^b\chi(a+\frac b2 -i,\frac b2)+\sum_{i=0}^{b}\sum_{j=1}^{\min\{i,b-i\}}\chi(a+\frac b2-i,\frac b2-j)\nonumber\\
		=&\sum_{i=0}^{b}\sum_{j=0}^{\min\{i,b-i\}}\chi(a+\frac b2-i,\frac b2-j),\label{eq48}
	\end{align}
	establishing the claim of the first statement.

	Finally, assume that $a<p$ and $b<p$. We do a change of basis to express our weights in coordinates with respect to the fundamental weights. We have
	$$(a+\frac b2-i)\epsilon_1+(\frac b2-j)\epsilon_2=(a-i+j)\omega_1+(b-2j)\omega_2.$$
	For the rest of the proof, we use coordinates with respect to the fundamental weights.
	Let $0\leq i\leq b<p$ and $0\leq j\leq \min\{i,b-i\}$. We set $\nu:=(a-i+j,b-2j)$. We have
	$$a-i+j\leq a<p \qquad\text{and} \qquad b-2j\leq b<p,$$ 
	so dominant weights appearing in line \eqref{eq48} are $p$-restricted. Moreover, $j\leq \frac b2$, thus $b-2j\geq 0$.
	If $(a-i+j)\geq -1$, then $\nu\in D$ and we are done. Assume $(a-i+j)< -1$. We have
	$$s_{\alpha_1}\sbullet (a-i+j,b-2j)=(i-a-j-2,b+2a-2i+2)=:\nu'.$$
	By assumption, we have 
	$$i-a-j-2=-(a-i+j)-2>1-2=-1.$$
	If $b+2a-2i+2\geq -1$, then $\nu'\in D$ and 
	$$2(i-a-j-2)+(b+2a-2i+2)=b-2j-2\leq b-2\leq p-3,$$
	so $\nu'\in \overline{C_1}$ and we are done.
	
	Finally, suppose $b+2a-2i+2<-1$. Then we have 
	$$s_{\alpha_2}\sbullet(i-a-j-2,b+2a-2i+2)=(b+a-i-j+1,2i-b-2a-4)=:\eta.$$
	By assumption, we have 
	$$2i-b-2a-4=-(b+2a-2i+2)-2>1-2=-1$$
	and 
	$$b+a-i-j+1=(b-i)-j+a+1\geq a+1>0,$$
	hence $\eta \in D$. Moreover, we have
	$$2(b+a-i-j+1)+(2i-b-2a-4)=b-2j-2\leq p-3$$
	so $\eta\in \overline{C_1}$. Therefore in all cases we are done.
\end{proof}

\begin{prop}
	Let $\lambda=(p-2,0),\mu=(0,d)\in X^+$ be $p$-restricted with $d>0$. Then $L(\lambda)\otimes L(\mu)$ is multiplicity-free if and only if $d=1$.
\end{prop}

\begin{proof}
	We set $M:=L(\lambda)\otimes L(\mu)$. Observe that $\lambda\in F_{2,3}$, so $\Delta(\lambda)\cong L(\lambda)$ by Lemma \ref{Weyl_B2}.\\
	
	Suppose that $d=1$. Using Corollary \ref{products_characters_bis} and the structure of $L(0,1)$, we have
	\begin{align*}
		\ch M=&\sum_{\nu\in X}m_{L(\mu)}(\nu)\chi(\lambda+\nu)\\
		=&\chi(p-2,1)+\chi(p-3,1)+\chi(p-2,-1)+\chi(p-1,-1).
	\end{align*}
Observe that $(p-2,1)\in F_{3,4}$ and $(p-3,1)\in F_{2,3}$ so $\chi(p-2,1)=\ch L(p-2,1)$ and $\chi(p-3,1)=\ch L(p-3,1)$ by Lemma \ref{Weyl_B2}. Moreover, $(p-2,-1),(p-1,-1)\in D\setminus X^+$ so $\chi(p-2,-1)=\chi(p-1,-1)=0$ by Lemma \ref{action_characters}. Thus we get
$$\ch M=\ch L(p-2,1)+\ch L(p-3,1).$$
In particular, $M$ is multiplicity-free.\\

Now suppose that $d>1$. In particular, $\lambda+\mu \in C_4\cup( F_{4,7}\setminus F_{4,6})$. By Lemma \ref{Weyl_B2}, both $L(\lambda)$ and $L(\mu)$ are tilting modules, so $M$ is a tilting module. Since $L(\lambda+\mu)$ is a composition factor of $M$ but $T(\lambda+\mu)$ is not irreducible, we conclude by Lemma \ref{action_characters} that $M$ has multiplicity. 
\end{proof}

\begin{prop}\label{B2_ad_p-1}
	Let $\lambda=(p-1,0),\mu=(0,d)\in X^+$ be $p$-restricted with $d>0$. Then $L(\lambda)\otimes L(\mu)$ is multiplicity-free if and only if $d=1$.
\end{prop}

\begin{proof}
	We set $M:=L(\lambda)\otimes L(\mu)$. Observe that $\lambda\in F_{4,6}$, so $\Delta(\lambda)\cong L(\lambda)$ by Lemma \ref{Weyl_B2}. Moreover, $\lambda+\mu \in F_{4,6}$.\\
	
	Suppose that $d=1$. Using Corollary \ref{products_characters_bis} and the structure of $L(0,1)$, we get
	\begin{align*}
		\ch M=&\sum_{\nu\in X}m_{L(\mu)}(\nu)\chi(\lambda+\nu)\\
		=&\chi(p-1,1)+\chi(p-2,1)+\chi(p-1,-1)+\chi(p,-1).
	\end{align*}
Observe that $(p-1,1)\in F_{4,6}$ and $(p-2,1)\in F_{3,4}$ so $\chi(p-1,1)=\ch L(p-1,1)$ and $\chi(p-2,1)=\ch L(p-2,1)$ by Lemma \ref{Weyl_B2}. Moreover, $(p-1,-1),(p,-1)\in D\setminus X^+$ so $\chi(p-1,-1)=\chi(p,-1)=0$ by Lemma \ref{action_characters}. Thus we get
$$\ch M=\ch L(p-1,1)+\ch L(p-2,1).$$
In particular, $M$ is multiplicity-free.\\

	Now suppose that $d>1$. By Lemma \ref{Weyl_B2}, both $L(\lambda)$ and $L(\mu)$ are tilting modules, so $M$ is a tilting module. We use Argument \ref{argument1} to show that $L(\lambda+\mu-\alpha_1-\alpha_2)$ is a composition factor of $M$. Using Lemmas \ref{B2_weight_a0}, \ref{B2_weight_0b}, \ref{-1} and \ref{B2_-11}, we have
	\begin{align*}
		&m_{L(\lambda)}(\lambda-\alpha_1)=1,
		&&m_{L(\lambda)}(\lambda-\alpha_1-\alpha_2)=1,\\
		&m_{L(\mu)}(\mu-\alpha_2)=1,
		&&m_{L(\mu)}(\mu-\alpha_1-\alpha_2)=1,\\
		&m_{L(\lambda+\mu)}(\lambda+\mu-\alpha_1)=1,
		&&m_{L(\lambda+\mu)}(\lambda+\mu-\alpha_2)=1,\\
		&m_{L(\lambda+\mu)}(\lambda+\mu-\alpha_1-\alpha_2)=2.
	\end{align*}
	Therefore $m_M(\lambda+\mu-\alpha_1)=m_M(\lambda+\mu-\alpha_2)=1$ and $L(\lambda+\mu-\alpha_1),\; L(\lambda+\mu-\alpha_2)$ are not composition factors of $M$. Moreover, $m_M(\lambda+\mu-\alpha_1-\alpha_2)=3$, thus $L(\lambda+\mu-\alpha_1-\alpha_2)$ is a composition factor of $M$.
	Observe that $\lambda+\mu-\alpha_1-\alpha_2=(p-2,d)\in C_4$. By Lemma \ref{Weyl_B2}, $T(\lambda+\mu-\alpha_1-\alpha_2)$ is not irreducible and by Lemma \ref{argument_tilting}, $M$ has multiplicity.
\end{proof}

\begin{prop}
	Let $\lambda=(a,0),\mu=(0,d)\in X^+$ with $2a+d\leq p-3$ (i.e. $\lambda+\mu \in \widehat{C_1}$). Then $L(\lambda)\otimes L(\mu)$ is multiplicity-free.
\end{prop}

\begin{proof}
	We apply Corollary \ref{sum_C1} and Theorem \ref{B2_p0} to conclude that $L(\lambda)\otimes L(\mu)$ is multiplicity-free.
\end{proof}

\begin{prop}
	Let $\lambda=(a,0),\mu=(0,d)\in X^+$ with $0<a\leq \frac{p-3}2$, $ 0<d\leq p-3$ and $2a+d> p-3$ (i.e. $\lambda,\mu\in \widehat{C_1}$ and $\lambda+\mu\notin \widehat{C_1}$). Then $L(\lambda)\otimes L(\mu)$ is multiplicity-free if and only if $(a,d)=(1,p-3)$.
\end{prop}

\begin{proof}
	We set $M:=L(\lambda)\otimes L(\mu)$. Observe that $\lambda+\mu\in C_2\cup C_3\cup F_{2,3}$. Moreover, $L(\lambda)$ and $L(\mu)$ are tilting modules, so $M$ is a tilting module.\\
	
	Suppose that $a+d\neq p-2$ (i.e. $\lambda+\mu\notin F_{2,3}$, so $\lambda+\mu \in C_2\cup C_3$). The tilting module $T(\lambda+\mu)$ is thus not irreducible, and since $L(\lambda+\mu)$ is a composition factor of $M$, we conclude by Lemma \ref{argument_tilting} that $M$ has multiplicity.\\
	
	Suppose that $a+d=p-2$ and $a\neq 1$. Using the same argument as in the proof of Proposition \ref{B2_ad_p-1}, we get that $L(\lambda+\mu-\alpha_1-\alpha_2)$ is a composition factor of $M$.
	Observe that $\lambda+\mu-\alpha_1-\alpha_2=(a-1,d)\in C_2$ so $T(\lambda+\mu-\alpha_1-\alpha_2)$ is not irreducible. We conclude by Lemma \ref{argument_tilting} that $M$ has multiplicity.\\
	
	Finally, suppose that $(a,d)=(1,p-3)$. Using Proposition \ref{product_characters} and the structure of $L(1,0)$, we get
	\begin{align*}
		\ch M=&\sum_{\nu \in X}m_{L(1,0)}(\nu)\chi(\mu+\nu)\\
		=&\chi(1,p-3)+\chi(-1,p-3)+\chi(-1,p-1)+\chi(0,p-3)+\chi(1,p-5)
	\end{align*}
Observe that $(1,p-3)\in F_{2,3}$ and $(0,p-3),(1,p-5)\in F_{1,2}$, so $\chi(1,p-3)=\ch L(1,p-3)$, $\chi(0,p-3)=\ch L(0,p-3)$ and $\chi(1,p-5)=\ch L(1,p-5)$ by Lemma \ref{Weyl_B2}. Moreover, $(-1,p-3),(-1,p-1)\in D\setminus X^+$ so $\chi(-1,p-3)=\chi(-1,p-1)=0$ by Lemma \ref{action_characters}. Thus we get
$$\ch M=\ch L(1,p-3)+\ch L(0,p-3)+\ch L(1,p-5).$$
In particular, $M$ is multiplicity-free.
\end{proof}

\begin{prop}
	Let $\lambda=(a,0),\mu=(0,p-2)\in X^+$ with $0<a\leq \frac{p-3}2$ (i.e. $\lambda\in \widehat{C_1}$ and $\mu\in F_{2,3}$). Then $L(\lambda)\otimes L(\mu)$ has multiplicity.
\end{prop}

\begin{proof}
	By Lemma \ref{Weyl_B2}, both $L(\lambda)$ and $L(\mu)$ are tilting modules, thus $L(\lambda)\otimes L(\mu)$ is a tilting module. Moreover, $\lambda+\mu\in C_3$, hence $T(\lambda+\mu)$ is not irreducible. Since $L(\lambda+\mu)$ is a composition factor of $L(\lambda)\otimes L(\mu)$, we conclude by Lemma \ref{argument_tilting}.
\end{proof}

\begin{prop}
	Let $\lambda=(a,0),\mu=(0,p-1)\in X^+$ with $0<a\leq \frac{p-3}2$ (i.e. $\lambda\in \widehat{C_1}$ and $\mu\in F_{3,5}$). Then $L(\lambda)\otimes L(\mu)$ is multiplicity-free if and only if $a=1$.
\end{prop}

\begin{proof}
	We set $M:=L(\lambda)\otimes L(\mu)$. Observe that $L(\lambda)$ and $L(\mu)$ are tilting modules, thus $M$ is a tilting module.
	
	Suppose that $a=1$. Using Corollary \ref{products_characters_bis} and the structure of $L(1,0)$, we have
	\begin{align*}
		\ch M=&\sum_{\nu \in X}m_{L(1,0)}(\nu)\chi(\mu+\nu)\\
		=&\chi(1,p-1)+\chi(-1,p-1)+\chi(-1,p+1)+\chi(0,p-1)+\chi(1,p-3).
	\end{align*}
Observe that $(1,p-1),(0,p-1)\in F_{3,5}$ and $(1,p-3)\in F_{2,3}$, so $\chi(1,p-1)=\ch L(1,p-1)$, $\chi(0,p-1)=\ch L(0,p-1)$ and $\chi(1,p-3)=\ch L(1,p-3)$ by Lemma \ref{Weyl_B2}. Moreover, $(-1,p-1),(-1,p+1)\in D\setminus X^+$ so $\chi(-1,p-1)=\chi(-1,p+1)=0$ by Lemma \ref{action_characters}. Thus we get
$$\ch M=\ch L(1,p-1)+\ch L(0,p-1)+\ch L(1,p-3).$$
In particular, $M$ is multiplicity-free.\\

	Now suppose $a>1$. By Lemma \ref{Weyl_B2}, we have $L(\lambda)\cong \Delta(\lambda)$, $L(\mu)\cong \Delta(\mu)$ and \linebreak $L(\lambda+\mu)\cong \Delta(\lambda+\mu)$. We use Argument \ref{argument1} to show that $L(\lambda+\mu-\alpha_1-2\alpha_2)$ is a composition factor of $M$. By Lemma \ref{table_B2}, we have
	\begin{align*}
		&m_{L(\lambda)}(\lambda-\alpha_1)=1,
		&&m_{L(\lambda)}(\lambda-\alpha_1-\alpha_2)=1,\\
		&m_{L(\lambda)}(\lambda-\alpha_1-2\alpha_2)=1,
		&&m_{L(\mu)}(\mu-\alpha_2)=1,\\
		&m_{L(\mu)}(\mu-\alpha_1-\alpha_2)=1,
		&&m_{L(\mu)}(\mu-\alpha_1-2\alpha_2)=2,\\
		&m_{L(\mu)}(\mu-2\alpha_2)=1,
		&&m_{L(\lambda+\mu)}(\lambda+\mu-\alpha_1)=1,\\
		&m_{L(\lambda+\mu)}(\lambda+\mu-\alpha_2)=1,
		&&m_{L(\lambda+\mu)}(\lambda+\mu-\alpha_1-\alpha_2)=2,\\
		&m_{L(\lambda+\mu)}(\lambda+\mu-2\alpha_2)=1,
		&&m_{L(\lambda+\mu)}(\lambda+\mu-\alpha_1-2\alpha_2)=3,\\
		&m_{L(\lambda+\mu-\alpha_1-\alpha_2)}(\lambda+\mu-\alpha_1-2\alpha_2)=1.
	\end{align*}
	Therefore, $m_M(\lambda+\mu-\alpha_1)=m_M(\lambda+\mu-\alpha_2)=m_M(\lambda+\mu-2\alpha_2)=1,\; m_M(\lambda+\mu-\alpha_1-\alpha_2)=3$ and $m_M(\lambda+\mu-\alpha_1-2\alpha_2)=5$. We deduce that $L(\lambda+\mu-\alpha_1-2\alpha_2)$ is a composition factor of $M$. Observe that $\lambda+\mu-\alpha_1-2\alpha_2=(a,p-3)\in C_3$, thus $T(\lambda+\mu-\alpha_1-2\alpha_2)$ is not irreducible. We conclude by Lemma \ref{argument_tilting} that $M$ has multiplicity.
\end{proof}

\begin{prop}
	Let $\lambda=(a,0),\mu=(0,b)\in X^+$ be $p$-restricted with $\frac{p-1}{2}\leq a\leq p-3$ and $a+b\leq p-2$ (i.e. $\lambda\in C_2$ and $\lambda+\mu\in \widehat{C_2}$). Then $L(\lambda)\otimes L(\mu)$ is multiplicity-free.
\end{prop}

\begin{proof}
	In this proof, we use Euclidean coordinates. We set $M:=L(\lambda)\otimes L(\mu)$. 
	
	For $i\in \{0,\ldots, b\}$, we have $p-2-2a+i\leq i$.
	By Proposition \ref{product_characters} and Lemma \ref{Weyl_B2}, and using Lemma \ref{B2_ad} in the second equality below, we get
	\begin{align}
		\ch M=&\chi (b\omega_2)(\chi(a,0)-\chi(p-3-a,0))\nonumber\\
		=&\sum_{i=0}^b\sum_{j=0}^{\min\{i,b-i\}}\chi(a+\frac b2-i,\frac b2-j)-\sum_{i=0}^b\sum_{j=0}^{\min\{i,b-i\}}\chi(p-3-a+\frac b2-i,\frac b2-j)\nonumber\\
		=&\sum_{i=0}^b \;\sum_{j=\max\{0,p-2-2a+i\}}^{\min\{i,b-i\}}\chi(a+\frac b2-i,\frac b2-j)\nonumber\\
		&\qquad+\sum_{i=0}^b \;\sum_{j=0}^{\min\{b-i,p-3-2a+i\}}\chi(a+\frac b2-i,\frac b2-j)\nonumber\\
		&\qquad-\sum_{i=0}^b \;\sum_{j=\max\{0,b-i+p-2-2a\}}^{\min\{i,b-i\}}\chi(p-3-a+\frac b2-i,\frac b2-j)\nonumber\\
		&\qquad-\sum_{i=0}^b \;\sum_{j=0}^{\min\{i,b-i+p-3-2a\}}\chi(p-3-a+\frac b2-i,\frac b2-j)\nonumber\\
		=&\sum_{i=0}^b \;\sum_{j=\max\{0,p-2-2a+i\}}^{\min\{i,b-i\}}\chi(a+\frac b2-i,\frac b2-j)\label{eq334}\\
		&\qquad-\sum_{i=0}^b \;\sum_{j=\max\{0,b-i+p-2-2a\}}^{\min\{i,b-i\}}\chi(p-3-a+\frac b2-i,\frac b2-j)\label{eq335}\\
		&\qquad+\sum_{i=0}^b \;\sum_{j=0}^{\min\{b-i,p-3-2a+i\}}\chi(a+\frac b2-i,\frac b2-j)\label{eq336}\\
		&\qquad-\sum_{i=0}^b \;\sum_{j=0}^{\min\{i,b-i+p-3-2a\}}\chi(p-3-a+\frac b2-i,\frac b2-j).\label{eq337}
	\end{align}
	First, we work on the terms of lines \eqref{eq336} and \eqref{eq337}. We claim that they sum to zero. To that end, observe that in line \eqref{eq336}, the second sum is empty whenever $i<2a+3-p$ and in line \eqref{eq337},\linebreak the second sum is empty whenever $i>b+p-3-2a$. Thus we get
	\begin{align}
		&\sum_{i=0}^b \;\sum_{j=0}^{\min\{b-i,p-3-2a+i\}}\chi(a+\frac b2-i,\frac b2-j)-\sum_{i=0}^b \;\sum_{j=0}^{\min\{i,b-i+p-3-2a\}}\chi(p-3-a+\frac b2-i,\frac b2-j)\nonumber\\
		&=\sum_{i=2a+3-p}^b\sum_{j=0}^{\min\{b-i,p-3-2a+i\}}\chi(a+\frac b2-i,\frac b2-j)\nonumber\\
		&\hspace{4em}-\sum_{i=0}^{b+p-3-2a}\;\sum_{j=0}^{\min\{i,b-i+p-3-2a\}}\chi(p-3-a+\frac b2-i,\frac b2-j)\nonumber\displaybreak\\
		&=\sum_{i=2a+3-p}^b\sum_{j=0}^{\min\{b-i,p-3-2a+i\}}\chi(a+\frac b2-i,\frac b2-j)\nonumber\\
		&\hspace{4em}-\sum_{i=2a+3-p}^{b}\sum_{j=0}^{\min\{i-2a-3+p,b-i\}}\chi(a+\frac b2-i,\frac b2-j)\nonumber\\
		&=0.\nonumber
	\end{align}
	Therefore, only lines \eqref{eq334} and \eqref{eq335} remain and we get
	\begin{align}
		\ch M=&\sum_{i=0}^b \;\sum_{j=\max\{0,p-2-2a+i\}}^{\min\{i,b-i\}}\chi(a+\frac b2-i,\frac b2-j)\nonumber\\
		&\qquad-\sum_{i=0}^b \;\sum_{j=\max\{0,b-i+p-2-2a\}}^{\min\{i,b-i\}}\chi(p-3-a+\frac b2-i,\frac b2-j)\nonumber\\
		=&\sum_{i=0}^b \;\sum_{j=\max\{0,p-2-2a+i\}}^{\min\{i,b-i\}}\chi(a+\frac b2-i,\frac b2-j)\nonumber\\
		&\qquad-\sum_{i=0}^b \;\sum_{j=\max\{0,i+p-2-2a\}}^{\min\{i,b-i\}}\chi(p-3-a-\frac b2+i,\frac b2-j)\nonumber\\
		=&\sum_{i=0}^b \;\sum_{j=\max\{0,p-2-2a+i\}}^{\min\{i,b-i\}}\chi(a+\frac b2-i,\frac b2-j)-\chi(p-3-a-\frac b2+i,\frac b2-j).\label{eq338}
	\end{align}
	At this step, observe that if $i>\frac{b+2a+2-p}{2}$, then $p-2-2a+i>b-i$ and the second sum in line \eqref{eq338} is empty. Thus we get
	$$\ch M=\sum_{i=0}^{\min\{b,\lfloor\frac{b+2a+2-p}{2}\rfloor\}}\sum_{j=\max\{0,p-2-2a+i\}}^{\min\{i,b-i\}}\chi(a+\frac b2-i,\frac b2-j)-\chi(p-3-(a+\frac b2-i),\frac b2-j).$$
	For $0\leq i\leq \min\{b,\lfloor\frac{b+2a+2-p}{2}\rfloor\}$ and $\max\{0,p-2-2a+i\}\leq j\leq \min\{i,b-i\}$, we have
	\begin{ceqn}
	\begin{align*}
		&a+\frac b2-i\geq a+\frac b2-(\frac{b+2a+2-p}{2})=\frac{p-2}{2}>\frac{p-3}{2},\\
		&\frac b2 -j\geq 0\qquad \text{and}\\
		&(a+\frac b2-i)+(\frac b2 -j)=a+b-i-j\leq a+b \leq p-2.
	\end{align*}
\end{ceqn}
	Therefore, $(a+\frac b2-i,\frac b2-j)\in \widehat{C_2}$ for all $0\leq i\leq \min\{b,\lfloor\frac{b+2a+2-p}{2}\rfloor\}$ and \linebreak$\max\{0,p-2-2a+i\}\leq j\leq \min\{i,b-i\}$ and by Remark \ref{B2_character_symmetry}, we get 
	$$\ch M =\sum_{i=0}^{\min\{b,\lfloor\frac{b+2a+2-p}{2}\rfloor\}}\sum_{j=\max\{0,p-2-2a+i\}}^{\min\{i,b-i\}}\ch L(a+\frac b2-i,\frac b2-j).$$
	In particular, $M$ is multiplicity-free.
\end{proof}

\begin{prop}
	Let $\lambda=(a,0),\mu=(0,b)\in X^+$ be $p$-restricted with $\frac{p-1}{2}\leq a\leq p-3$ and $a+b> p-2$ (i.e. $\lambda\in C_2$ and $\lambda+\mu\in \widehat{C_3}\cup \widehat{C_4}$). Then $L(\lambda)\otimes L(\mu)$ has multiplicity.
\end{prop}

\begin{proof}
	 In this proof, we use Euclidean coordinates. We set $M:=L(\lambda)\otimes L(\mu)$. By \linebreak Proposition \ref{product_characters}, Lemma \ref{B2_ad} and Remark \ref{B2_character_symmetry}, we have
	\begin{align}
		\ch M&=\chi(\lambda)\chi(\mu)-\chi(s_{\alpha_1+\alpha_2,p}\sbullet \lambda)\chi(\mu)\nonumber\\
		&=\sum_{i=0}^b\sum_{j=0}^{\min\{i,b-i\}}\chi(a+\frac b2-i,\frac b2-j)-\sum_{i=0}^b\sum_{j=0}^{\min\{i,b-i\}}\chi(p-3-a+\frac b2-i,\frac b2-j).\label{eq70}
	\end{align}

	We set $$A:=\{(a+\frac b2-i,\frac b2-j)|\; 0\leq i\leq b,\; 0\leq j\leq \min\{i,b-i\}\}$$ 
	and 
	$$B:=\{(p-3-a+\frac b2-i,\frac b2-j)|\; 0\leq i\leq b,\; 0\leq j\leq \min\{i,b-i\}\}.$$
	
	By Lemma \ref{B2_ad}, we have $A,B\subseteq D\cup W\sbullet \overline{C_1}$.\\
	
	If $a+b$ is even, we set $t:=\frac{a+b-p+1}2$, $\nu_3:=(a+\frac b2-t,\frac{b}{2}-t)\in C_3$ and \linebreak$\nu_2:=s_{\alpha_1+2\alpha_2,p}\sbullet \nu_3= (a+\frac b2-(t+1),\frac{b}{2}-(t+1))\in C_2$.\\
	
	If $a+b$ is odd, we set $t:=\frac{a+b-p}2$, $\nu_3:=(a+\frac b2-(t+1),\frac{b}{2}-t)\in C_3$ and\linebreak $\nu_2:=s_{\alpha_1+2\alpha_2,p}\sbullet \nu_3= (a+\frac b2-(t+2),\frac{b}{2}-(t+1))\in C_2$.\\
	
	In both cases, we show that $\nu_2$ has multiplicity $2$ in $M$. Observe that $\nu_2,\nu_3\in A$ (this follows from $t+1\leq \frac b2$).
	Moreover, we claim that $\nu_2, \nu_3\notin B$. We check $\nu_3$ in case $a+b$ even, the other cases are similar. Suppose that $(a+\frac b2-t,\frac{b}{2}-t)=(p-3-a+\frac b2-k,\frac b2-r)$ for some $k,r$. Then
	 $r=t$ and $k=\frac{p+b-5-3a}2<r$, thus $\nu_3\notin B$.
	
	
	
	For every weight $\nu\in X$, we fix $w_\nu\in W$ such that $w_\nu \sbullet \nu \in D$ (recall that $D$ is a fundamental domain for the dot action of $W$ on $X$). Moreover, we take the convention that $\Delta(\nu)=0$ for every $\nu\in D\setminus X^+$. Using line \eqref{eq70} and Lemma \ref{action_characters}, we have
	$$[M:L(\nu_2)]=\sum_{\eta\in A}\det(w_\eta)[\Delta(w_\eta \sbullet \eta):L(\nu_2)]-\sum_{\eta\in B}\det(w_\eta)[\Delta(w_\eta \sbullet \eta):L(\nu_2)].$$
By the Strong Linkage Principle (Proposition \ref{strong_linkage}), $[\Delta(\eta):L(\nu_2)]=0$ unless $\nu_2\uparrow\eta$. By Lemma \ref{Weyl_B2}, we have $[\Delta(\eta):L(\nu_2)]=0$ for all $\eta\in C_4$. Thus, if $\eta$ is $p$-restricted and such that $[\Delta(\eta):L(\nu_2)]\neq0$, we have either $\eta=\nu_2$ or $\eta=\nu_3$. At this step, recall that every $\eta\in (A\cup B)\cap X^+$ is $p$-restricted.
By Lemma \ref{B2_ad}, if $\eta\in A\cup B$ and $w_\eta\neq \id$, then $w_\eta \sbullet \eta\in \overline{C_1}$ and $[\Delta(w_\eta \sbullet \eta):\nu_2]=0$. Therefore, we have
$$[M:L(\nu_2)]=\sum_{\eta\in A\cap \{\nu_2,\nu_3\}}[\Delta( \eta):L(\nu_2)]-\sum_{\eta\in B\cap \{\nu_2,\nu_3\}}[\Delta(\eta):L(\nu_2)].$$
By the previous observations, $B\cap \{\nu_2,\nu_3\}=\emptyset$ and $A\cap \{\nu_2,\nu_3\}=\{\nu_2,\nu_3\}$. Therefore, 
$$[M:L(\nu_2)]=[\Delta(\nu_2):L(\nu_2)]+[\Delta(\nu_3):L(\nu_2)]=2$$
and $M$ has multiplicity.
\end{proof}

\subsubsection{$L(a,b)\otimes L(0,d)$}

\begin{prop}
	Let $\lambda=(a,b),\mu=(0,d)\in X^+$ with $1\leq a<p$ and $2\leq b,d<p$. If $2a+b+2\not\equiv 0\bmod p$ and $a+b\neq p-1$, then $L(\lambda)\otimes L(\mu)$ has multiplicity.
\end{prop}

\begin{proof}
	We set $M:=L(\lambda)\otimes L(\mu)$. Using Argument \ref{argument1}, we show that either \linebreak$[M:L(\lambda+\mu-\alpha_1-\alpha_2)]\geq 2$ or $[M:L(\lambda+\mu-\alpha_1-2\alpha_2)]\geq 2$. 
	
	Using Lemmas \ref{-1}, \ref{B2_-11}, \ref{table_B2} and \ref{B2_-12}, we have
	\begin{align*}
		&m_M(\lambda+\mu)=1,&&m_M(\lambda+\mu-\alpha_1)=1,&&m_M(\lambda+\mu-\alpha_2)=2,\\
		&m_M(\lambda+\mu-\alpha_1-\alpha_2)=4,&&m_M(\lambda+\mu-2\alpha_2)=3,&&m_M(\lambda+\mu-\alpha_1-2\alpha_2)=9.
	\end{align*}
Therefore $$[M:L(\lambda+\mu)]=[M:L(\lambda+\mu-\alpha_2)]=[M:L(\lambda+\mu-2\alpha_2)]=1$$ and $$[M:L(\lambda+\mu-\alpha_1)]=0.$$

If $m_{L(\lambda+\mu)}(\lambda+\mu-\alpha_1-\alpha_2)=1$, then $[M:L(\lambda+\mu-\alpha_1-\alpha_2)]=2$ and $M$ has multiplicity.

If $m_{L(\lambda+\mu)}(\lambda+\mu-\alpha_1-\alpha_2)=2$, then $[M:L(\lambda+\mu-\alpha_1-\alpha_2)]=1$. In this case, we have $$m_{L(\lambda+\mu)}(\lambda+\mu-\alpha_1-2\alpha_2)\leq 3 \quad\text{and}\quad m_{L(\lambda+\mu-\alpha_2)}(\lambda+\mu-\alpha_1-2\alpha_2)\leq 2.$$ Therefore, $[M:L(\lambda+\mu-\alpha_1-2\alpha_2)]\geq 9-3-2-1-1=2$ and $M$ has multiplicity.
\end{proof}

\begin{prop}
	Let $\lambda=(a,b),\mu=(0,1)\in X^+$ with $1\leq a,b<p$. Then $L(\lambda)\otimes L(\mu)$ is multiplicity-free if and only if $\lambda\in C_1\cup C_2\cup C_3\cup C_4$.
\end{prop}

\begin{proof}
	In this proof, we take the convention that $\ch L(\nu)=0$ for all $\nu\notin X^+$. Moreover, for $\nu=(c,d)\in X$, we set 
	\begin{align*}
	&\delta_2(\nu):=\begin{cases}
		0 &\text{if } 2c+d+2=p,\\
		1&\text{else.}
	\end{cases}\qquad \delta_3(\nu):=\begin{cases}
	0 &\text{if } c+d=p-1,\\
	1&\text{else.} 
\end{cases}\\
&\delta_4(\nu):=\begin{cases}
0 &\text{if } 2c+d+2=2p,\\
1&\text{else.} 
\end{cases}\end{align*}
We set $M:=L(\lambda)\otimes L(\mu)$. By Lemma \ref{Weyl_B2}, $L(\mu)$ is a tilting module. 
	
	\begin{itemize}
		\item If $\lambda\in F_{1,2}\cup F_{2,3}\cup F_{3,4}$ then $L(\lambda)$ is a tilting module so $M$ is a tilting module. In this case $\lambda +\mu\in C_2\cup C_3\cup C_4\cup (F_{4,7}\setminus F_{4,6})$, hence $T(\lambda+\mu)$ is not irreducible by Lemma \ref{Weyl_B2}.\linebreak We conclude by Lemma \ref{argument_tilting} that $M$ has multiplicity.
		\item If $\lambda \in F_{3,5}\cup F_{4,7}$, then $b=p-1$ so $\lambda+\mu$ is not $p$-restricted. We conclude by Corollary \ref{sum_p_res}\linebreak that $M$ has multiplicity.
		\item If $\lambda\in F_{4,6}$ (i.e. $a=p-1$), observe that $L(\lambda+\mu-\alpha_2)$ is a composition factor of $M$ (by Argument \ref{argument1}). But $\lambda+\mu-\alpha_2$ is not $p$-restricted, hence we conclude by Corollary \ref{composition_p}\linebreak that $M$ has multiplicity.
		\item If $\lambda\in C_1$, then $\lambda+\mu\in \widehat{C_1}$ and we apply Corollary \ref{sum_C1} and Theorem \ref{B2_p0} to conclude that $M$ is multiplicity-free.
		\item If $\lambda\in C_2$, let $\lambda_1:=s_{\alpha_1+\alpha_2,p}\sbullet \lambda\in C_1$ and $\delta_2:=\delta_2(\lambda)$. By Proposition \ref{product_characters} and \linebreak Lemma \ref{Weyl_B2} we have
		\begin{align*}
			\ch M&= \chi(\mu)(\chi(\lambda)-\chi(\lambda_1))\\
			&=\chi(\lambda+\mu)+\chi(\lambda+\mu-\alpha_2)+\chi(\lambda+\mu-\alpha_1-\alpha_2)+\chi(\lambda+\mu-\alpha_1-2\alpha_2)\\
			&\hspace{3em}-\chi(\lambda_1+\mu)-\chi(\lambda_1+\mu-\alpha_2)-\chi(\lambda_1+\mu-\alpha_1-\alpha_2)\\
			&\hspace{3em}-\chi(\lambda_1+\mu-\alpha_1-2\alpha_2)\\
			&=(\chi(\lambda+\mu)-\chi(\lambda_1+\mu-\alpha_1-\alpha_2))+(\chi(\lambda+\mu-\alpha_2)-\chi(\lambda_1+\mu-\alpha_1-2\alpha_2))\\
			&\hspace{3em}+(\chi(\lambda+\mu-\alpha_1-\alpha_2)-\chi(\lambda_1+\mu))\\
			&\hspace{3em}+(\chi(\lambda+\mu-\alpha_1-2\alpha_2)-\chi(\lambda_1+\mu-\alpha_2))\\
			&=\ch L(\lambda+\mu)+\ch L(\lambda+\mu-\alpha_2)+\delta_2\cdot\ch L(\lambda+\mu-\alpha_1-\alpha_2)\\
			&\hspace{3em}+\delta_2\cdot\ch L(\lambda+\mu-\alpha_1-2\alpha_2).
		\end{align*}
		Thus $M$ is multiplicity-free.
		\item If $\lambda\in C_3$, let $\lambda_2:=s_{\alpha_1+2\alpha_2,p}\sbullet \lambda\in C_2$, $\delta_2:=\delta_2(\lambda_2)$ and $\delta_3:=\delta_3(\lambda)$. By Proposition \ref{product_characters}, Lemma \ref{Weyl_B2} and the previous case, we have
		\begin{align*}
			\ch M&= \chi(\mu)(\chi(\lambda)-\ch L(\lambda_2))\\
			&=\chi(\lambda+\mu)+\chi(\lambda+\mu-\alpha_2)+\chi(\lambda+\mu-\alpha_1-\alpha_2)+\chi(\lambda+\mu-\alpha_1-2\alpha_2)\\
			&\hspace{3em}-\ch L(\lambda_2+\mu)-\ch L(\lambda_2+\mu-\alpha_2)-\delta_2\cdot\ch L(\lambda_2+\mu-\alpha_1-\alpha_2)\\
			&\hspace{3em}-\delta_2\cdot \ch L(\lambda_2+\mu-\alpha_1-2\alpha_2)\\
			&=(\chi(\lambda+\mu)-\delta_2\cdot \ch L(\lambda_2+\mu-\alpha_1-2\alpha_2))\\
			&\hspace{3em}+(\chi(\lambda+\mu-\alpha_2)-\ch L(\lambda_2+\mu-\alpha_2))\\
			&\hspace{3em}+(\chi(\lambda+\mu-\alpha_1-\alpha_2)-\delta_2\cdot \ch L(\lambda_2+\mu-\alpha_1-\alpha_2))\\
			&\hspace{3em}+(\chi(\lambda+\mu-\alpha_1-2\alpha_2)-\ch L(\lambda_2+\mu))\\
			&=\ch L(\lambda+\mu)+\ch L(\lambda+\mu-\alpha_2)+\ch L(\lambda+\mu-\alpha_1-\alpha_2)\\
			&\hspace{3em}+\delta_3\cdot\ch L(\lambda+\mu-\alpha_1-2\alpha_2).
		\end{align*}
		Thus $M$ is multiplicity-free.
		\item If $\lambda\in C_4$, let $\lambda_3:=s_{\alpha_1+\alpha_2,2p}\sbullet \lambda\in C_3$, $\delta_3:=\delta_3(\lambda_3)$ and $\delta_4:=\delta_4(\lambda)$. By Proposition \ref{product_characters}, Lemma \ref{Weyl_B2} and the previous case, we have
		\begin{align*}
			\ch M&= \chi(\mu)(\chi(\lambda)-\ch L(\lambda_3))\\
			&=\chi(\lambda+\mu)+\chi(\lambda+\mu-\alpha_2)+\chi(\lambda+\mu-\alpha_1-\alpha_2)+\chi(\lambda+\mu-\alpha_1-2\alpha_2)\\
			&\hspace{3em}-\ch L(\lambda_3+\mu)-\ch L(\lambda_3+\mu-\alpha_2)-\ch L(\lambda_3+\mu-\alpha_1-\alpha_2)\\
			&\hspace{3em}-\delta_3\cdot \ch L(\lambda_3+\mu-\alpha_1-2\alpha_2)\displaybreak\\
			&=(\chi(\lambda+\mu)-\ch L(\lambda_3+\mu-\alpha_1-\alpha_2))\\
			&\hspace{3em}+(\chi(\lambda+\mu-\alpha_2)-\delta_3\cdot \ch L(\lambda_3+\mu-\alpha_1-2\alpha_2))\\
			&\hspace{3em}+(\chi(\lambda+\mu-\alpha_1-\alpha_2)-\ch L(\lambda_3+\mu))\\
			&\hspace{3em}+(\chi(\lambda+\mu-\alpha_1-2\alpha_2)-\ch L(\lambda_3+\mu-\alpha_2))\\
			&=\ch L(\lambda+\mu)+\ch L(\lambda+\mu-\alpha_2)+\delta_4\cdot \ch L(\lambda+\mu-\alpha_1-\alpha_2)\\
			&\hspace{3em}+\delta_4\cdot \ch L(\lambda+\mu-\alpha_1-2\alpha_2).
		\end{align*}
		Thus $M$ is multiplicity-free. \qedhere
	\end{itemize}

The classification of multiplicity-free tensor products of simple modules with $p$-restricted highest weight for an algebraic group of type $B_2$ is not completed in this thesis. In the previous sequence of propositions, we fully treated the following cases:
\begin{itemize}
	\item $\lambda=(a,b),\;\mu=(c,d)$ with $a\cdot b=0$ and $c\cdot d=0$,
	\item $\lambda=(a,b),\; \mu=(0,1)$.
\end{itemize}
It remains to consider the following cases:
\begin{itemize}
	\item $\lambda=(a,b),\;\mu=(0,d)$ with $a\neq 0,\; b\neq 0$ and $d\geq 2$,
	\item $\lambda=(a,b),\;\mu=(c,0)$ with $a\neq 0,\; b\neq 0$ and $c\neq 0$,
	\item $\lambda=(a,b),\;\mu=(c,d)$ with $a\cdot b\cdot c \cdot d\neq 0$.
\end{itemize}

\end{proof}

\cleardoublepage

\section{$\SL_n$ for $p=2$}\label{An_p2}

In this section, we classify multiplicity-free tensor products of simple $\SL_n$-modules with \linebreak$p$-restricted highest weight when $p=2$. To that end, we will use the classification of completely reducible tensor products of simple $\SL_n$-modules with $p$-restricted highest weight for $p=2$ (\cite[Theorem 7.12]{article_jonathan}).

\begin{thm}\label{A_n_completely_reducible}
	Let $G$ be of type $A_n$ and $p=2$. Let $\lambda,\mu\in X^+$ be nonzero $2$-restricted dominant weights. Up to the reordering of $\lambda$ and $\mu$, $L(\lambda)\otimes L(\mu)$ is completely reducible if and only if one of the following holds:
	\begin{enumerate}[label = \emph{(\arabic*)}] 
		\item $\lambda=\omega_1$ and $\mu=\omega_{i_1}+\ldots+\omega_{i_r}$ for even numbers $1<i_1<\ldots<i_r\leq n$,
		\item $\lambda=\omega_n$ and $\mu=\omega_{i_1}+\ldots+\omega_{i_r}$ for certain $i_1<\ldots<i_r< n$ such that $n+1-i_j$ is even for all $j\in \{1,\ldots,r\}$,
		\item $\lambda=\omega_2$ and $\mu=\omega_j$ for some $2<j\leq n$ with $j-2\equiv 3\bmod 4$,
		\item $\lambda=\omega_{n-1}$ and $\mu=\omega_j$ for some $1\leq j< n-1$ with $n-1-j\equiv 3\bmod 4$.
	\end{enumerate}
\end{thm}

\begin{thm}[{\cite[Theorem 1.1.A]{Stembridge}}]\label{A_n_char0}
	Let $G$ be of type $A_n$, $\lambda\in X^+$ and $i\in \{1,\ldots,n\}$. Then $L_{\C}(\omega_i)\otimes L_{\C}(\lambda)$ is multiplicity-free.
\end{thm}

\begin{thm}
	Let $G$ be of type $A_n$ and $p=2$. Let $\lambda,\mu\in X^+$ be nonzero $2$-restricted dominant weights. Up to the reordering of $\lambda$ and $\mu$ and up to duality, $L(\lambda)\otimes L(\mu)$ is multiplicity-free if and only if one of the following holds:
	\begin{enumerate}[label = \emph{(\arabic*)}] 
	\item $\lambda=\omega_1$ and $\mu=\omega_{i_1}+\ldots+\omega_{i_r}$ for even numbers $1<i_1<\ldots<i_r\leq n$,
	\item $\lambda=\omega_n$ and $\mu=\omega_{i_1}+\ldots+\omega_{i_r}$ for certain $i_1<\ldots<i_r< n$ such that $n+1-i_j$ is even for all $j\in \{1,\ldots,r\}$,
	\item $\lambda=\omega_2$ and $\mu=\omega_j$ for some $2<j\leq n$ with $j-2\equiv 3\bmod 4$,
	\item $\lambda=\omega_{n-1}$ and $\mu=\omega_j$ for some $1\leq j< n-1$ with $n-1-j\equiv 3\bmod 4$.
\end{enumerate}
\end{thm}

\begin{proof}
	Suppose that $\lambda$ and $\mu$ do not satisfy the conditions. Then $L(\lambda)\otimes L(\mu)$ is not completely reducible by Theorem \ref{A_n_completely_reducible}, and hence it has multiplicity by Lemma \ref{completely_reducible}.
	
	Now suppose that $\lambda,\;\mu$ verify the condition of the theorem. By Theorem \ref{A_n_completely_reducible}, \linebreak$L(\lambda)\otimes L(\mu)$ is completely reducible. Moreover, $L_{\C}(\lambda)\otimes L_{\C}(\mu)$ is multiplicity-free by \linebreak Theorem \ref{A_n_char0}. Therefore, $L(\lambda)\otimes L(\mu)$ is multiplicity-free by Theorem \ref{comp_red_0}.
\end{proof}

\cleardoublepage
\nocite{*}
\bibliographystyle{alpha}  
\bibliography{biblio}  

\end{document}

%% file: Titlepage.tex
\begin{titlepage}
\newcommand{\HRule}{\rule{\linewidth}{0.5mm}} 

\center 
 

\vspace{3cm}
\textsc{\LARGE École polytechnique fédérale de Lausanne}\\[1.5cm] 
\textsc{\Large Master thesis}\\[0.5cm] 
\textsc{\large Master in Mathematics}\\[0.5cm] 


\HRule \\[0.4cm] 
{ \Huge \bfseries Multiplicity-free tensor products of irreducible modules over simple algebraic groups in positive characteristic}\\[0.4cm] 
\HRule \\[1.5cm]
 

\begin{minipage}{0.4\textwidth}
\begin{flushleft} \large
\emph{Author:}\\
Gaëtan \textsc{Mancini} 
\end{flushleft}
\end{minipage}
~
\begin{minipage}{0.4\textwidth}
\begin{flushright} \large
\emph{Supervisor:} \\
Donna \textsc{Testerman} 
\end{flushright}
\end{minipage}\\[10cm]


 

\vfill 

\end{titlepage}